\documentclass[12pt]{amsart}
\usepackage{latexsym}
\usepackage{amssymb}
\usepackage{amscd}
\usepackage{mathrsfs}
\renewcommand\a{\alpha}
\renewcommand\b{\beta}
\newcommand\g{\gamma}
\renewcommand\d{\delta}
\newcommand\la{\lambda}
\newcommand\z{\zeta}
\newcommand\e{\eta}
\renewcommand\th{\theta}

\newcommand\io{\iota}
\newcommand\m{\mu}
\newcommand\n{\nu}
\newcommand\s{\sigma}

\newcommand\f{\phi}
\newcommand\vf{\varphi}

\renewcommand\t{\tau}
\renewcommand\r{\rho}

\newcommand\w{\omega}

\newcommand\vS{\varSigma}

\newcommand\vD{\varDelta}

\newcommand\vG{\varGamma}
\newcommand\ve{\varepsilon}

\newcommand{\OO}{\mathbb O}

\newcommand\SA{\mathscr{A}}
\newcommand\SB{\mathscr{B}}
\newcommand\SC{\mathscr{C}}
\newcommand\SD{\mathscr{D}}
\newcommand\SE{\mathscr{E}}
\newcommand\SF{\mathscr{F}}
\newcommand\SG{\mathscr{G}}
\newcommand\SL{\mathscr{L}}

\newcommand\SN{\mathscr{N}} 
\newcommand\SO{\mathscr{O}}
\newcommand\SP{\mathscr{P}}

\newcommand\ST{\mathscr{T}}
\newcommand\SH{\mathscr{H}}
\newcommand\ScS{\mathscr{S}}
\newcommand\SW{\mathscr{W}}

\newcommand\Ql{\bar{\mathbf Q}_l}

\newcommand\BP{\mathbf P}

\newcommand\BC{\mathbf C}

\newcommand\BZ{\mathbf Z}

\newcommand\BH{\mathbf H}

\newcommand\Bk{\mathbf k}

\newcommand\Bxi{\boldsymbol{\xi}}

\newcommand\Fb{\mathfrak b}

\newcommand\Fg{\mathfrak g}

\newcommand\Fqq{\mathfrak q}

\newcommand\Fl{\mathfrak l}
\newcommand\Ft{\mathfrak t}
\newcommand\Fm{\mathfrak m}
\newcommand\Fn{\mathfrak n}

\newcommand\iv{^{-1}}
\newcommand\wh{\widehat}
\newcommand\wt{\widetilde}
\newcommand\wg{^{\wedge}}

\newcommand\ol{\overline}

\newcommand\hra{\hookrightarrow}
\newcommand\lra{\leftrightarrow}

\newcommand\IC{\operatorname{IC}}
\newcommand\Ker{\operatorname{Ker}}
\newcommand\Hom{\operatorname{Hom}}
\newcommand\End{\operatorname{End}}

\newcommand\ind{\operatorname{ind}}
\newcommand\Ind{\operatorname{Ind}}

\newcommand\supp{\operatorname{supp}\,}
\newcommand\Lie{\operatorname{Lie}}

\newcommand\ch{\operatorname{ch}}

\newcommand\ad{\operatorname{ad}}

\newcommand\reg{_{\operatorname{reg}}}

\newcommand\uni{_{\operatorname{uni}}}
\newcommand\nil{_{\operatorname{nil}}}
\newcommand\id{\operatorname{id}}

\newcommand\lp{\operatorname{\!\langle\!}}
\newcommand\rp{\operatorname{\!\rangle\!}}
\renewcommand\Im{\operatorname{Im}}

\newcommand\dw{\dot w}

\newcommand{\isom}{\,\raise2pt\hbox{$\underrightarrow{\sim}$}\,}
\numberwithin{equation}{section}

\newtheorem{thm}{Theorem}[section]
\newtheorem{lem}[thm]{Lemma}
\newtheorem{cor}[thm]{Corollary}
\newtheorem{prop}[thm]{Proposition}

\def \para#1{\par\medskip\textbf{#1}
              \addtocounter{thm}{1}}

\def \remark#1{\par\medskip\noindent
                \textbf{Remark #1}
                \addtocounter{thm}{1}}

\def \definition#1{\par\medskip\noindent
                \textbf{Definition #1}
                \addtocounter{thm}{1}}

\begin{document}
\setlength{\baselineskip}{4.9mm}
\setlength{\abovedisplayskip}{4.5mm}
\setlength{\belowdisplayskip}{4.5mm}
\renewcommand{\theenumi}{\roman{enumi}}
\renewcommand{\labelenumi}{\rm{(\theenumi)}}
\renewcommand{\thefootnote}{\fnsymbol{footnote}}
\renewcommand{\thefootnote}{\fnsymbol{footnote}}
\allowdisplaybreaks[2]
\parindent=20pt
\medskip

\medskip
\begin{center}
 {\bf Generalized Springer correspondence for symmetric spaces \\
      associated to orthogonal groups }
\\
\vspace{1cm}
Toshiaki Shoji and Gao Yang
\\
\title{}
\end{center}

\begin{abstract}
Let $G = GL_N(\Bk)$, where $\Bk$ is an algebraically closed field of $\ch \Bk \ne 2$, 
and $\th$ an involutive 
automorphism of $G$ such that $H = (G^{\th})^0$ is isomorphic to $SO_N(\Bk)$. 
Then $G^{\io\th} = \{ g \in G \mid \th(g) = g\iv\}$ is regarded as 
a symmetric space $G/G^{\th}$. Let $G^{\io\th}\uni$ be the set of unipotent 
elements in $G^{\io\th}$. $H$ acts on $G^{\io\th}\uni$ by the conjugation. 
As an analogue of the generalized Springer correspondence in the case of 
reductive groups, we establish in this paper  the generalized Springer correspondence
between $H$-orbits in $G^{\io\th}\uni$ and 
irreducible representations of various symmetric groups.  

\end{abstract}

\maketitle

\markboth{SHOJI - YANG}{GENERALIZED SPRINGER CORRESPONDENCE}
\pagestyle{myheadings}

\begin{center}
{\sc Introduction}
\end{center}
\par\medskip\noindent
In [L1], Lusztig established the theory of generalized Springer correspondence
for reductive groups.  Let $G$ be a connected reductive group, and $G\uni$ the set
of unipotent elements in $G$.  $G\uni$ has finitely many conjugacy classes, and
let $\SN_G$ be the set of pairs $(C, \SE)$, where $C$ is a unipotent class, and 
$\SE$ is a $G$-equivariant simple ($\Ql$-) local system on $C$. Lusztig gave a notion 
of cuspidal pairs for $\SN_G$, and showed that $\SN_G$ has a natural partition 
\begin{equation*}
\SN_G = \coprod_{c \in \ScS_G}\SN_G^{(c)}.
\end{equation*} 
Here $\ScS_G$ is the set of isomorphism classes of the triples 
$(L, C_L, \SE_L)$, where $L$ is a Levi subgroup of some parabolic subgroup of 
$G$, and $(C_L, \SE_L)$ is a cuspidal pair on $L$. 
He showed that $\SW_L = N_G(L)/L$ is a Coxeter group, and that there exists 
a natural bijection between $\SN_G^{(c)}$ and the set $\SW_L\wg$ 
of irreducible representations of $\SW_L$ over $\Ql$, namely, 

\begin{equation*}
\tag{*}
\SN_G \simeq \coprod_{(L, C_L, \SE_L)} \SW_L\wg,
\end{equation*}
where $(L,C_L,\SE_L)$ runs over the classes in  $\ScS_G$. 
(*) is a generalization of the Springer correspondence, which is a natural 
injective map  
$W\wg \to \SN_G$, where $W$ is the Weyl group of $G$. 
The correspondence $\SN_G^{(c)} \simeq \SW_L\wg$ is obtained by considering  
a certain semisimple perverse sheaf on $G$ arising from a finite Galois covering
related to $\SW_L$, and by restricting each simple component to $G\uni$. 
\par
In this paper, we consider a similar problem for symmetric spaces associated 
to orthogonal groups. 
Let $G = GL_n(\Bk)$, where $\Bk$ is an algebraically 
closed field of odd characteristic.  We consider 
an involutive automorphism $\th : G \to G$ such that its 
subgroup $G^{\th}$ of $\th$-fixed points is isomorphic to $O_N(\Bk)$.
Put $H = (G^{\th})^0$.  Let $\io : G \to G$ be the anti-automorphism 
$g \mapsto g\iv$, and consider a closed subset 
$G^{\io\th} = \{ g \in G \mid \th(g) = g\iv \}$ of $G$. As a variety, 
$G^{\io\th}$ is isomorphic to the symmetric space $G/G^{\th}$.
Let $G^{\io\th}\uni$ be the set of unipotent elements in $G^{\io\th}$.  
$G^{\io\th}$ and $G^{\io\th}\uni$ are stable under the conjugation action of 
$H$ on $G$.  It is known that 
$G^{\io\th}\uni$ has finitely many $H$-orbits ([R]).      
\par
We consider the set $\SN_G$ of pairs $(\SO, \SE)$, where $\SO$ is an $H$-orbit
in $G^{\io\th}\uni$, and $\SE$ is an $H$-equivariant simple local system on $\SO$.   
We can formulate the notion of cuspidal pairs for $\SN_G$, and so obtain 
a similar set $\ScS_G$ as before.  However some difficulty occurs for constructing 
 a semisimple 
perverse sheaf on $G^{\io\th}$ related to a finite Galois covering.  Our discussion 
is based on the choice of a $\th$-stable Borel subgroup $B$ and a $\th$-stable maximal
torus $T$ contained in $B$.  But in our case, 
$T^{\io\th}$ is not maximal $\th$-anisotropic, and the role of 
semisimple elements becomes restricted in contrast to the case of reductive groups (see 1.4).  
In order to overcome this difficulty, we 
introduce a certain subgroup $\SD$ of $U$ such that $\SD$ is a subset of $U^{\io\th}$, 
where $U$ is the unipotent radical of $B$.    
By making use of $\SD$, one can construct a finite Galois covering, and a semisimple 
perverse sheaf on $G^{\io\th}$ associated to it.  Once this is done, 
basically a similar strategy as in [L1] can be applied for the remaining
part, and we can establish 
the generalized Springer correspondence.  By making use of the restriction 
theorem analogous to [L1], we give an explicit description of the 
generalized Springer correspondence, based on the discussion used in 
Lusztig-Spaltenstein [LS] in the case of classical groups.
\par
Let $\Fg$ be the Lie algebra of $G$, and $\th : \Fg \to \Fg$ an induced automorphism .  
The set $\Fg^-\nil = \{ x \in \Fg\nil \mid \th(x) = -x\}$ 
of nilpotent elements in $\Fg^-$ is isomorphic to $G^{\io\th}\uni$, 
compatible with $H$-action. One can consider the set $\SN_{\Fg}$ of pairs $(\SO, \SE)$, 
where $\SO$ is an $H$-orbit in $\Fg^-\nil$ and $\SE$ is an $H$-equivariant simple local 
system on $\SO$.  In the case where $\Bk = \BC$, 
Chen-Vilonen-Xue [CVX1] considered a similar problem for $\SN_{\Fg}$, and determined 
the Springer correspondence, which is a part of $\SN_{\Fg}$, 
 by making use of the Fourier-Deligne  transform of 
perverse sheaves on $\Fg^-\nil$, instead of considering the restriction of 
perverse sheaves on $G$ to $G\uni$.
In [CVX2], they treated the whole of $\SN_{\Fg}$, but 
the generalized Springer correspondence is not formulated there, and 
the explicit correspondence is not determined. 
In [LY], Lusztig-Yun
studied perverse sheaves on $\Fg^-\nil$ associated to arbitrary symmetric spaces, 
and more generally, associated to $\BZ/m\BZ$-graded Lie algebras, and established 
the results closely related to the generalized Springer correspondence. In their case, 
they also use the Fourier-Deligne transforms, instead of the restriction.

\par\bigskip\bigskip\noindent
{\bf Contents}

\par\medskip\noindent
Introduction   \\
1. Preliminaries on symmetric spaces  \\
2. Cuspidal local systems  \\
3. Admissible complexes   \\
4. Sheaves on the variety of semisimple orbits \\
5. Generalized Springer correspondence \\
6. Restriction theorem  \\
7. Unipotent orbits   \\
8. Structure of the algebra $\SA_{\SE_1}$ \\
9.  Determination of the generalized Springer correspondence  \\
10.  Induction   \\
Appendix 
\par\bigskip
 
\section{Preliminaries on symmetric spaces}
\para{1.1.}
Let $G$ be a connected reductive group over an algebraically closed field 
$\Bk$ of $\ch \Bk \ne 2$, and $\th : G \to G$ be an involutive automorphism 
on $G$.   
Then $G^{\th} = \{ g \in G \mid \th(g) = g\}$ is a reductive subgroup of $G$.
We put $H = (G^{\th})^0$. Let $\io : G \to G, g \mapsto g\iv$ be 
the anti-automorphism on $G$. We consider the set 
$G^{\io\th} = \{ g \in G \mid \th(g) = g\iv\}$ of $\io\th$-fixed points in $G$. 
Then $G$ acts on $G^{\io\th}$ by $g : x \mapsto gx\th(g)\iv$. This action is called 
the twisted action of $G$. $G$ acts transitively on each 
connected component of $G^{\io\th}$ with respect to the twisted $G$-action ([R]).
Thus there are only finitely many $G$-orbits in $G^{\io\th}$, and each $G$-orbit 
is closed.  In particular, the set $\{ g\th(g)\iv \mid g \in G\}$ is a connected 
component of $G^{\io\th}$. 
\par 
It is known by [St, \S 7] that there exist a $\th$-stable Borel subgroup $B$ and 
a $\th$-stable maximal torus $T$ of $G$ such that $B \supset T$. 
We show a lemma.

\begin{lem}  
Let $(B, T), (B', T')$ be $\th$-stable pairs of Borel subgroup and maximal 
torus of $G$.  Assume that $T^{\io\th}$ is connected.  
Then there exists $g \in G^{\th}$ such that 
$gBg\iv = B', gTg\iv = T'$.   
\end{lem} 

\begin{proof}
We choose $x \in G$ such that $xBx\iv = B', xTx\iv = T'$.
Then $\th(x)B\th(x)\iv = xBx\iv, \th(x)T\th(x)\iv = xTx\iv$. 
Hence $x\iv \th(x) \in B \cap N_G(T)  = T$. 
In particular, $x\iv \th(x) = t \in T^{\io\th}$. 
Since $T^{\io\th}$ is a torus, 
one can find $t_1 \in T^{\io\th}$ such 
that $t = (t_1)^2$, i.e., $t = t_1 \th(t_1)\iv$.
If we put $g = xt_1$, we have $g \in G^{\th}$, and 
$B' = gBg\iv, T' = gTg\iv$.  The lemma is proved.    
\end{proof}

\remark{1.3.}
In general, $\th$-stable pairs $(B, T)$, $(B', T')$ are not necessarily $H$-conjugate, 
(see 1.9).

\para{1.4.}
We fix a $\th$-stable Borel subgroup $B$ and a $\th$-stable maximal torus 
$T$ such that $T \subset B$. 
A torus $S$ is called $\th$-anisotropic if $\th(t) = t\iv$ for any $t \in S$. 
A maximal $\th$-anisotropic tours is a $\th$-anisotropic torus which is maximal with respect to  
the inclusion relation.  It is known by [V] that a $\th$-anisotropic torus exists 
if $\th \ne \id$, and every maximal $\th$-anisotropic tori are conjugate under $H$.   
Moreover, any semisimple element in $G^{\io\th}$ is contained in some maximal $\th$-anisotropic 
torus. 
Here $(T^{\io\th})^0$ is a $\th$-anisotropic torus, but in general, it is not 
maximal $\th$-anisotropic as the following example shows.  
\par
Let $G = GL_2$ and define $\th : G \to G$ by $\th(g) = J\iv({}^tg\iv)J$ 
with $J = \begin{pmatrix}
             0 & 1 \\
             1 & 0
          \end{pmatrix}$. 
We define $x \in G$ by 
\begin{equation*}
x = \begin{pmatrix}
       1  &   -\z   \\
      1/2  &  \z/2 
    \end{pmatrix},
\end{equation*}
where $\z  \in \Bk^*$ such that $\z^2 = -1$. 
One can check that $x\iv \th(x) = J$. Let $T$ be the group of 
diagonal matrices, and $B$ the group of upper triangular matrices in $G$. 
Then $(B, T)$ is a $\th$-stable pair of Borel subgroup and maximal torus of $G$.
In this case, 
\begin{equation*}
T^{\io\th} = \biggl\{ \begin{pmatrix}
                    a   &  0  \\
                    0   &  a
                \end{pmatrix} \mid a \in \Bk^* \biggr\}  = Z(G). 
\end{equation*}
Put $T_1 = xTx\iv$.  Since $x\iv\th(x) = J \in N_G(T)$, $T_1$ is also a $\th$-stable
maximal torus. 
One can check that $T_1^{\io\th} = T_1$, and so $T_1$ is a maximal 
$\th$-anisotropic torus. 
Since $\dim T^{\io\th} = 1, \dim T_1 = 2$, $T^{\io\th}$ is not maximal.  

\para{1.5.}
Put $B_H = B \cap H$, and $T_H = T \cap H$.  Then 
as remarked in [R, Lemma 5.1], $B_H$ is a Borel subgroup of $H$ and
$T_H$ is a maximal torus of $H$ contained in $B_H$.  
Let $\Fg = \Lie G$ be the Lie algebra of $G$, and put 
$\Fb = \Lie B, \Ft = \Lie T$. The differential $d\th$ of $\th$ induces 
an involutive automorphism of $\Fg$, which we also denote by $\th$.   
Thus, $\Fb, \Ft$ are $\th$-stable. We have a decomposition 
$\Fg = \Fg^+ \oplus \Fg^-$, where 
$\Fg^{\pm} = \{ x \in \Fg \mid \th(x) = \pm x \}$. 
Here $\Fg^+ = \Lie H$, and $\Fg^-$ coincides with the tangent space of 
$G/G^{\th} \simeq G^{\io\th}$.  
Let $\vD \subset X(T)$ be the root system of $G$ with respect to $T$ 
(here $X(T)$ is the character group of $T$), and $\vD_+$ the set 
of positive roots in $\vD$ with respect to $(T, B)$. We have the root space 
decomposition 
$\Fg = \Ft \oplus \bigoplus_{\a \in \vD}\Fg_{\a}$. 
The map $\a \mapsto \a\circ\th$ on $X(T)$ induces a bijection 
on $\vD$, which we denote by $\s$. $\s : \vD \to \vD$ is compatible 
with the root space decomposition, and $\th$ gives an isomorphism  
$\th : \Fg_{\a} \to  \Fg_{\s(\a)}$.  $\s$ preserves $\vD_+$, and so produces 
an automorphism of the Dynkin diagram of $G$. Put
\begin{equation*}
\tag{1.5.1}
\vD_0 = \{ \a \in \vD_+ \mid \s(\a) = \a \}. 
\end{equation*} 
$\s$ acts freely on $\vD_1 = \vD_+ - \vD_0$, and we denote by $\ol\vD_1$ 
the set of $\s$-orbits in $\vD_1$.  For each $\b = \{ \a, \s(\a) \} \in \ol\vD_1$, 
put $\Fg_{\b} = \Fg_{a} \oplus \Fg_{\s(\a)}$. Then $\Fg_{\b}$ is $\th$-stable, and  
$\Fg_{\b}$ is decomposed as $\Fg_{\b} = \Fg_{\b}^+ \oplus \Fg_{\b}^-$ with 
$\Fg_{\b}^+ \subset \Fg^+, \Fg_{\b}^- \subset \Fg^-$. 
If $\a \in \vD_0$, then $\th$ gives a linear isomorphism $\Fg_{\a} \to \Fg_{\a}$. Since 
$\th^2 = 1$, we have $\th = \pm 1$ on $\Fg_{\a}$, i.e., $\Fg_{\a} \subset \Fg^+$ or 
$\Fg_{\a} \subset \Fg^-$.    
We denote by $\vD_0^+$ (resp. $\vD_0^-$) 
the set of $\a \in \vD_0$ such that $\Fg_{\a} \subset \Fg^+$ 
(resp. $\Fg_{\a} \subset \Fg^- $).  
Let $\Fn$ be the nilpotent radical of $\Fb$.  $\Fn$ is $\th$-stable, and is decomposed as 
$\Fn = \Fn^+ \oplus \Fn^-$, where $\Fn^{\pm} = \Fn \cap \Fg^{\pm}$. 
We have a root space decomposition 

\begin{equation*}
\tag{1.5.2}
\Fn^{\pm} = \bigoplus_{\b \in \ol\vD_1}\Fg_{\b}^{\pm} \oplus
                     \bigoplus_{\a \in \vD_0^{\pm}}\Fg_{\a}.
\end{equation*} 

\para{1.6.}
Let $U$ be the unipotent radical of $B$, which is $\th$-stable, 
and put $U_H = U^{\th} = U \cap H$. 
For each $\a \in \vD_+$, we have a one parameter 
subgroup $U_{\a} \subset U$ such that $\Lie U_{\a} = \Fg_{\a}$. 
We have an isomorphism 
\begin{equation*}
U \simeq \prod_{\a \in \vD_+} U_{\a}
\end{equation*}
for a choice of the total order on $\vD_+$. 
Here $\th :U \to U$ induces an isomorphism $\th : U_{\a} \to U_{\s(a)}$ for each $\a$, 
and according to the decomposition in (1.5.2) in the Lie algebra case, we have
\begin{equation*}
\tag{1.6.1}
U_H \simeq \prod_{\{\a, \s(\a)\} \in \ol\vD_1}(U_{\a}U_{\s(\a)})^{\th}
              \times \prod_{\a \in \vD^+_0}U_{\a}. 
\end{equation*} 
Note that $U(\vD_0) = \prod_{\a \in \vD_0}U_{\a}$ 
is a $\th$-stable closed subgroup of $U$, and its $\th$-fixed point subgroup 
coincides with $\prod_{\a \in \vD^+_0}U_{\a}$. 
\par
Concerning $U^{\io\th}$, we obtain a similar description as in (1.6.1).  But since 
$\io\th$ is not a group homomorphism, we need a special care.
One can find a filtration $U \supset U_1 \supset U_2 \supset \cdots$ 
 of $U$ by $\th$-stable normal subgroups, by making use of the commutator relations 
for $U$, such that each quotient group $U_i/U_{i+1}$ is of the form 
$U_{\a}U_{\s(\a)}$ for $\{\a, \s(\a)\} \in \ol \vD_1$ or of the form $U_{\a}$ 
for $\a \in \vD^-_0$. 
In that case, $U^{\io\th}$ has a filtration 
$U^{\io\th} \supset U_1^{\io\th} \supset U_2^{\io\th} \supset \cdots$ 
(by affine subspaces) such that   
$U^{\io\th}_i/U^{\io\th}_{i+1} \simeq (U_{\a}U_{\s(\a)})^{\io\th}$ or 
$\simeq U_{\a}^{\io\th}$ accordingly.  In particular, $U_H$ (resp. $U^{\io\th}$) 
is isomorphic to an affine space with $\dim U_H = \dim \Fn^+$ 
(resp. $\dim U^{\io\th} = \dim \Fn^-$). 

\para{1.7.}
In the remaining part of the paper, we concentrate on the symmetric spaces 
associated to orthogonal groups. 
Let $V$ be an $N$ dimensional vector space over $\Bk$, 
and let $G =  GL_N = GL(V)$ with $N \ge 2$.  
Consider an involutive automorphism  $\th : G \to G$ defined by
$\th(g) = J\iv({}^tg\iv)J$, where 
\begin{align*}
J &= \begin{pmatrix}
          1 &    0   &  0  \\ 
          0  &   0   &  1_n  \\
          0  &  1_n  &  0
         \end{pmatrix}   \quad \text{ if } N = 2n+1,  \\  \\
J &= \begin{pmatrix}
            0  &  1_n \\
            1_n  &  0     
      \end{pmatrix}  \qquad\quad \text{ if } N = 2n,  
 \end{align*}
with $1_n$ the identity matrix of degree $n$.
Let $H = (G^{\th})^0$.  Then 
$H$ is a special orthogonal group $SO_N$ with respect to the 
symmetric bilinear form $\lp u, v\rp = {}^tuJv$ on $V$ 
($u,v \in V$), under the identification $V \simeq \Bk^N$ via 
the basis $\{ e_0, e_1, \dots, e_n, f_1, \dots. f_n\}$ 
in the case where $N = 2n+1$, and the basis 
$\{ e_1, \dots, e_n, f_1, \dots, f_n\}$ in the case where $N = 2n$. 
It follows that 
\begin{align*}
\lp e_i, f_j \rp &= \lp f_j, e_i \rp = \d_{i,j} \text{ for $1 \le i, j \le n$,} \\
\lp e_i, e_j \rp &= \lp f_i, f_j\rp = 0  \text{ for $1 \le i,j \le n$, } 
\end{align*}
and $\lp e_0, e_0 \rp = 1$, $\lp e_0, e_i\rp = \lp e_0, f_i\rp = 0$ for $i \ge 1$.
\par
Let $\Fg = \Lie G$, and $\th : \Fg \to \Fg$ be the induced automorphism.
Then $\th(x) = -J\iv ({}^tx)J$ for $x \in \Fg$. 
Let $x^*$ be the adjoint of $x \in \Fg$ with respect to the bilinear form 
$\lp\ ,\ \rp$.  Then we have $x^* = J\iv ({}^tx)J$, and so 
\begin{equation*}
\Fg^{\pm} = \{ x \in \Fg \mid x^* = \mp x\}.
\end{equation*}
We have

\begin{align*}
\Fg^+ &= \{ x \in \Fg \mid \lp xv, w \rp = - \lp v, xw\rp \}, \\
\Fg^- &= \{ x \in \Fg \mid \lp xv, w \rp = \lp v,xw\rp \}.
\end{align*}
In particular, $\Fg^-$ coincides with the set of self-adjoint matrices in 
$\Fg = \Fg\Fl_N$. 
Correspondingly, $G^{\io\th}$ coincides with the set of non-degenerate 
self-adjoint matrices in $\Fg\Fl_N$.  In particular, $G^{\io\th}$ is connected, 
and by 1.1 we have 
\begin{equation*}
\tag{1.7.1}
G^{\io\th} = \{ g\th(g)\iv \mid g \in G\}.
\end{equation*} 

\para{1.8.}
We fix a $\th$-stable Borel subgroup $B$ and a $\th$-stable maximal torus $T \subset B$ 
as follows. 
First assume that $N = 2n+1$.
Let $B$ be the subgroup of $G$ consisting of the matrices of the form 
\begin{equation*}
\begin{pmatrix}
a        &  0    &  d_1  \\
{}^td_2  &  b_1  &   c  \\
0        &   0   &  b_2
\end{pmatrix},
\end{equation*}
where $a \in \Bk^*$, $d_1, d_2$ are (row) vectors in $\Bk^n$, 
and $b_1, b_2, c$ are square matrices of degree $n$ with $b_1$
upper triangular, $b_2$ lower triangular.  
Let $T$ be the set of all diagonal matrices in $G$.  Then $B$ is a Borel 
subgroup of $G$, and $T$ is a maximal torus in $G$ with $T \subset B$.
$B, T$ are both $\th$-stable.  
We have
\begin{equation*}
T^{\th} = \biggl\{\begin{pmatrix}
              \pm 1  &  0   &  0  \\
                0    &  b    &  0  \\
                0    &  0    &  b\iv
            \end{pmatrix}  \bigg| \ b \in D_n \biggr\}, 
\qquad 
T^{\io\th} = \biggl\{ \begin{pmatrix}
                       a   &   0   &   0   \\
                       0   &   b   &   0   \\
                       0   &   0   &   b
                      \end{pmatrix}  \bigg| \ a \in \Bk^*, b \in D_n \biggr\},
\end{equation*}
where $D_n$ is the group of diagonal matrices of $GL_n$. 
Moreover, we have

\begin{align*}
B^{\th} &= \biggl\{\begin{pmatrix}
                           a      &  0  &   d_1   \\
                           {}^td_2    &  b  &   c   \\
                           0      &  0  &   {}^tb\iv
                  \end{pmatrix} \in B  \ \bigg| \  a = \pm 1, 
           d_2 = -a\iv d_1{}^tb,  \ {}^tc = - ({}^td_1d_1 + b\iv c){}^tb  \biggr\},  \\
B^{\io\th} &= \biggl\{    \begin{pmatrix}
                           a       &   0   &   d  \\
                           {}^td   &   b   &   c   \\
                           0       &   0   &   {}^tb
                             \end{pmatrix}  \in B \ \bigg| \ {}^tc = c  \biggr\}.  
\end{align*}
\par
Next assume that $N = 2n$. Let $B$ be the subgroup of $G$ consisting of the 
matrices of the form
\begin{equation*}
\begin{pmatrix}
b_1  &  c  \\
0    &  b_2
\end{pmatrix},
\end{equation*} 
where $b_1, b_2, c$ are square matrices of degree $n$ with $b_1$ upper triangular, 
$b_2$ lower triangular.
Let $T$ be the set of diagonal matrices in $G$.  Then $B$ is a Borel subgroup of $G$ 
and $T$ is a maximal torus in $G$ with $T \subset B$. $B, T$ are both $\th$-stable.
We have

\begin{equation*}
T^{\th} = \biggl\{ \begin{pmatrix}
             b  &  0  \\
             0  &  b\iv
          \end{pmatrix}  \big| \ b \in D_n\biggr\},
\quad
T^{\io\th} = \biggl\{ \begin{pmatrix}
                 b  &  0  \\
                 0  &  b   
                      \end{pmatrix}  \big| \ b \in D_n \biggr\}.
\end{equation*}   
Moreover, we have 

\begin{align*}
B^{\th} &= \biggl\{ \begin{pmatrix}
                       b   &   c   \\
                       0   &   {}^tb\iv   
                   \end{pmatrix} \in B \ \big| \ {}^tc = - b\iv c \, {}^tb \biggr\}, \\
 B^{\io\th} &=  \biggl\{  \begin{pmatrix}
                             b  &   c   \\
                             0  &   {}^tb
                           \end{pmatrix} \in B \ \big|\ {}^tc = c \biggr\}.  
\end{align*}
\par
Let $U$ be the unipotent radical of $B$. Then $U$ is $\th$-stable.
Put $B_H = (B^{\th})^0, T_H = (T^{\th})^0$ and $U_H = U^{\th}$. 
Thus $B_H \supset T_H$ is a pair of a Borel subgroup and a maximal torus 
in $H$, and $U_H$ is the unipotent radical of $B_H$.

\para{1.9.} 
Let $T \subset B$ be as above. 
For $i = 1, \dots, n-1$, let $s_i$ be the permutation of the basis in $V$ 
such that $e_i \lra e_{i+1}, f_i \lra f_{i+1}$ and that it fixes all other basis.
Also for $i = 1, \dots, n$, 
let $t_i$ be the permutation $e_i \lra f_i$ which fixes all other basis.
Assume that $N$ is odd.  Then $N_H(T)$ is generated by $s_1, \dots, s_{n-1}, t_n$ and $T$.
$N_H(T)/T$ is isomorphic to the Weyl group $W_n$ of type $B_n$.  On the other hand, 
$N_{G^{\th}}(T)$ is generated by $N_H(T)$ and $-1$, so that 
$N_{G^{\th}}(T)/T \simeq W_n \times \BZ/2\BZ$. 
Hence in view of Lemma 1.2 (note that $T^{\io\th}$ is connected), 
any $\th$-stable pair $(B', T')$ is $H$-conjugate to 
$(B, T)$. 
\par
Next assume that $N$ is even. In this case, $N_H(T)$ is generated by 
$s_1, \dots, s_{n-1}$, $t_{n-1}t_n$  and $T$. $N_H(T)/T$ is isomorphic to the Weyl 
group $W_n'$ of type $D_n$.  Moreover, $N_{G^{\th}}(T)$ is generated by $N_H(T)$ and 
$t_n$, hence $N_{G^{\th}}(T)/T \simeq W_n$. 
One can check that 
$B_1 = t_n B t_n\iv$ is not contained in the set of $W '_n$-conjugates of $B$.  
It follows that the $\th$-stable pair $(B_1, T)$ is not $H$-conjugate to $(B, T)$, 
and $(B,T), (B_1, T)$ give representatives of $H$-conjugates of 
$\th$-stable pairs in $G$.

\para{1.10.} 
Let $\vD$ and $\vD_+$ be as in 1.5 with respect to the pair $(B, T)$ in 1.8.  
Let $\Ft = \Lie T$, and fix the basis of $\Ft^*$ as 
$\{ \ve_0, \ve_1,\dots, \ve_n, \ve_1',\dots, \ve_n'\}$ in the case where $N = 2n+1$, and 
$\{\ve_1, \dots, \ve_n, \ve_1',\dots, \ve_n'\}$ 
in the case where $N = 2n$.  Here $\ve_i$ (resp. $\ve_j', \ve_0$) 
is the weight vector corresponding to $e_i$ (resp. $f_j$, $e_0$).   
We write those vectors as $\wt\ve_1, \dots, \wt\ve_N$, in the order 
$\ve_1, \dots, \ve_n, \ve_0, \ve_n', \dots, \ve_1'$ if $N = 2n + 1$, and 
$\ve_1, \dots, \ve_n, \ve_n', \dots, \ve_1'$ if $N = 2n$. 
Thus $\vD_+$ can be written as
\begin{equation*}
\vD_+ = \{ \wt\ve_i - \wt\ve_j \mid 1 \le i < j \le N \}.
\end{equation*}  

It follows from the description of $B^{\io\th}$ in 1.8, we see that 
\begin{equation*}
\tag{1.10.1}
\vD_0^+ = \emptyset, \qquad \vD_0^- = \{ \wt\ve_i - \wt\ve_{N - i+1} \mid 1 \le i \le n\}.
\end{equation*}

Recall that $U(\vD_0) = \prod_{\a \in \vD_0}U_{\a}$ is a $\th$-stable subgroup of $U$.
It follows from (1.10.1), we have $U(\vD_0) = U(\vD_0)^{\io\th}$, which we denote by 
$\SD$. By using (1.10.1), $\SD$ can be written explicitly 
as follows. 
\begin{equation*}
\tag{1.10.2}
\begin{aligned}
\SD &=  \biggl\{ \begin{pmatrix}
                     1   &   0   &  0   \\
                     0   &   1_n &  c  \\
                     0   &   0   &  1_n
                  \end{pmatrix} \in U \ \bigg| \ c \text{ : diagonal} \biggr\} 
                    &\qquad   &(N = 2n+1), \\
\SD &=  \biggl\{ \begin{pmatrix}
                     1_n &  c  \\
                     0   &  1_n
                  \end{pmatrix} \in U  \ \big|\ c \text{ : diagonal} \biggr\} 
                    &\qquad   &(N = 2n). \\
\end{aligned}
\end{equation*}
Thus $\SD$ is $T^{\th}$-stable, and $\SD \simeq \Bk^n$. 

\para{1.11.}
Let $P = LU_P$ be a $\th$-stable parabolic subgroup of $G$ containing $B$, where 
$L$ is the Levi subgroup of $P$ containing $T$ and $U_P$ 
is the unipotent radical of $P$. Here we consider the special case where 
$L^{\th} \simeq (GL_1)^a \times GL_{N_0}^{\th}$ with $N_0 = N - 2a$ for 
$0 \le a \le n$ (we understand that 
$GL_1^{\th} = \{ \pm 1\}$, $GL_0^{\th} = \{ 1\}$).  
Put $\SD_P = \SD \cap U_P$.  Then $\SD_P$ is a closed subgroup of $U_P$, and 
is contained in $U_P^{\io\th}$.  We have $\SD_P \simeq \Bk^a$.  
It is easy to see that $\SD_P$ is stable under the conjugation action of $L$ on 
$U_P$. More precisely, we have
\par\medskip\noindent
(1.11.1) \ Assume that $x \in L^{\io\th}$ and $u \in \SD_P$.  Then we have 
$xu = ux$. 
\par\medskip
The proof is done by a direct computation by using (1.10.2). 
Note that in the Lie algebra case, this corresponds to the fact that
if $x, y \in \Fg^-$ such that $[x,y] \in \Fg^-$, then $[x, y] = 0$ as
$[x, y] \in \Fg^+$. 

\remark{1.12.} 
For the comparison, we briefly discuss the case of symplectic groups. 
Let $V$ be an $N = 2n$-dimensional vector space over $\Bk$, and let 
$G = GL_N = GL(V)$. The involutive automorphism $\th : G \to G$ is defined 
as in 1.7, but by replacing $J$ by $J = \begin{pmatrix}
                                          0 & 1_n  \\
                                         -1_n &  0
                                        \end{pmatrix}$. 
Then $H = G^{\th}$ is the symplectic group $Sp_N$, and we consider $G^{\io\th}$ with 
$H$-action. We choose a pair of $\th$-stable Borel subgroup $B$ and a $\th$-stable 
maximal torus $T \subset B$ as in [SS, 1]. Then the root system 
$\vD$ of $G$ and the set of positive roots with respect to the pair 
$(B,T)$ are given similarly to
1.10. In particular, the set $\vD_0 = \vD_0^+ \cup \vD_0^-$ is defined, and 
$\vD_0^{\pm}$ are determined by using the computation in [SS, 1], namely, we have
\begin{equation*}
\tag{1.12.1}
\vD_0^+ = \{ \wt\ve_i - \wt\ve_{N-i+1} \mid 1 \le i \le n\}, \qquad
\vD_0^- = \emptyset.
\end{equation*}
(1.12.1) shows that in the symplectic case, an opposite situation 
occurs compared to the orthogonal case (1.10.1).  
\par\bigskip
\section{Cuspidal local systems}

\para{2.1.}
In this section, we extend the notion of cuspidal local systems in the case 
of reductive groups given in [L1] to the 
case of symmetric spaces. 
Let $G = GL_N$, and $\th : G \to G$ be as in 1.7. 
Let $H = (G^{\th})^0$.
We denote the twisted action of $G$ on $G^{\io\th}$ given 
by $g : x \mapsto gx\th(g)\iv$ ($g \in G, x \in G^{\io\th}$) by $g\cdot x$. 
\par
Let $P = LU_P$ be a $\th$-stable parabolic subgroup of $G$ containing $B$, where
$L$ is the $\th$-stable Levi subgroup containing $T$ and $U_P$ is the unipotent radical 
of $P$.
Then $L^{\th} \simeq GL_{n_1} \times \cdots \times GL_{n_a} \times GL^{\th}_{N_0}$, where 
$N_0 + \sum_{i=1}^a 2n_i = N$. We have
\begin{equation*}
\tag{2.1.1}
L^{\io\th} \simeq GL_{n_1}\times \cdots \times GL_{n_a} \times GL_{N_0}^{\io\th},
\end{equation*}
on which $L^{\th}$ acts via the conjugation action. 
The natural projection $P \to P/U_P \simeq L$ induces a surjective map 
$\eta_P : P^{\io\th} \to L^{\io\th}$. Note that $L^{\th}$ acts on $U_P^{\io\th}$ 
by conjugation action.  Since $L^{\io\th}$ is connected by (2.1.1), 
we have an isomorphism $L^{\io\th} \simeq L/L^{\th}$ by 1.1. 
We note the following. 
\par\medskip\noindent
(2.1.2) \ The map 
$L \times U_P^{\io\th} \to P^{\io\th}, (\ell, u) \mapsto \ell\cdot u$ 
induces an isomorphism $L \times^{L^{\th}}U_P^{\io\th} \simeq P^{\io\th}$.
Moreover, the map $\eta_P$ is regarded as an affine space bundle over $L^{\io\th}$
through the maps 
\begin{equation*}
P^{\io\th} \simeq L \times^{L^{\th}}U_P^{\io\th} \to L/L^{\th} \simeq L^{\io\th}, 
\end{equation*} 
where for $g = \ell\cdot u \in P^{\io\th}$, $\eta_P(g) = \ell\th(\ell)\iv$. 
\par\medskip
In fact,  take $x \in P^{\io\th}$, and write it as $x = \ell_1u_1$ with 
$\ell_1 \in L, u_1 \in U_P$. Then we have
$\th(\ell_1) = \ell_1\iv, \th(u_1) = \ell_1u_1\iv \ell_1\iv$. 
Since $\ell_1 \in L^{\io\th}$, there exists $\ell \in L$ such that 
$\ell_1 = \ell\th(\ell)\iv$. If we put $u = \th(\ell)\iv u_1\th(\ell)$, 
one can check that $u \in U_P^{\io\th}$, and we have  
\begin{equation*}
x = \ell_1u_1 = \ell\bigl(\th(\ell)\iv u_1 \th(\ell)\bigr)\th(\ell)\iv 
              = \ell u  \th(\ell)\iv. 
\end{equation*} 
The pair $(\ell, u)$ is uniquely determined from $g$ modulo the action of 
$L^{\io\th}$, $y : (\ell, u) \mapsto (\ell y\iv, yuy\iv)$ for $y \in L^{\io\th}$. 
Hence (2.1.2) holds. 
\par\medskip
Put $P_H = (P^{\th})^0$ and $L_H = (L^{\th})^0$. 
Then $P_H$ is the parabolic subgroup of $H$ containing $B_H$. 
Let $W_H = N_H(T_H)/T_H$ be the Weyl group of $H$.  We define a Weyl subgroup 
$W_{L_H}$ of $W_H$ as the Weyl group of $L_H$.   
Let $\SO$ be an $H$-orbit in $G^{\io\th}$ and 
$\SO_L$ an $L_H$-orbit in $L^{\io\th}$.   
Let $Z_L$ be the center of $L$.  Then $Z_L^{\io\th} \simeq (GL_1)^a$. 

Consider the varieties

\begin{align*}
Z = \{(x, &gP_H, g'P_H) \in G^{\io\th} \times H/P_H \times H/P_H   \\
           &\mid g\iv xg  \in \eta_P\iv(Z_L^{\io\th}\SO_L), 
               {g'}\iv xg' \in \eta_P\iv(Z_L^{\io\th}\SO_L)\}, \\
Z' = \{(x, &gP_H, g'P_H) \in G^{\io\th} \times H/P_H \times H/P_H   \\
           &\mid g\iv xg  \in \eta_P\iv(\SO_L), 
               {g'}\iv xg' \in \eta_P\iv(\SO_L)\}. 
\end{align*}
We consider the partition $H/P_H \times H/P_H = \coprod \OO_{\w}$ into $H$-orbits, where 
$\OO_{\w}$ is an $H$-orbit containing $(P_H, \w P_H)$ for 
$\w \in W_{L_H}\backslash W_H/W_{L_H}$. 
Let 
$Z_{\w} = p\iv(\OO_{\w})$, $Z_{\w}' = {p'}\iv(\OO_{\w})$, where $p : Z \to H/P_H \times H/P_H$, 
$p' : Z' \to H/P_H \times H/P_H$ are projections onto second and third factors. 
An orbit $\OO_{\w}$ is said to be good if $P_H$ and $\w P_H\w\iv$ have a common Levi subgroup,
otherwise $\OO_{\w}$ is said to be bad. The good orbit corresponds to $\w$ such that 
$\w W_{L_H} = W_{L_H}\w$.  
\par
In order to proceed the induction process smoothly for proving the proposition below, 
we replace $G$ by groups appearing as a Levi subgroup of some parabolic subgroup 
of $G$, namely we consider 
\begin{equation*}
\tag{2.1.3}
G = G_0 \times \prod_{i = 1}^t (G_ i \times G_i),
\end{equation*}
where $G_0 = GL_{N}, G_i = GL_{N_i}$.  We consider an involutive automorphism 
$\th : G \to G$ such that $\th$ acts on $G_i \times G_i$ as a permutation of factors, 
and $\th$ acts on $G_0$ so that $G_0^{\th} \simeq O_{N}$. 
Hence 
\begin{equation*}
\tag{2.1.4}
G^{\io\th} \simeq G_0^{\io\th} \times \prod_{i = 1}^tG_i.
\end{equation*} 
Under this setting, $P = LU_P$ and the varieties $Z, Z'$ are defined similarly.
In particular, $L^{\th}$ can be written as 
$L^{\th} \simeq GL_{n_1} \times \cdots \times GL_{n_a} \times GL_{N_0}^{\th}$, where 
$GL^{\th}_{N_0}$ is a subgroup of $G^{\th}_0$. 
Put $\nu_H = \dim U^{\th}$, and $\ol \nu = \nu_{L_H}$. Let $c = \dim \SO$ and 
$\ol c = \dim \SO_L$. Put $r = \dim Z_L^{\io\th}$. 
We define $\vD_P$ by 
\begin{equation*}
\tag{2.1.5}
\vD_P =  (N - N_0)/2.  
\end{equation*}
\par
The following result is an extension of [L1, Proposition 1.2]. 
The proof is done along a similar line.  But since the Levi decomposition 
$P = LU_P$ does not hold for $P^{\io\th}$, we need a special care for it. 
The discussion based on the Levi decomposition in [L1] will be replaced 
by the discussion using (2.1.2). 

\begin{prop}   
Under the notation above, 
\begin{enumerate}
\item
For $\ol x \in \SO_L$, we have 
$\dim (\SO \cap \eta_P\iv(\ol x)) \le (c - \ol c)/2 + \vD_P/2$.
\item 
For $x \in \SO$, 
\begin{equation*}
\dim \{ gP_H \in H/P_H \mid g\iv xg \in \eta_P\iv(\SO_L)\} 
         \le (\nu_H - c/2) - (\ol\nu - \ol c/2) + \vD_P/2.
\end{equation*}
\item
Put $d = 2\nu_H - 2\ol\nu + \ol c + r$. Then $\dim Z_{\w} \le d + \vD_P$ if $\OO_{\w}$
is good, and $\dim Z_{\w} < d + \vD_P$ if $\OO_{\w}$ is bad.  Hence $\dim Z \le d + \vD_P$. 
\item 
Put $d' = 2\nu_H - 2\ol\nu  + \ol c$.  Then $\dim Z'_{\w} \le d' + \vD_P$ for any $\w$.
Hence $\dim Z' \le d' + \vD_P$.    
\end{enumerate}
\end{prop}

\begin{proof}
Since the proposition 
holds in the case where $G = T \times T$ for a torus $T$, 
we may assume that the proposition holds for a proper Levi subgroup $L$ of $G$.  
\par
Consider an orbit $\OO_{\w}$. Let $w \in W_H$ be a representative of 
$\w \in W_{L_H}\backslash W_H/W_{L_H}$, and $\dw \in N_H(T_H)$ a representative of $w$.  
In order to show (iii) and (iv), it is enough
to see that
\begin{align*}
\tag{2.2.1}
\dim \{ x \in \eta_P\iv&(Z_L^{\io\th}\SO_L) \mid 
       \dw\iv x\dw \in \eta_P\iv(Z_L^{\io\th}\SO_L)\}  \\
           &\le 2\nu_H - 2\ol\nu + \ol c + r - \dim \OO_{\w} + \vD_P,  \\
\tag{2.2.2}
\dim \{ x \in \eta_P\iv&(\SO_L) \mid \dw\iv x\dw \in \eta_P\iv(\SO_L)\}  \\
           &\le 2\nu_H - 2\ol\nu + \ol c - \dim \OO_{\w} + \vD_P,
\end{align*} 
and 
that (2.2.1) is a strict inequality if $\OO_{\w}$ is bad.
Now by (2.1.2), an element 
$x \in P^{\io\th} \cap {}^wP^{\io\th}$ can be written as 
$x = \ell\cdot u = \ell'\cdot u'$ with $\ell \in L$, $\ell' \in {}^wL$,  
$u \in U^{\io\th}_P, u' \in {}^wU^{\io\th}_P$. Moreover, there exists a unique element 
$z \in L \cap {}^wL$ such that $\ell = zy', \ell' = zy$ 
with $y' \in L \cap {}^wU_P$, $y \in {}^wL \cap U_P$. 
Put 
\begin{align*}
\wt Y = \{ (&u,u',y,y', z) \in U^{\io\th}_P \times {}^wU^{\io\th}_P  
         \times ({}^wL \cap U_P) \times (L \cap {}^wU_P) 
         \times (L \cap {}^wL)   \\ 
   &\mid y'\cdot u = y\cdot u', zy'\th(zy')\iv \in Z_L^{\io\th}\SO_L, 
    \dw\iv zy\th(zy)\iv \dw \in Z_L^{\io\th}\SO_L \} 
\end{align*}
and let $Y$ be the quotient of $\wt Y$ by  $E^{\th}$, where 
$E = ({}^wL \cap U_P)\times (L \cap {}^wU_P) \times (L \cap {}^wL)$ and 
$E^{\th}$ acts on $\wt Y$ by 
\begin{equation*}
(\ell_1, \ell_2, z_1) : (u,u',y,y',z) \mapsto 
    (\ell_2u\ell_2\iv, \ell_1u'\ell_1\iv,  z_1y\ell_1\iv, z_1y'\ell_2\iv, zz_1\iv).
\end{equation*} 
Then (2.2.1) can be rewritten as 
\begin{equation*}
\tag{2.2.3}
\dim  Y \le 2\nu_H - 2\ol\nu + \ol c + r - \dim \OO_{\w} + \vD_P,  
\end{equation*}
where it is the strict inequality if $\OO_{\w}$ is bad. 
Moreover (2.2.2) can be rewritten as the inequality obtained by dropping 
$Z_L^{\io\th}$ in the definition of $\wt Y$, and by dropping $r$ from (2.2.3). 
\par
We now consider the projection $Y \to ({}^wL \cap U_P) \times (L \cap {}^wU_P) 
          \times (L \cap {}^wL)$ by $(u,u',y,y',z) \mapsto (y,y',z)$.
For fixed $y,y',z$, the fibre $Y_{y,y',z}$ can be written as
\begin{equation*}
Y_{y,y',z} = \{ (u, u') \in U^{\io\th}_P \times {}^wU^{\io\th}_P  
                  \mid y'\cdot u = y\cdot u'\}.  
\end{equation*}
We note that  
\begin{equation*}
\tag{2.2.4}
\dim Y_{y,y',z} = \dim (U_P \cap {}^wU_P)^{\io\th}.
\end{equation*} 
In fact, put 
\begin{align*}
\wt u &= ({y'}\iv yy')\iv \cdot u, \\
\wt u' &= {y'}\iv\cdot u'.
\end{align*}
Since $y \in U_P, y' \in L$, we have ${y'}\iv yy' \in U_P$.  Thus 
$\wt u \in U_P^{\io\th}$.  On the other hand, since $y' \in {}^wU_P$, 
we have $\wt u' \in {}^wU_P^{\io\th}$. Thus 
the variety $Y_{y,y',z}$ is isomorphic to the variety 
$\{ (\wt u, \wt u') \in U_P^{\io\th} \times {}^wU_P^{\io\th} \mid \wt u = \wt u'\}$. 
Hence $\dim Y_{y,y',z} = \dim (U_P \cap {}^wU_P)^{\io\th}$. (2.2.4) holds. 
\par
Since the fibres $Y_{y,y',z}$ have constant dimension, (2.2.3) can be rewritten as 
follows; 
\begin{align*}
\tag{2.2.5}
\dim D \le 2\nu_H - 2\ol\nu + \ol c + r 
          - \dim \OO_{\w} + \vD_P - \dim (U_P \cap {}^wU_P)^{\io\th},
\end{align*}  
where 
\begin{align*}
D = \{ (y,&y',z) \in ({}^wL \cap U_P) \times (L \cap {}^wU_P) \times (L \cap {}^wL)  \\
           &\mid zy'\th(zy')\iv \in Z_L^{\io\th}\SO_L, 
          \dw\iv zy\th(zy)\iv \dw \in Z_L^{\io\th}\SO_L \}/E^{\th},
\end{align*}
and the action of $E^{\th}$ is defined similarly.
\par
We now compute the difference of $\dim (U_P \cap {}^wU_P)^{\th}$ and 
$\dim (U_P \cap {}^wU_P)^{\io\th}$. According to the decomposition in (2.1.3), 
$U_P \cap {}^wU_P$ is the direct product of those subgroups corresponding 
to $G_0$ or $G_i \times G_i$ ($1 \le i \le t$) in (2.1.3).  For $G_i \times G_i$, 
its $\th$-fixed part and $\io\th$-fixed part are isomorphic. Hence we have only to
consider the part $G_0 \cap P$. We assume that $G_0 \cap L^{\io\th}$ is expressed 
as in the right hand side of (2.1.1). 
Put $n = [N/2], n_0 = [N_0/2]$. The Weyl subgroup $W_0$ of $W_H$ corresponding to $G_0$
is isomorphic to $S_n\ltimes (\BZ/2\BZ)^{n'}$, where $n' = n$ (resp $n' = n-1$) if 
$N$ is odd (resp. even),  which we identify with a subgroup of 
signed permutations of $n$ letters $\{ 1, \dots, n\}$.   For $w \in W_H$, 
let $w_0$ be the element corresponding to $W_0$ under the decomposition in (2.1.3).  
Put 
\begin{equation*}
\tag{2.2.6}
b_w = \sharp\{i \mid 1 \le i \le n-n_0, 1 \le w_0\iv(i) \le n-n_0 \}. 
\end{equation*}       
Note that the action of $\th$ on $U_P \cap {}^wU_P$ 
can be described from  the formula in (1.5.2), 
and its group version in 1.6, by using (1.10.1) and (1.10.2).   
In particular, we have

\begin{equation*}
\tag{2.2.7}
\dim (U_P \cap {}^wU_P)^{\th} = \dim (U_P \cap {}^wU_P)^{\io\th} - b_w.
\end{equation*}
\par
Since $L \cap {}^wP$ is a parabolic subgroup of $L$ with a Levi decomposition 
$L \cap {}^wP = (L \cap {}^wL)(L \cap {}^wU_P)$, and similarly for ${}^wL \cap P$, 
${}^wL \cap P = ({}^wL \cap L)({}^wL \cap U_P)$, we have 
$\dim (L \cap {}^wL) + \dim (L \cap {}^wU_P) + \dim ({}^wL \cap U_P) = \dim L$. 
Thus 
\begin{equation*}
\dim (P \cap {}^wP)^{\th} = 2\ol\nu + \dim T^{\th} + \dim (U_P \cap {}^wU_P)^{\th}.
\end{equation*}
It follows, by (2.2.7), that 
\begin{equation*}
\tag{2.2.8}
\dim (U_P \cap {}^wU_P)^{\io\th} = 2\nu_H - 2\ol\nu - \dim \OO_{\w} + b_w.
\end{equation*}
Hence (2.2.5) is equivalent to the form 
\begin{align*}
\tag{2.2.9}
\dim D \le \ol c + r + \vD_P - b_w.
\end{align*}  
Thus in order to prove (2.2.1), we have only to show (2.2.9), 
where the strict inequality holds if $\OO_{\w}$ is bad. 
\par
A similar discussion shows that, in order to prove (2.2.2), we have only to show 
\begin{equation*}
\tag{2.2.10}
\dim D ' \le \ol c + \vD_P - b_w, 
\end{equation*}
where 
\begin{align*}
D '= \{ (y,&y',z) \in ({}^wL \cap U_P) \times (L \cap {}^wU_P) \times (L \cap {}^wL)  \\
           &\mid zy'\th(zy')\iv \in \SO_L, 
          \dw\iv zy\th(zy)\iv \dw \in \SO_L \}/E^{\th}. 
\end{align*}
\par
We consider (2.2.10). 
Since ${}^wP$ and $L$ contain a common maximal torus $T$, $Q = L \cap {}^wP$ 
is a $\th$-stable parabolic subgroup of $L$ with Levi decomposition 
$Q = MU_Q$, where $M = L \cap {}^wL$ and $U_Q = L \cap {}^wU_P$. 
Hence by replacing $G, P, L$ by $L, Q, M$, one can define a map 
$\eta_Q : Q^{\io\th} \to M^{\io\th}$ as in the case of $\eta_P$.   Similarly, for a parabolic 
subgroup $Q' = {}^wL \cap P = MU_{Q'}$ of ${}^wL$ with $U_{Q'} = {}^wL \cap U_P$, 
the map $\eta_{Q'} : Q'^{\io\th} \to M^{\io\th}$ can be defined. 
Here $\eta_Q(zy'\th(zy')\iv) = z\th(z)\iv$ for $z \in M, y' \in U_Q$, and similarly, 
$\eta_{Q'}(zy\th(zy)\iv) = z\th(z)\iv$. We note that 
\par\medskip\noindent
(2.2.11) \ There exist finitely many $(M^{\th})^0$-orbits $\wh\SO_1, \dots, \wh\SO_m$ in 
$M^{\io\th}$ such that $z\th(z)\iv$ is contained in $\bigcup_i\wh\SO_i$  
if $(y,y', z) \in D '$. 
\par\medskip
In fact, by definition of $D'$, $z\th(z)\iv$ is contained in 
$\SO_LU_Q \cap M^{\io\th}$, hence its semisimple part $z_1$ is contained 
in $\SO_s$, where $\SO_s$ is a single $L_H$-orbit in $L^{\io\th}$ 
obtained from a semisimple element, say $s \in M ^{\io\th}$.
But $\SO_s \cap M^{\io\th}$ splits into finitely many $(M^{\th})^0$-orbits.   
Here note that 
$Z_M (s)$ is $\th$-stable, and 
$Z_M(s)^{\io\th}$ has only finitely many 
unipotent $Z_M^0(s)^{\th}$-orbits ([R, Proposition 7.4]). 
(2.2.11) follows from this. 
\par\medskip
Let $\pi_3 : D' \to M ^{\io\th}$ be the map defined by 
$(y,y',z) \mapsto \ol z = z\th(z)\iv$.
By (2.2.11), $\ol z \in \bigcup_i\wh \SO_i$, and  
for each $\ol z \in \wh\SO_i$, $\pi_3\iv(\ol z)$ is isomorphic to the product 
of the varieties as in (i), namely, 
$(\SO_L \cap \eta_Q\iv(\ol z)) \times (\SO_{{}^wL} \cap \eta_{Q'}\iv(\ol z))$. 
Hence by induction hypothesis, we have 
\begin{equation*}
\dim \pi_3\iv(\ol z) \le \frac{1}{2}(\ol c - \dim \wh\SO_i) 
               + \frac{1}{2}(\ol c - \dim\wh\SO_i) + \vD_Q.
\end{equation*} 
(Note that $\vD_Q = \vD_{Q'}$).   
It follows that $\dim \pi_3\iv(\ol\SO_i) \le \ol c + \vD_Q$. 
Since this is true for any $i$, we have $\dim D' \le \ol c + \vD_Q$.
In order to show (2.2.10), it is enough to see that 
\begin{equation*}
\tag{2.2.12}
\vD_Q \le \vD_P - b_w.
\end{equation*}
We can express $\vD_Q$ as 
\begin{equation*}
\vD_Q = \sharp\{ i \mid n-n_0 +1 \le i \le n, 1 \le w_0\iv(i) \le n - n_0 \}.
\end{equation*} 
By comparing this with (2.2.6), we obtain (2.2.12). 
Thus (2.2.10) is proved. 
\par
Next consider (2.2.9). 
A similar argument as in the proof of (2.2.10) shows that 
there exist finitely many orbits $\wh \SO'_1, \dots, \wh \SO'_{m'}, 
\wh\SO''_1, \dots, \wh \SO''_{m''}$ in $(M^{\io\th})^0$ such that 
$z\th(z)\iv$ is contained in the intersection of 
$Z_L^{\io\th}(\wh\SO'_1 \cup \cdots \cup \wh \SO'_{m'})$ and 
$Z_{{}^wL}^{\io\th}(\wh\SO''_1 \cup \cdots \cup \wh\SO''_{m''})$.
Since $Z_L$ and $Z_{{}^wL}$ are contained in the center of $M$, 
$z\th(z)\iv$ is contained in 
$(Z_L \cap Z_{{}^wL})^{\io\th}(\wh\SO_1\cup \cdots \cup \wh\SO_m)$
for some $(M^{\th})^0$-orbits $\wh\SO_1, \dots, \wh\SO_m$ in $M^{\io\th}$. 
Thus as in the proof of (2.2.10), we have 
$\dim D' \le \ol c + \dim (Z_L \cap Z_{{}^wL})^{\io\th} + \vD_Q$.  
Here $\dim (Z_L \cap Z_{{}^wL})^{\io\th} \le r$, and the strict inequality 
holds if $\OO_{\w}$ is bad. (2.2.9) follows from this by using (2.2.12).  
Hence we have proved (iii) and (iv), assuming the induction hypothesis. 
\par
Next we show (ii).
Put $Z'_{\SO} = \{ (x, gP_H, g'P_H) \in Z' \mid x \in \SO\}$.  
If $Z'_{\SO}$ is empty, then the variety in (ii) is also empty, and 
the inequality holds.  So
we assume that $Z'_{\SO}$ is non-empty. 
From (iv), we have $\dim Z'_{\SO} \le d' + \vD_P$. Consider the projection 
$p: Z'_{\SO} \to \SO$ to the first factor.  Then each fibre is isomoprhic to 
the product of two copies of the variety in (ii).   Thus we have 
\begin{align*}
\dim \{ gP_H \in H/P_H \mid g\iv xg \in \eta_P\iv(\SO_L)\} 
         &= (\dim Z'_{\SO} - \dim\SO)/2  \\
         &\le (d' + \vD_P - c)/2 \\
         &= \nu_H - \ol\nu + (\ol c - c + \vD_P)/2. 
\end{align*}
Hence (ii) holds. 
\par
Finally we show (i).  Consider the variety 
$R = \{ (x, gP_H) \in \SO \times H/P_H \mid x \in \eta_P\iv(\SO_L)\}$. 
By projecting to the first factor, and by using (ii), we see that 
$\dim R \le \nu_H - \ol\nu + (c + \ol c)/2 + \vD_P/2$. 
If we project to the second factor, each fibre is isomorphic to the variety 
$\SO \cap\eta_P\iv(\SO_L)$.  
Hence 
\begin{align*}
\dim (\SO \cap \eta_P\iv(\SO_L)) &\le \nu_H -\ol\nu + 
        (c + \ol c)/2  + \vD_P/2 - \dim H/P_H \\
        &= (c + \ol c)/2 + \vD_P/2.
\end{align*}
Now we consider the map $\SO \cap \eta_P\iv(\SO_L) \to \SO_L$ by $x \mapsto \eta_P(x)$. 
Then each fibre is isomorphic to the variety considered in (i). Hence the dimension 
of this variety is $\le (c + \ol c)/2 + \vD_P/2 - \ol c = (c - \ol c)/2 + \vD_P/2$. 
This proves (i).  The proposition is proved. 
\end{proof}

\para{2.3.}
We keep the setting in 2.1.
Let $G\uni$ be the set of unipotent elements in $G$, and put 
$G^{\io\th}\uni = G^{\io\th} \cap G\uni$. By [R, proposition 7.4], 
$G^{\io\th}\uni$ has finitely many $H$-orbits.  We define $L^{\io\th}\uni$
for $L$ similarly to $G$.    
Let $\SO \subset G^{\io\th}\uni, \SO_L \subset L^{\io\th}\uni$.
Take $u \in \SO, v \in \SO_L$, and fix them.
We define varieties
\begin{align*}
Y_{u,v} &= \{ gZ^0_{L_H}(v)U_P^{\th} \mid g \in H, g\iv ug \in \eta_P\iv(v) \}, \\ 
\wt Y_{u,v} &= \{ g \in H \mid g\iv ug \in \eta_P\iv(v) \}.
\end{align*} 
Note that $\eta_P\iv(v) = (vU_P)^{\io\th}$, hence $Y_{u,v}$ is well-defined.
$Z_H(u) \times Z_{L_H}(v)U_P^{\th}$ acts on $\wt Y_{u,v}$ by 
$(x,y) : g \mapsto xgy\iv$ ($x \in Z_H(u), y \in Z_{L_H}(v)U_P^{\th}, g \in \wt Y_{u,v}$).
Let $\f : \wt Y_{u,v} \to Y_{u,v} \simeq \wt Y_{u,v}/Z_{L_H}^0(v)U_P^{\th}$, 
$\vf : \wt Y_{u,v} \to Z^0_H(u)\backslash \wt Y_{u,v}$ be the quotient maps.
We define $\xi : \wt\SO =  H/Z^0_H(u) \to H/Z_H(u) \simeq \SO$ by 
$gZ_H^0(u) \mapsto gug\iv$.  Then $\xi$ is a finite Galois covering with Galois group 
$A_H(u) = Z_H(u)/Z_H^0(u)$.
We have the following commutative diagram.

 \begin{equation*}
\tag{2.3.1}
\begin{CD}
Y_{u,v} @<\f<< \wt Y_{u,v} @>\vf>> Z_H^0(u)\backslash \wt Y_{u,v} \\
    @.           @V\tau VV           @VV\ol\tau V   \\
        @.    \SO \cap \eta_P\iv(v)     @<\xi<<   \xi\iv(\SO \cap \eta_P\iv(v)),        
\end{CD}
\end{equation*}
\par\noindent
where $\tau : g \mapsto g\iv ug, \ol\tau : Z_H^0(u)g \mapsto g\iv Z_H^0(u)$.  
Note that $\ol\tau$ gives an isomorphism 
$Z_H^0(u)\backslash \wt Y_{u,v} \isom \xi\iv(\SO \cap \eta_P\iv(v))$.
Put
\begin{align*}
\tag{2.3.2}
\d &= (\dim \SO -\dim \SO_L)/2 + \vD_P/2, \\
\tag{2.3.3}
s  &= (\dim Z_H(u) - \dim Z_{L_H}(v))/2 + \vD_P/2.
\end{align*}
By (2.3.1), we have 
$\dim (\SO \cap \eta_P\iv(v)) = \dim \xi\iv(\SO \cap \eta_P\iv(v))
     = \dim Z_H^0(u)\backslash \wt Y_{u,v}$. 
Hence if we put $d_{u,v} = \dim (\SO \cap \eta_P\iv(v))$, we have
\begin{align*}
\dim \wt Y_{u,v} &= d_{u,v} + \dim Z_H(u), \\
\dim Y_{u,v}     &= d_{u,v} + \dim Z_H(u) - \dim Z_{L_H}(v) - \dim U_P^{\th}. 
\end{align*}
\par
On the other hand, by Proposition 2.2 (i), we have $d_{u,v} \le \d$. 
It follows that $\dim Y_{u,v} \le s$, and that the equality holds if and only if 
$d_{u,v} = \d$. 
Let $I(Y_{u,v})$ be the set of irreducible components of $Y_{u,v}$ of dimension $s$.
Similarly, let $I(\wt Y_{u,v})$ (resp. $I_{u,v}$) be the set of irreducible components 
of $\wt Y_{u,v}$ (resp. $\xi\iv(\SO \cap \eta_P\iv(v))$) of dimension $\d + \dim Z_H(u)$
(resp. dimension $\d$). 
By (2.3.1), we have a natural bijection $I(Y_{u,v}) \simeq I(\wt Y_{u,v}) \simeq I_{u,v}$.     
\par
Put $A_H(u) = Z_H(u)/Z_H^0(u)$. 
Since $Z_H(u)$ acts on $Y_{u,v}$ from the left, $A_H(u)$ acts on $I(Y_{u,v})$ 
as permutations of irreducible components. 
Similarly, $A_H(u)$ acts on $I(\wt Y_{u,v})$, $I_{u,v}$, and the above bijection turns out 
to be $A_H(u)$-equivariant. 

\definition{2.4.}  
Let $\SO$ be an $H$-orbit in $G^{\io\th}\uni$, and take $u \in \SO$. 
$\tau \in A_H(u)\wg$ is said to be cuspidal if $\tau$ does not appear in the 
permutation representation of $A_H(u)$ on $I(Y_{u,v})$ for any $\th$-stable 
Levi subgroup $L$ of any $\th$-stable proper parabolic subgroup $P$ of $G$, and for any 
$v \in L^{\io\th}\uni$.
An $H$-equivariant simple local system $\SE$ on $\SO$ corresponding to 
$\tau \in A_H(u)\wg$ is said to be cuspidal if $\tau$ is cuspidal.  
\par\medskip
It follows from the definition that if $\tau$ is cuspidal, then its dual representation 
$\tau^*$ is also cuspidal.

\begin{lem}  
The local system $\SE$ on $\SO$ is cuspidal if and only if for any proper 
$\th$-stable parabolic subgroup $P$ of $G$, and for any $v \in L^{\io\th}\uni$, 
we have 
\begin{equation*}
H^{2\d}_c(\SO \cap \eta_P\iv(v), \SE) = 0.
\end{equation*}
\end{lem}

\begin{proof}
Let $\xi : \wt\SO \to \SO$ be the finite Galois covering with Galois group $A_H(u)$.
Then $\xi_*\Ql$ is a semisimple local system on $\SO$ equipped with $A_H(u)$-action, and 
we have $\SE = \Hom_{A_H(u)}(\tau, \xi_*\Ql)$. Hence
\begin{align*}
H^i_c(\SO \cap \eta_P\iv(v),\SE) &\simeq 
       (H_c^i(\SO \cap \eta_P\iv(v), \xi_*\Ql)\otimes \tau^*)^{A_H(u)}  \\
               &\simeq (H_c^i(\xi\iv(\SO \cap \eta_P\iv(v)), \Ql)\otimes \tau^*)^{A_H(u)},
\end{align*} 
where $\tau^*$ is the dual representation of $\tau$. 
This implies that the condition $\tau$ does not appear in the permutation representation 
of $I_{u,v}$ is equivalent to the condition $H^{2\d}_c(\SO \cap \eta_P\iv(v),\SE) = 0$. 
Since $I_{u,v} \simeq I(Y_{u,v})$ with $A_H(u)$-action, the lemma follows. 
\end{proof}

\para{2.6.}
More generally, we consider $G$ and an involution 
$\th : G \to G$ as in (2.1.3).  The definition of cuspidal local system $\SE$ on 
$\SO$ can be generalized to this case, and Lemma 2.5 holds. 
Let $\SE$ be a local system on $\SO$.  Then, under the isomorphism in (2.1.4), 
$\SO \simeq \SO_0 \times \SO_1 \times \cdots \times \SO_t$, where 
$\SO_0 \subset G_0^{\io\th}$ and 
$\SO_i \subset (G_i \times G_i)^{\io\th} \simeq G_i$
for $i = 1, \dots, t$.  Thus $\SE$ can be written as 
$\SE = \SE_0\boxtimes \SE_1\boxtimes \cdots \boxtimes \SE_t$, where $\SE_i$ is a local 
system on $\SO_i$ for each $i$. 
By Lemma 2.5, it is easy to see that $\SE$ is cuspidal on $\SO$ if and only if 
$\SE_i$ is cuspidal on $\SO_i$ for each $i$. 
But note that the definition of cuspidality for $\SE_i$  ($i \ge 1$) is exactly the same 
as the definition of cuspidality in the case of reductive groups in [L1, 2.4].  
It is well-known that in the case of $GL_n$, there does not exist a cuspidal local system 
unless $n = 1$.  Hence we have the following.

\begin{lem}  
Assume that $H = SO_N$. Let $L$ be a $\th$-stable Levi subgroup of 
a $\th$-stable parabolic subgroup of $G$. 
Let $\SO_L$ be an $L_H$-orbit in $L^{\io\th}\uni$.
If there exists a cuspidal local system on $\SO_L$, then 
$L_H \simeq (GL_1)^a \times SO_{N_0}$, where $a = (N - N_0)/2$.  
 
\end{lem}  

\remark{2.8.} 
In the case where $H = Sp_N$, Proposition 2.2 still holds by 
putting $\vD_P = 0$. 

\par\bigskip
\section{Admissible complexes}

\para{3.1.}
We follow the setting in 1.7, and consider $H = (G^{\th})^0 \simeq SO_N$. 
Let us fix a $\th$-stable Borel subgroup 
$B$ and a $\th$-stable maximal torus $T$ contained in $B$. 
By Lemma 1.2, the pair $(B, T)$ is $G^{\th}$-conjugate to the specific choice 
of the Borel subgroup and the maximal torus given in 1.8.  
Let $U$ be the unipotent radical of $B$. 
We define a subgroup $\SD$ of $U$ by $\SD = U(\vD_0)$.  
$\SD$ is conjugate under $G^{\th}$ to the corresponding group defined in 1.10.
In particular, we have $\SD = U(\vD_0)^{\io\th}$.
\par
Let $P = LU_P$ be a $\th$-stable parabolic subgroup of $G$ containing $B$ such that 
$T \subset L$. 
Here we assume that $L^{\th} \simeq (GL_1)^a \times GL_{N_0}^{\th}$ with 
$N_0 = N -2a$ for $0 \le a \le n$ (see Lemma 2.7). 
Let $\SO_L$ be an $L_H$-orbit in $L^{\io\th}\uni$, and consider 
$\vS = Z_L^{\io\th} \times \SO_L \subset L^{\io\th}$.
Note that $Z_L^{\io\th} \simeq (GL_1)^a$.    
Let $(Z_L)^{\io\th}\reg$ be the set of $(t_1, \dots, t_a) \in (GL_1)^a$ 
such that $t_i$ are all distinct, under the above isomorphism.
Put $\vS\reg = (Z_L)^{\io\th}\reg \times \SO_L$. Then $\vS\reg$ is open dense 
in $\vS$.     
Put $\SD_P = \SD \cap U_P$.
Then $\SD_P$ is a closed subgroup of $U_P$ contained in $U_P^{\io\th}$ 
such that $\SD_P \simeq \Bk^a$. 
Moreover by 1.11, $\SD_P$ is stable under the conjugation action of $L_H$ on $U_P^{\io\th}$. 
By (1.11.1), any element in $\vS$ commutes with any element in $\SD_P$.  Hence 
\begin{equation*}
\tag{3.1.1}
\vS\SD_P \subset P^{\io\th}. 
\end{equation*}
Note that $\dim \SD_P = a = \vD_P$ (see (2.1.6)).
\par
Let $\SD_P^0$ be the subset of $\SD_P \simeq \Bk^a$ consisting of 
$\Bxi = (\xi_1, \dots, \xi_a) \in \Bk^a$ such that $\xi_i \ne 0$ for any $i$. We define
varieties 
\begin{align*}
\wt Y_{(L,\vS)} &= \{ (x, gL_H) \in G^{\io\th} \times H/L_H 
                 \mid g\iv xg \in \vS\reg\SD^0_P \}  \\  
Y_{(L, \vS)} &= \bigcup_{g \in H}g(\vS\reg\SD^0_P)g\iv.
\end{align*}
We show a lemma.

\begin{lem}  
\begin{enumerate}
\item 
$ Y_{(L,\vS)}$ is a smooth, irreducible variety with  
\begin{equation*}
\tag{3.2.1}
\dim Y_{(L,\vS)} = 2\nu_H - 2\nu_{L_H} + \dim \vS + \vD_P. 
\end{equation*}
\item
Let $P' = L'U_{P'}$ be another $\th$-stable parabolic subgroup with $\th$-stable 
Levi subgroup $L'$, 
$(L',\vS')$ a pair defined for $L'$, similarly to $(L,\vS)$.  Assume that 
$(L',\vS')$ is not $H$-conjugate to $(L,\vS)$.  Then 
we have $Y_{(L,\vS)} \cap Y_{(L', \vS')} = \emptyset$.
\end{enumerate}
\end{lem}

\begin{proof}
For each $g \in G^{\io\th}$, let $g_s$ be the semisimple part of $g$ in $G^{\io\th}$, 
and $Z_H(g_s)$ be the stabilizer of $g_s$ in $H$.  
Put $Z_g = Z_{Z_H(g_s)}^0$, the connected center of $Z_H(g_s)$.
Then $H(g) = Z_H(Z_g)$ is a Levi subgroup of some parabolic subgroup of $H$.  
If $g \in Y_{(L,\vS)}$, $H(g)$ is a Levi subgroup conjugate to $L_H$ under $H$.
By identifying $H/N_H(L_H)$ with the set of Levi subgroups of $H$ conjugate to $L_H$, 
we define a map $\zeta : Y_{(L,\vS)} \to H/N_H(L_H)$ by $g \mapsto H(g)$.     
Then one can show that 
\begin{equation*}
\tag{3.2.2}
\zeta : Y_{(L,\vS)} \simeq H \times^{N_H(L_H)}\zeta\iv(L_H) \to H/N_H(L_H)
\end{equation*}
is a locally trivial fibration with fibre isomorphic to $\zeta\iv(L_H)$. 
We have
\begin{align*}
\zeta\iv(L_H) &= \bigcup_{w \in N_H(L_H)/L_H} w(\vS\reg\SD^0_P)w\iv.  
\end{align*}
Here $N_H(L_H)/L_H \simeq S_a \ltimes (\BZ/2\BZ)^{a'}$, where $S_a$ is the symmetric 
group of degree $a$, and $a' = a$ unless $a = n$ and $N$ is even, in which case
$a' = n-1$. 
$w \in N_H(L_H)/L_H$ leaves $\vS\reg$ stable, and $\s \in S_a$ leaves
$\SD^0_P$ stable. Hence
\begin{align*}
\tag{3.2.3}
\xi\iv(L_H) \simeq \vS\reg \times \bigcup_{w \in (\BZ/2\BZ)^{a'}}w\SD^0_Pw\iv. 
\end{align*}
Since $w\SD^0_Pw\iv$ are mutually disjoint for $w \in (\BZ/2\BZ)^{a'}$, we see that 
$\zeta\iv(L_H)$ is a disjoint union of smooth pieces $\vS\reg \times w\SD^0_Pw\iv$. 
Hence $\zeta\iv(L_H)$ is smooth, and so $Y_{(L,\vS)}$ is smooth by (3.2.2).
Since $\vS\reg\SD^0_P$ is irreducible, $Y_{(L,\vS)}$ is also irreducible.  
The dimension formula (3.2.1) follows from (3.2.2) and (3.2.3).   Thus (i) holds.
(ii) is immediate since for $g \in Y_{(L,\vS)}$, $H(g)$ determines a unique Levi 
subgroup in $H$, and once $L_H = H(g)$ is given, $g$ determines a unique $L_H$-orbit
$\SO_L$, hence determines $\vS$ uniquely.  The lemma is proved.  
\end{proof}

\remark{3.3}  
The definition of $Y_{(L,\vS)}$ given here depends 
on the special choice of $P$ as in 3.1.
The discussion in 3.1 can not be applied for arbitrary $L$. 
So, we cannot discuss the partition of $G^{\io\th}$ in terms of various $Y_{(L,\vS)}$
as given in [L1, 3.1].

\para{3.4.}  
By fixing $(L,\vS)$, put $Y = Y_{(L,\vS)}, \wt Y = \wt Y_{(L,\vS)}$. 
Recall the map $\eta_P : P^{\io\th} \to L^{\io\th}$.  
Let $\ol\vS$ be the closure of $\vS$ in $L^{\io\th}$.  
We define varieties $X, \wt X$ by
\begin{align*}
\wt X &= \{ (x, gP_H) \in G^{\io\th} \times H/P_H \mid 
      g\iv xg \in \eta_P\iv(\ol\vS)\}, \\   
   X &=  \bigcup_{g \in H}g\bigl(\eta_P\iv(\ol\vS)\bigr)g\iv
\end{align*}
and a map $\pi : \wt X \to X$ by $(x, gP_H) \mapsto x$. 
Then $\pi$ is proper and surjective. 
We have
\begin{equation*}
\tag{3.4.1}
\wt X \simeq H \times^{P_H}\eta_P\iv(\ol\vS). 
\end{equation*}

We have a lemma.

\begin{lem}  
\begin{enumerate}
\item 
$X$ is a closed irreducible subvariety of $G^{\io\th}$ such that 
the closure $\ol Y$ of $Y$ coincides with $X$. We have $\dim \wt X = \dim X$. 
\item 
The map $(x, gL_H) \mapsto (x, gP_H)$ gives an isomorphism 
$\g : \wt Y \isom \pi\iv(Y)$. 
\item
$Y$ is an open dense subset of $X$.  Hence $Y$ is an irreducible, locally 
closed smooth subvariety of $G^{\io\th}$.  
\end{enumerate}
\end{lem}

\begin{proof}
(i) \ Since $\eta_P$ is an affine space bundle with fibre isomorphic to $U_P^{\io\th}$ 
by (2.1.2), we have 
$\dim \eta_P\iv(\ol\vS) = \dim \vS + \dim U_P^{\io\th}$.
Thus 
(3.4.1) shows that $\wt X$ is irreducible with 
\begin{align*}
\dim \wt X &= \dim H - \dim P_H + \pi_P\iv(\ol\vS)  \\
            &= \dim H - \dim P_H + \dim \vS + \dim U_P^{\io\th}  \\
            &= 2\nu_H - 2\nu_{L_H} + \dim \vS - \dim U_P^{\th} + \dim U_P^{\io\th}. 
\end{align*} 
By (1.10.1) and (1.10.2), we have $\dim U_P^{\io\th} = \dim U_P^{\th} + \vD_P$.  Hence 
by comparing the last equality with (3.2.1), we have
\begin{equation*}
\tag{3.5.1}
\dim \wt X = \dim Y.
\end{equation*}
Since $\wt X$ is irreducible and $\pi$ is proper, $X$ is an irreducible closed 
subset of $G^{\io\th}$. Since $Y \subset X$, we have $\ol Y \subset X$. 
We have
\begin{equation*}
\dim  Y  \le \dim  X \le \dim \wt X = \dim  Y
\end{equation*} 
by (3.5.1).  Hence $\dim  Y = \dim  X = \dim \wt X$.  Since $X$ and $\ol Y$ are 
irreducible, closed with same dimension, we have $X = \ol Y$.  This proves (i). 
\par
(ii) \ The map $(x, gL_H) \mapsto (x, gP_H)$ gives a well-defined map 
$\g : \wt Y \to \pi\iv(Y)$. 
For $u \in U_P, t \in (Z_L)\reg$, if $tu = ut$, then $u \in \SD_P$.  Thus 
for $1 \ne u \in U_P^{\th}, t \in (Z_L)^{\io\th}\reg$, 
$t\iv u\iv tu$ produces an element in $U_P - \SD_P$. The injectivity of 
$\g$ follows from this. 
We will show that $\g(\wt Y) = \pi\iv(Y)$.    
Take an element in $\pi\iv(x)$ for $x \in \vS\reg\SD^0_P$, which is 
of the form $(x, gP_H)$ with $g\iv xg \in \eta_P\iv(\ol\vS)$. 
Let $x_s$ be the semisimple part of $x$.  Then $x_s \in (Z_L)^{\io\th}\reg$, and 
$Z^0_H(x_s) = L_H$. On the other hand, since $\ol\vS = Z_L^{\io\th}\ol\SO_L$, 
$g\iv x_sg \in Z_L^{\io\th}$. In particular, $Z_H(g\iv x_sg) \supset L_H$. 
But since $Z_H(g\iv x_sg) = g\iv L_Hg$, we have $g\iv L_Hg = L_H$, namely, 
$g \in N_H(L_H)$. Here note that $N_H(L_H)/L_H \simeq S_a\ltimes (\BZ/2\BZ)^{a'}$.
Thus $g \in N_H(L_H)$ leaves $\vS\reg$ invariant. 
By our assumption $g\iv xg \in \ol \vS U_P$, $g \in N_H(L_H)$ should be contained 
in the inverse image of $S_a$ under the map $N_H(L_H) \to N_H(L_H)/L_H$. 
Such a $g$ leaves $\SD^0_P$ invariant also. Hence $(x, gL_H) \in \wt Y$, and 
we have $(x, gP_H) = \g((x, gL_H))$.  Since any element in $\pi\iv(Y)$ is $H$-conjugate 
to the element discussed above, we see that $\g$ is surjective. 
Hence $\g$ is a bijection.  The inverse morphism  $\g\iv : \pi\iv(Y) \to \wt Y$  
is constructed from the above discussion.  (ii) is proved.  
\par
(iii) \ Since $X$ is irreducible, it is enough to show that $Y$ is an open subset of $X$. 
Other properties are already shown in Lemma 3.2.  Put 
$(Z_L)^{\io\th}_1 = Z_L^{\io\th} - (Z_L)^{\io\th}\reg$ and 
$\vS_1 = (Z_L)^{\io\th}_1\ol\SO_L$. Then 
$\vS_1$ is a closed subset of $\ol\vS$, and  
$X_1 = \bigcup_{g \in H}g(\eta_P\iv(\vS_1))g\iv$ is a closed subset of $X$. 
Thus $X - X_1$ is an open subset of $X$, and in fact, it coincides with 
the subset of $X$ consisting of $x$ such that 
the semisimple part of $x$ is $H$-conjugate to an element in $(Z_L)^{\io\th}\reg$. 
Let $\vS_2 = Z_L^{\io\th}(\ol\SO_L - \SO_L)$.  Then $\vS_2$ is a closed 
subset of $\ol\vS$, and $X_2 = \bigcup_{g \in H}g(\eta_P\iv(\vS_2))g$ is a closed 
subset of $X$. Since $Y = X - (X_1 \cup X_2)$, $Y$ is open in $X$. 
(iii) is proved. 
\end{proof}

\para{3.6.}
We consider the diagram 

\begin{align*}
\begin{CD}
\vS @<\a_0<<  \wh Y @>\psi_0>>  \wt Y @>\pi_0>>  Y, 
\end{CD}
\end{align*}
where 
\begin{equation*}
\wh Y = \{ (x, g) \in G^{\io\th} \times H \mid g\iv xg \in \vS\reg\SD^0_P \},
\end{equation*}
and maps are defined by 
$\pi_0(x, gL_H) = x, \psi_0(x, g) = (x, gL_H), \a_0(x,g) = \eta_P(g\iv xg)$.
Here $H$ acts on $\wt Y$ by $h : (x, gL_H) \mapsto (hxh\iv, hgL_H)$, and 
$\pi_0$ is $H$-equivariant with respect to the conjugation action of $H$ on $Y$.
On the other hand, $H \times L_H$ acts on $\wh Y$ by 
$(h, \ell) : (x,g) \mapsto (hxh\iv, hg\ell\iv)$, and $\psi_0$ is $H$-equivariant. 
$H \times L_H$ also acts on $\vS$ by 
$(h,\ell) :y \mapsto \ell y \ell\iv$, and $\a_0$ is $H \times L_H$-equivariant.  
\par
Let $\SE_1$ be an $L_H$-equivariant simple local system on $\vS$.
Since $\vS = Z_L^{\io\th} \times \SO_L$, $\SE_1$ can be written as 
$\SE_1 = \ScS\boxtimes \SE_1^{\dag}$, where $\ScS$ (resp. $\SE_1^{\dag}$) 
is a simple local system on $Z_L^{\io\th}$ (resp. $\SO_L$).  We say 
that $(\vS, \SE_1)$ is a cuspidal pair if $\SE_1^{\dag}$ is a cuspidal local 
system on $\SO_L$.  Note that this definition is weaker than the definition of 
the cuspidal pair in [L1, 2.4].  By abuse of the notation, we also say that 
$(\SO_L, \SE_1^{\dag})$ is a cuspidal pair if $\SE_1^{\dag}$ is cuspidal on $\SO_L$.      
Now consider an arbitrary $\SE_1$.  
Then $\SE_1$ is $(H \times L_H)$-equivariant with respect to the trivial action of $H$ 
on $\vS$, and $\a_0^*\SE_1$ is an $(H \times L_H)$-equivariant local system on $\wh Y$. 
Since $\psi_0$ is a principal bundle over $\wt Y$ with respect to the free action of $L_H$, 
there exists a unique local system $\wt\SE_1$ on $\wt Y$, up to isomoprhism,  such that
\begin{equation*}
\a_0^*\SE_1 \simeq \psi_0^*\wt\SE_1.
\end{equation*}
Since $\a_0$ is smooth with connected fibre $H \times \SD_P$, 
$\wt\SE_1$ turns out to be an $H$-equivariant simple local system on $\wt Y$. 
\par
Here $\SW = N_H(L_H)/L_H \simeq S_a\ltimes (\BZ/2\BZ)^{a'}$ (see the proof of 
Lemma 3.2).  Let 
$\SW_1$ be the stabilizer of $\vS\reg\SD^0_P$ in $\SW$.  Then $\SW_1 \simeq S_a$. 
We note that $\pi_0 : \wt Y \to Y$ is a finite Galois covering with group 
$\SW _1$. 
In fact, we have an isomorphism $\wt Y \simeq H \times^{L_H}\vS\reg\SD^0_P$, and 
by (3.2.2), the map $\pi_0$ is given by the canonical map
\begin{equation*}
\pi_0 : H \times^{L_H}\vS\reg\SD^0_P \to H \times^{N_H(L_H)}\z\iv(L_H), 
\end{equation*}
where $\z\iv(L_H) = \bigcup_{w \in (\BZ/2\BZ)^{a'}}w(\vS\reg \SD^0_P)w\iv$, see (3.2.3).
Thus $\pi_0$ is a finite Galois covering with group $\SW_1$.  
\par
We consider $(\pi_0)_*\wt\SE_1$. Since $\pi_0$ is a finite Galois covering, 
$(\pi_0)_*\wt\SE_1$ is a semisimple local system on $Y$. 
Put
\begin{equation*}
\SW_{\SE_1} = \{ n \in N_H(L_H) \mid n(\vS\reg\SD^0_P)n\iv  = \vS\reg\SD^0_P, 
                    \ad (n)^*\SE_1 \simeq \SE_1 \}/L_H, 
\end{equation*}
where $\ad(n) : \vS \to \vS, y \mapsto nyn\iv$. 
Then $\SW_{\SE_1}$ is a subgroup of $\SW_1$.  In a similar way as in 
[L1, Proposition 3.5], one can show that the endomorphism algebra 
$\SA_{\SE_1} = \End ((\pi_0)_*\wt\SE_1)$ of $(\pi_0)_*\wt\SE_1$ is isomorphic to
the (twisted) group algebra $\Ql[\SW_{\SE_1}]$.  
Thus $(\pi_0)_*\wt\SE_1$ can be decomposed as 
\begin{equation*}
\tag{3.6.1}
(\pi_0)_*\wt\SE_1 \simeq \bigoplus_{\r \in \SA_{\SE_1}\wg}\r \otimes ((\pi_0)_*\wt\SE_1)_{\r},
\end{equation*}
where $((\pi_0)_*\wt\SE_1)_{\r} = \Hom_{\SA_{\SE_1}} (\r, (\pi_0)_*\wt\SE_1)$ 
is a simple local system on $Y$ corresponding to $\r \in \SA_{\SE_1}\wg$. 

\para{3.7.}
We consider the following commutative diagram

\begin{equation*}
\begin{CD}
\ol\vS @< \a <<  \wh X   @> \psi >>  \wt X   @>\pi>>  X    \\  
@AAA           @AAA           @AAA          \\
\vS  @<\a' <<    \wh X_0   @>\psi' >>   \wt X_0,   
\end{CD}
\end{equation*}
where 
\begin{align*}
\wh X &= \{ (x,g) \in G^{\io\th} \times H \mid g\iv xg \in \eta_P\iv(\ol\vS) \},  \\ 
\wh X_0 &=  \{ (x,g) \in G^{\io\th} \times H \mid g\iv xg \in \eta_P\iv(\vS) \}, \\
\wt X_0 &= \{ (x, gP_H) \in G^{\io\th} \times H/P_H \mid g\iv xg \in \eta_P\iv(\vS) \}.
\end{align*}
Here $\vS, \wh X_0, \wt X_0$ are open dense smooth subsets of 
$\ol\vS, \wh X, \wt X$, and the horizontal maps are natural inclusions.  
$\psi$ is defined by $(x,g) \mapsto (x, gP_H)$, $\a$ is defined by 
$(x,g) \mapsto \eta_P(g\iv xg)$, and $\psi', \a'$ are their restrictions on 
$\wh X_0$. 
$H$ acts on $\wt X$ by $h : (x, gP_H) \mapsto (hxh\iv, hgP_H)$, and 
$H \times P_H$ acts on $\wh X$ by $(h,p) : (x,g) \mapsto (hxh\iv, hgp\iv)$.
With the trivial action of $U_{P_H}$, $P_H$ acts on $\ol\vS$, and $\a$ 
is $H \times P_H$-equivariant, $\psi$ is $H$-equivariant. 
Similar properties hold for $\a', \psi'$ as the restriction of $\a, \psi$. 
\par
Let $\SE_1$ be the local system on $\vS$ as before. 
Since $\psi'$ is a principal bundle over $\wt X_0$ with 
respect to the free action of $P_H$, by a similar argument as in 3.6, 
there exists a unique local system (up to isomorphism)  
$\ol\SE_1$ on $\wt X_0$ such that
\begin{equation*}
(\a')^*\SE_1 \simeq (\psi')^*\ol\SE_1.
\end{equation*}
Note that $\wt Y$ is an open dense subset of $\wt X_0$, and the restriction of 
$\ol\SE_1$ on $\wt Y$ coincides with $\wt\SE_1$ defined in 3.6. 
Put 
$K_{\ol\SE_1} = \IC(\wt X, \ol\SE_1)[\dim \wt X]$.
Then $K_{\ol\SE_1}$ is an $H$-equivariant simple perverse sheaf on $\wt X_0$.
Since $\a$ is smooth with connected fibre $H \times U_P^{\io\th}$, and $\psi$ 
is a principal bundle with respect to the free action of $P_H$, the discussion 
in 3.6 works in the level of perverse sheaves, and we have
\begin{equation*}
\a^*\IC(\ol\vS, \SE_1)[\dim \vS] \simeq \psi^*\IC(\wt X, \ol\SE_1)[\dim \wt X].
\end{equation*}

We consider the complex $\pi_*K_{\ol\SE_1}$ on $X = \ol Y$.
The following result is an analogue of [L1, Proposition 4.5]. 
Note that here we don't need to assume that $(\vS, \SE_1)$ is cuspidal. 

\begin{prop}  
We have $\pi_*K_{\ol\SE_1} \simeq \IC(X, (\pi_0)_*\wt\SE_1)[\dim X]$.
\end{prop} 

\begin{proof}
By Lemma 3.5 and Lemma 3.2, we have 
$\dim X = 2\nu_H - 2\nu_{L_H} + \dim \vS + \vD_P$. 
Then the proof of the proposition is done by an entirely similar 
way as the proof of Proposition 4.5 in [L1], by making use of Proposition 2.2.
We omit the details. 
\end{proof}

\para{3.9.}
By Proposition 3.8, $\pi_*K_{\ol\SE_1}$ is a semisimple perverse sheaf. 
By (3.6.1), it is decomposed as 

\begin{equation*}
\tag{3.9.1}
\pi_*K_{\ol\SE_1} \simeq \bigoplus_{\r \in \SA_{\SE_1}\wg}\r \otimes 
                   \IC(X, ((\pi_0)_*\wt\SE_1)_{\r})[\dim X].
\end{equation*}  

A simple perverse sheaf isomorphic to a direct summand of $\pi_*K_{\ol\SE_1}$ 
obtained from the various cuspidal pair $(\vS, \SE_1)$ is called an admissible 
complex. 

\remark{3.10.}   
In the case where $P = B$ and $L = T$, we have $\vS = T^{\io\th}$.
Then 
\begin{align*}
\wt X  &= \{ (x, gB_H) \in G^{\io\th} \times H/B_H \mid g\iv xg \in B^{\io\th}\}, \\
    X &= \bigcup_{g \in H}gB^{\io\th}g\iv, 
\end{align*}
and $\pi : \wt X \to X$ is the first projection. 
A similar map can be defined even in the case where $H = Sp_N$.  
However in that case, we have $\dim \wt X > \dim X$, and Lemma 3.5  
does not hold. Nevertheless, in the case where $\SE_1$ is a constant
sheaf $\Ql$ on $T^{\io\th}$, it was proved in [H] (see also [SS]) 
that $\pi_*\Ql$ is a semisimple complex equipped with $S_n$-action, and 
some modified formula of (3.9.1) holds.
If we consider the exotic symmetric space $G^{\io\th} \times V$ associated 
to the symplectic group $Sp(V)$, the map $\pi : \wt X \to X$ is also defined. 
In that case, an exactly analogous formula of (3.9.1) holds ([SS], [K]).    

\par\bigskip
\section{Sheaves on the variety of semisimple orbits}

\para{4.1.}
We keep the setting in Section 3. 
In particular, consider $Y = Y_{(L,\vS)}$ and $X = \ol Y$ 
with $\vS = Z_L^{\io\th}\SO_L$.
In this section, we assume that the local system $\SE_1$ on $\vS$ is 
of the form $\SE_1 = \Ql\boxtimes \SE^{\dag}_1$, where $\Ql$ is the constant sheaf 
on $Z_L^{\io\th}$ and $\SE^{\dag}_1$ is an $L_H$-equivariant simple local system 
on $\SO_L$. 

\par  
Let $\s' : G \to T/S_N$ be the Steinberg map with respect to $G = GL_N$. 
We have a natural embedding $A = T^{\io\th}/S_n \hra T/S_N$, where 
$A$ has a structure of an affine variety.
Let $G^{\io\th}_+$ be the subset of $G^{\io\th}$ consisting of $g \in G^{\io\th}$
such that its semisimple part $g_s$ is $H$-conjugate to an element in $T^{\io\th}$. 
(Note that $T^{\io\th}$ is not a maximal $\th$-anisotropic torus. See the example in 1.4, 
which corresponds to our setting with $N = 2$.)
Then the restriction of $\s'$ on $G^{\io\th}_+$ induces a map 
$\s : G^{\io\th}_+ \to A$. 
Since $\ol Y \subset G^{\io\th}_+$, 
one can consider $\s(\ol Y)$.  We have
\begin{equation*}
\s(\ol Y) = \s(\eta_P\iv(\ol\vS)) = \s(\ol\vS) = \s(Z_L^{\io\th}) \subset A, 
\end{equation*}
which we denote by $A_{Y}$. 
\para{4.2.}
Let $(B', T')$ be another $\th$-stable pair. 
By 1.9, the $H$-conjugate class of $(B',T')$ is described as follows. 
If $N$ is odd, we may choose $(B', T') = (B,T)$. 
If $N$ is even, then $(B',T') = (B, T)$ or $B' = t_n B t_n\iv, T' = T$, where
$t_n \in N_{G^{\th}}(T)$ is given as in 1.9. 
We put $B_1 = t_n B t_n\iv$. 
\par
Assume that $N$ is even, and let $P$ be a $\th$-stable parabolic subgroup
of $G$ containing $B$, and $L$ the Levi subgroup of $P$ containing $T$. 
Then $P' = t_nPt_n\iv$ is the $\th$-stable parabolic subgroup containing $B_1$, 
and $L' = t_n L t_n\iv$ is the Levi subgroup of $P'$ containing $T$. 
Note that $L' = L$ if $L \ne T$. 

\para{4.3.}
We now consider another $\th$-stable parabolic subgroup $P'$ containing $B'$ and 
its $\th$-stable Levi subgroup $L'$ containing $T$, where $B' = B$ or $B' = B_1$. 
We consider a set 
$\vS' \subset {L'}^{\io\th}$ similarly to $\vS \subset L^{\io\th}$. We shall denote 
by $Y', X'$, etc. various objects associated to $\vS', L', P'$ 
by attaching prime to corresponding objects $Y, X$, etc. associated to 
$\vS , L, P$.   
Put
\begin{align*}
Z = \{ (x, gP_H), (x',g'P'_H)) \in \wt X \times \wt X' \mid x = x' \}.  
\end{align*}
Thus $Z$ is isomorphic to the fibre product 
$\wt X \times_{G^{\io\th}} \wt X'$ of $\wt X$ and $\wt X'$ over $G^{\io\th}$. 
Here we assume that $X \cap X' \ne \emptyset$, otherwise $Z = \emptyset$. 
We define a map 
$\wt\s : Z \to A_Y \cap A_{Y'}$ by the composite $\s\circ p_1$ of the first projection 
$p_1 : Z \to X \cap X'$ and $\s$. 
For $a \in A_{Y}$, put $Z^a = \wt\s(a)\iv \subset Z$. 
\par
Since $\ol\vS = Z_L^{\io\th}\ol\SO_L$, $\ol\vS$ has a stratification 
$\ol\vS = \coprod_{\b}\vS_{\b}$ with smooth strata 
$\vS_{\b} = Z_L^{\io\th}\SO_{\b}$, 
where we put $\ol\SO_L = \coprod_{\b}\SO_{\b}$ ($\SO_{\b}$ : $L_H$-orbit in 
$L^{\io\th}\uni$). We put $\vS_{\b_0} = \vS = Z_L^{\io\th}\SO_L$ : 
the open dense stratum. By defining $\wt X_{\b}$ in a similar way as $\wt X_0$ in 3.7, 
we obtain 
a stratification $\wt X = \coprod_{\b}\wt X_{\b}$, with $\wt X_0 = \wt X_{\b_0}$,
where $\wt X_{\b}$ are locally closed subvarieties of $\wt X$. 
Given strata $\b$ of $\ol\vS$ and $\b'$ of $\ol\vS'$, 
we put $Z^a_{\b,\b'} = Z^a \cap (\wt X_{\b} \times_{G^{\io\th}} \wt X'_{\b'})$. 
Then the sets $Z^a_{\b,\b'}$ form a partition of $Z^a$ into locally closed pieces,
where $Z^a_{\b_0,\b'_0}$ is open dense in $Z^a$. 
Put
\begin{equation*}
\tag{4.3.1}
d_0 = 2\nu_H - \nu_{L_H} - \nu_{L'_H} + (\dim \SO_L + \dim \SO_{L'})/2 
      + (\vD_P + \vD_{P'})/2.
\end{equation*}
Take local systems $\SE_1$ on $\vS$ and $\SE_1'$ on $\vS'$, 
and consider the external tensor product 
$K_{\ol\SE_1}\boxtimes K'_{\ol\SE'_1}$ on $Z$.  We define 
\begin{equation*}
\tag{4.3.2}
\wt\ST = \SH^{-r - r'}(\wt\s_!(K_{\ol\SE_1}\boxtimes K'_{\ol\SE'_1})),
\end{equation*} 
where we put $r = \dim Z_L^{\io\th}$, $r' = \dim Z_{L'}^{\io\th}$. 
$\wt\ST$ is a constructible sheaf on $A_Y \cap A_{Y'}$. 

\par
Put $Z_0 = \wt X_0\times_{G^{\io\th}}\wt X'_0$.  Then $Z_0$ is an open dense subset
of $Z$.  Let $\wt\s_0: Z_0 \to A_Y \cap A_{Y'}$ be the restriction of $\wt\s$ on $Z_0$.
Noticing that 
\begin{equation*}
K_{\ol\SE_1}\boxtimes K'_{\ol\SE'_1}[-\dim X - \dim X']|_{Z_0} 
    = \ol\SE_1 \boxtimes \ol\SE'_1, 
\end{equation*} 
we define 
\begin{equation*}
\tag{4.3.3}
\ST = \SH^{-r -r'}((\wt\s_0)_!\bigl(K_{\ol\SE_1}\boxtimes K'_{\ol\SE'_1}|_{Z_0}\bigr)
         = \SH^{2d_0}\bigl((\wt\s_0)_!(\ol\SE_1\boxtimes\ol\SE'_1)\bigr).
\end{equation*}

$\ST$ is also a constructible sheaf on $A_Y \cap A_{Y'}$.  
Since $Z_0$ is open in $Z$, we have 
a natural map $\ST \to \wt\ST$. 
In order to obtain a relationship between $\wt\ST$ and $\ST$, we prepare a lemma, 
which is an analogue of [L1, Lemma 5.3].  The proof is done by a similar argument 
as in the proof of [loc. cit.], by using Proposition 2.2, and we omit the proof. 

\begin{lem}  
\begin{enumerate}
\item 
For any $\b,\b'$, we have
\begin{equation*}
\dim Z^a_{\b,\b'} \le d_0 - (\dim \SO_L + \dim \SO_{L'})/2  
    + (\dim \SO_{\b} + \dim \SO_{\b'})/2.
\end{equation*}
In particular, we have $\dim Z^a \le d_0$. 
\item
The natural map
\begin{equation*}
\hspace*{-5mm}
  H^{k + \dim X + \dim X'}_c(Z^a_{\b_0,\b_0}, \ol\SE_1\boxtimes\ol\SE'_1) \simeq 
   \BH^{k}_c(Z^a_{\b_0,\b_0}, K_{\ol\SE_1}\boxtimes K_{\ol\SE'_1}) 
        \to \BH_c^{k}(Z^a, K_{\ol\SE_1}\boxtimes K_{\ol\SE'_1})
\end{equation*}
is an isomorphism for $k > -r -r'$, and is surjective if $k = -r -r'$. 
It is an isomorphism for $k = -r -r'$ (i.e, for $k + \dim X + \dim X' = 2d_0$) 
if $a$ is such that $Z^a_{\b,\b'}$ is 
empty whenever exactly one of $\b,\b'$ is equal to $\b_0$ or $\b_0'$. 
\end{enumerate}
\end{lem}

\para{4.5.}
Let $\wt\ST_a$ (resp. $\ST_a$) be the stalk of $\wt\ST$ (resp. $\ST$) 
for $a \in A_Y \cap A_{Y'}$.  Then we have
\begin{align*}
\wt\ST_a &\simeq \BH^{-r -r'}_c(Z^a, K_{\ol\SE_1}\boxtimes K'_{\ol\SE'_1}), \\
\ST_a &\simeq \BH^{-r -r'}_c(Z^a_{\b_0,\b'_0}, K_{\ol\SE_1}\boxtimes K'_{\ol\SE'_1}), 
\end{align*}
and the natural map $\ST_a \to \wt\ST_a$ corresponds to the map 
in Lemma 4.4 (ii) for $k = -r -r'$.  Thus, by Lemma 4.4 (ii), we see that
\par\medskip\noindent
(4.5.1) \ The natural map of sheaves $\ST \to \wt\ST$ is surjective. 
\par\medskip
Let $\OO$ be an $H$-orbit in $H/P_H \times H/P'_H$, and put 
$Z_0^{\OO} = q\iv(\OO)$, where $q : Z_0 \to H/P_H \times H/P'_H$ 
is the projection $(x, gP_H, g'P'_H) \mapsto (gP_H, g'P'_H)$.
Put 
\begin{equation*}
\ST_{\OO} = \SH^{-r -r'}(\wt\s)_!\bigl(K_{\ol\SE_1}\boxtimes 
                                    K'_{\ol\SE'_1}|_{Z^{\OO}_0}\bigr)
          = \SH^{2d_0}(\wt\s_0)_!\bigl(\ol\SE_1\boxtimes \ol\SE'_1|_{Z^{\OO}_0}\bigr).
\end{equation*}

$\ST_{\OO}$ is a constructible sheaf on $A_Y \cap A_{Y'}$.  
Recall that $\SW = N(L_H)/L_H$ and $\SW_1 \simeq S_a$ is the subgroup of $\SW$. 
Let $N_H(L_H)_1$ be the inverse image of $\SW_1$ under the map $N_H(L_H) \to \SW$. 

We show a lemma.

\begin{lem}  
Assume that $(\vS, \SE_1), (\vS',\SE_1')$ are cuspidal pairs.  
Let $\OO = \OO_{\w}$ be an $H$-orbit corresponding to 
$\w \in W_{L_H} \backslash W_H /W_{L'_H}$, and $w \in W_H$ be 
a representative of $\w$.  Let $\dw \in N_H(T_H)$ be a representative of $w$.   
\begin{enumerate}
\item 
If $(L, \vS)$ and $(L', \vS')$ are not conjugate under $H$, then $\ST_{\OO} = 0$. 
\item 
Assume that $P = P', L = L'$ and $\vS = \vS'$. 
If $\dw \in N_H(T_H)$ is not 
contained in $N_H(L_H)_1$. Then $\ST_{\OO} = 0$.
\item
Assume that $N$ is even, and $P = B, P' = B_1, L = L' = T$.  Then $\ST_{\OO} = 0$. 
\end{enumerate}
\end{lem}  

\begin{proof}
In order to see $\ST_{\OO} = 0$, we have only to show that,  
for any $a \in A_Y \cap A_{Y'}$, 
\begin{equation*}
H^{2d_0}_c(Z^a \cap Z^{\OO}_{0}, \ol\SE_1\boxtimes\ol\SE'_1) = 0.
\end{equation*} 
Since $\dim (Z^a \cap Z^{\OO}_{0}) \le d_0$ (Lemma 4.4 (i)), by using the 
fibration $Z^a \cap Z^{\OO}_{0} \to \OO$, it is enough to see, 
for any $a \in A_Y \cap A_{Y'}$ and for any $g \in H$, 
$H_c^{2d_0 - 2\dim \OO}(V^a, j^*(\SE_1\boxtimes\SE'_1)) = 0$, where 
$\OO$ is the $H$-orbit containing $(P_H, \dw P'_H)$ (here $\dw \in N_H(T_H)$ 
is a representative of $w \in W_H$), and 
\begin{equation*}
V^a = \{ x \in \s\iv(a) \mid g\iv xg \in \eta_P\iv(\vS), 
             \dw\iv g\iv xg\dw \in \eta_{P'}\iv(\vS')\},
\end{equation*}
and $j : V^a \to \vS \times \vS'$ is defined by 
$j(x) = (\eta_P(g\iv xg), \eta_{P'}(\dw\iv g\iv xg\dw))$.
We have $g\iv xg \in P^{\io\th} \cap {}^w {P'}^{\io\th}$.
As in the proof of Proposition 2.2, one can write $g\iv xg$ as 
$g\iv xg = (zy')\cdot u = (zy)\cdot u'$, where $z \in L \cap {}^wL'$, 
$y' \in L \cap {}^wU_{P'}$, $y \in {}^wL' \cap U_P$, $u \in U_P^{\io\th}$, 
$u' \in {}^wU_{P'}^{\io\th}$.  
Thus $V^a$ can be described as follows; put
\begin{align*}
  \wt V^a =  &\{ (u,u', y, y', z) \in U_P^{\io\th} \times {}^wU_{P'}^{\io\th} 
    \times ({}^wL' \cap U_P) \times (L \cap {}^wU_{P'}) \times (L \cap {}^wL')  \\
      &\mid  y'\cdot u = y\cdot u', (zy')\th(zy')\iv \in \vS \cap \s\iv(a), 
               (zy)\th(zy)\iv \in \dw\vS'\dw\iv \cap \s\iv(a) \}.
\end{align*}
Then $V^a$ is isomorphic to the quotient of $\wt V^a$ by $E^{\th}$, where
$E = ({}^wL' \cap U_P) \times (L \cap {}^wU_{P'}) \times (L \cap {}^wL')$
and the action of $E^{\th}$ is given by the same formula as in the proof 
of Proposition 2.2.
We now consider the map 
$\wt V^a/E^{\th} \to E/E^{\th}$ induced by the projection 
$(u, u', y,y',z)$ $\mapsto (y,y',z)$.  Its image is given by 
\begin{align*}
\ol V^a = \{ (y,&y',z) \in E \mid (zy')\th(zy')\iv \in \vS \cap \s\iv(a),  \\ 
                  &(zy)\th(zy)\iv \in \dw\vS'\dw\iv \cap \s\iv(a) \}/E^{\th}.  
\end{align*}  
\par
By a similar computation as in (2.2.4), we see that all the fibres of this map are isomorphic to 
$(U_P \cap {}^wU_{P'})^{\io\th}$. 
Moreover, by a similar discussion as in the proof of (2.2.8), one can show that 
\begin{equation*}
\dim (U_P \cap {}^wU_{P'})^{\io\th} = 2\nu_H - \nu_{L_H} - \nu_{L'_H} 
  - \dim \OO_{\w} + b'_w,
\end{equation*}
where 
\begin{equation*}
b'_w = \sharp\{ i \mid 1 \le i \le n - n_0, 1 \le w\iv (i) \le n - n_0'\}. 
\end{equation*}
\par\medskip\noindent
(Here we assume that $L^{\th} \simeq (GL_1)^{(N - N_0)/2} \times GL^{\th}_{N_0}$ 
and ${L'}^{\th} \simeq (GL_1)^{(N - N_0')/2} \times GL_{N_0'}^{\th}$, and 
put $n_0 = [N_0/2], n_0' = [N_0'/2]$. If $B' = B$, a similar argument as in the proof 
of Proposition 2.2 can be applied. 
In the case where $N$ is even, we need to consider $B' = B_1$ also.
But in this case, if we replace $W_H = N_H(T_H)/T_H$ by a group 
$N_{G^{\th}}(T_H)/T_H$, which is isomorphic to $S_n\ltimes (\BZ/2\BZ)^n$, 
a similar argument as in the odd $N$ case works.)  
\par
Thus we are reduced to showing that 
\begin{equation*}
\tag{4.6.1}
H^s_c(\ol V^a, \ol j^*(\SE_1\boxtimes\SE'_1)) = 0,
\end{equation*}
where $s = \dim \SO_L + \dim \SO_{L'} + \vD_P + \vD_{P'} - 2b'_w$, 
and $\ol j : \ol V^a \to \vS \times \vS'$
is defined by $(y,y',z) \mapsto ((zy)\th(zy)\iv, \dw\iv(zy')\th(zy')\iv\dw)$.
\par
We use the notation $Q = L \cap {}^wP'$ the $\th$-stable parabolic subgroup of 
$L$ with Levi decomposition 
$Q = MU_Q$, where $M = L \cap {}^wL', U_Q = L \cap {}^wU_{P'}$.  Similarly, 
we define $Q' = {}^wL' \cap P$ the $\th$-stable parabolic subgroup of ${}^wL'$
with Levi decomposition $Q' = MU_{Q'}$, where $U_{Q'} = {}^wL' \cap U_P$.   
Let $\pi_3 : \ol V^a \to M ^{\io\th}$ be the map defined by 
$(y,y',z) \mapsto z\th(z)\iv$.   
Since the semisimple part of $z\th(z)\iv$ is contained in a fixed $H$-orbit, 
$\pi_3(\ol V^a)$ consists of finitely many $(M^{\th})^0$-orbits $\wh \SO_1, \dots, \wh\SO_m$ 
in $M^{\io\th}$. 
Since $\dim \ol V^a \le s/2$, it is enough to show, for any $i$,  that 
$H^{s-2\dim \wh\SO_i}_c(\pi_3\iv(\xi), \ol j^*(\SE_1\boxtimes \SE'_1)) = 0$ 
for $\xi \in \wh\SO_i$. 
Here $\pi_3\iv(\xi)$ is isomorphic to $D \times D'$, where 
\begin{align*}
D  &= \{ q \in Q \cap \vS  \mid  \eta_Q(q) = \xi\}, \\
D' &= \{ q' \in Q' \cap \dw\vS'\dw\iv \mid \eta_{Q'}(q') = \xi \}.
\end{align*}
\par\noindent
Moreover, the restriction of $\ol j^*(\SE_1\boxtimes\SE'_1)$ on $\pi_3\iv(\xi)$ 
corresponds to the tensor product $\SE_1|_D \boxtimes n^*\SE'_1|_{D'}$ 
(here $n : \dw \vS'\dw\iv \isom \vS'$). 
Here by Proposition 2.2 (i), $2\dim D \le d = \dim \SO_L - \dim \wh\SO_i + \vD_Q$, 
$2\dim D' \le d' = \dim \SO_{L'} - \dim \wh \SO_i + \vD_{Q'}$. 
In a similar way as in (2.2.12), we have
\begin{align*}
\vD _Q + b_w' \le \vD_{P'}, \quad 
\vD_{Q'} + b_w' \le \vD_P,
\end{align*}
since one can write as 

\begin{align*}
\vD_Q &= \sharp\{ i \mid n - n_0 + 1 \le i \le n, 1 \le  w\iv(i) \le n - n_0' \}, \\ 
\vD_{Q'} &= \sharp\{ i \mid n - n_0' + 1 \le i \le n, 1 \le w(i) \le n-n_0 \}. 
\end{align*}

It follows that 
\begin{align*}
d + d' = s - 2\dim \wh\SO_i + \vD_Q  + \vD_{Q'} - \vD_P - \vD_{P'} 
                  + 2b'_w \le s - 2\dim \wh\SO_i.
\end{align*}
Hence by the K\"unneth formula we are reduced to showing
\begin{equation*}
\tag{4.6.2}
H_c^d(D,\SE_1)\otimes H_c^{d'}(D', n^*\SE'_1) = 0.  
\end{equation*} 

Unless $L = {}^wL'$, 
$Q = L \cap {}^wP'$ is a proper parabolic subgroup of $L$, or 
$Q' = {}^wL' \cap P$ is a proper parabolic subgroup of ${}^wL'$.
In that case, since the local systems $\SE_1$, $\SE_1'$ are cuspidal, we have 
$H^{d}_c(D, \SE_1) = 0$ or $H^{d'}_c(D', n^*\SE'_1) = 0$  
by Lemma 2.5.  Thus (4.6.2), and so (4.6.1) holds. 
Now assume that $B' = B$.  Then by the choice of $P$ and $P'$, 
we have $L = L'$ and $\dw \in N_H(L_H)$.  If $\vS  \cap \vS' = \emptyset$, 
then $V^a = \emptyset$, so (4.6.1) holds. This proves (i). 
Assume that $\dw \in N_H(L_H)$, but $\dw \notin N_H(L_H)_1$. 
Then $Q = L$ and $\vD_Q = 0$.  Moreover $\vD_P > b'_w$.  Hence $d + d' < s - 2\dim \wh\SO_i$.
This implies that $H_c^{s - 2\dim \wh\SO_i}(\pi_3\iv(\xi),\ol j^*(\SE_1\boxtimes\SE'_1)) = 0$,
and so (4.6.1) holds. This proves (ii).  
Next assume that $N$ is even and $B' = B_1$.  If $L \ne T$, we have 
$L' = L$ (see 4.2), and $\dw \in N_H(L_H)$.  The same 
argument as before works, and (i), (ii) are proved. 
So assume that $L = L' = T$, and $B' = B_1$. 
This case is essentially the same as the case where $B' =  B$, and 
$\dw \in N_{G^{\th}}(T_H)$.  
We see that if $\dw \in N_{G^{\th}}(T_H) - N_H(T_H)_1$, 
then $b'_w  = \sharp\{ 1 \le i \le n \mid w\iv(i) > 0\} < \vD_B = n$.  
Hence $d + d' < s$, and (4.6.2) holds.  This proves (iii). 
The lemma is proved. 
\end{proof}

\para{4.7.}
Following [L1, 5.4], we recall the notion of perfect sheaves. 
A constructible sheaf $\SE$ on an irreducible variety $V$ is said to be perfect, 
if 
\par\medskip
(i) \ $\SE = \IC(V, \SE|_{V_0})$, where $V_0$ is an open dense smooth subset of $V$
and $\SE|_{V_0}$ is locally constant, 
\par
(ii) \ the support of any non-zero constructible subsheaf of $\SE$ is dense in $V$.
\par\medskip
In particular, the complex $\IC(V, \SE|_{V_0})$ is reduced to a single sheaf.  
The following properties hold. 
\par\medskip\noindent
(4.7.1) \ 
If $\pi : V' \to V$ is a finite morphism, with $V'$ smooth, and if $\SE'$ is a locally 
constant sheaf on $V'$, then $\SE = \pi_*\SE'$ is a perfect sheaf on $V$.
\par\medskip\noindent
(4.7.2) \ If $0 \to \SE_1 \to \SE_2 \to \SE_3 \to 0$ is an exact sequence of 
constructible sheaves on $V$, with $\SE_1, \SE_3$ perfect, then $\SE_2$ is perfect.

\para{4.8.}
We assume that $(L,\vS) = (L',\vS')$.  We also assume 
$\OO = \OO_{\w}$, where a representative of $\w$ satisfies the condition 
$n = \dw \in N_H(L_H)$ with $w \in \SW_1$.  
In this case, we have $n\vS n\iv = \vS$. 
The computation in the proof of the lemma shows that 
\begin{equation*}
\tag{4.8.1}
\ST_{\OO} \simeq R^s\ol\s_!(\SE_1\otimes n^*\SE'_1)
\end{equation*}
where $\ol\s : \vS \to A_Y$ is the restriction of $\s$, 
$s = 2\dim \SO_L$ (note that $\vD_P = b'_w = b_w$). 
Here $\vS = Z_L^{\io\th} \times \SO_L$. We denote by 
$\pi_1 : \vS \to Z_L^{\io\th}$ the projection, and 
$\ol\s_1 : Z_L^{\io\th} \to A_Y$ the restriction of $\s$. 
Since $\ol\s_1$ is a finite morphism, if we put 
$\SG = R^s(\pi_1)_!(\SE_1\otimes n^*\SE'_1)$, we have
\begin{equation*}
\SF_{\OO} = (\ol\s_1)_*\SG.
\end{equation*}

We note that 
\par\medskip\noindent
(4.8.2) \ $\SG$ is a locally constant sheaf on $Z_L^{\io\th}$. 
\par\medskip
In fact, $Z_L^{\io\th}$ has a structure of a torus. Since 
$\SE_1 = \Ql \boxtimes \SE_1^{\dag}$, $\SE_1$ is $Z_L^{\io\th}$-equivariant
local system on $\vS$ with respect to the action of $Z_L^{\io\th}$ on $\vS$ 
by the left multiplication. Hence $\SE_1\otimes n^*\SE'_1$ is also 
$Z_L^{\io\th}$-equivariant.  Since $\pi_1$ is $Z_L^{\io\th}$-equivariant, 
$\SG$ is $Z_L^{\io\th}$-equivariant with respect to the transitive action of 
$Z_L^{\io\th}$ on it.  Hence $\SG$ is locally constant.  
\par\medskip
Now $\ol\s_1 : Z_L^{\io\th} \to A_Y$ is a finite morphism, with  
$Z_L^{\io\th}$ smooth, we have, by (4.7.1), 
\par\medskip\noindent
(4.8.3) \ $\SF_{\OO}$ is a perfect sheaf on $A_Y$.   
\par\medskip

We now show 

\begin{thm}  
\begin{enumerate}
\item Assume that $(L,\vS)$ and $(L',\vS')$ are not $H$-conjugate.  
Then $\ST = 0$. 
\item 
Assume that $N$ is even, and $P = B, P' = B_1, L = L' = T$.  Then $\SF = 0$. 
\item
Assume that $P = P', L = L', \vS = \vS'$. Then 
$\SF$ is a perfect sheaf on $A_Y$.
\item 
The natural map of sheaves $\SF \to \wt\SF$ is an isomorphism. 
\end{enumerate}
\end{thm}

\begin{proof}
Let $E$ be a locally closed subvariety of $H/P_H \times H/P'_H$ which is a
union of some $H$-orbits, and put $Z^E_0 = q\iv(E)$ ($q$ is as in 4.5).  One can define 
a constructible sheaf $\SF_E$ in a similar way as in the definition of $\SF_{\OO}$. 
According to the filtration of $H/P_H \times H/P'_H$ by various $E$, we obtain 
a filtration of $Z_0$ by various $Z^E_0$, which yields an exact sequence 
among various $\SF_E$. In the case of (i), since $\SF_{\OO} = 0$ by Lemma 4.6 (i), 
we have $\ST_E = 0$ for any $E$, and so $\ST = 0$. (ii) also follows from Lemma 4.6 (iii). 
On the other hand, since $\SF_{\OO}$ is perfect by (4.8.3), by using the property 
on exact sequences in (4.7.2), we see that $\SF$ is perfect. This proves (iii). 
\par
We show (iv). By (4.5.1) we know that the map $\SF \to \wt\SF$ is surjective. 
Hence it is enough to show that the kernel of $\SF \to \wt\SF$ is zero. Put 
$(A_Y)\reg = \s(\vS\reg) = \ol\s_1((Z_L)^{\io\th}\reg)$.
Then $(A_Y)\reg$ is an open dense subset of $A_Y$. Since $\SF$ is a perfect sheaf by (i), 
by the property (ii) of perfect sheaves, it is enough to show that the stalk of 
this kernel at any point $a \in (A_Y)\reg$ is zero.   
So, by Lemma 4.4 (ii), we have only to show that $Z^a_{\b,\b} = \emptyset$ 
if $a \in (A_Y)\reg$ and if $\b \ne \b_0$ or $\b' \ne \b_0$. 
Take $(x, gP_H, g'P'_H) \in Z$ such that the semisimple part $x_s$ of $x$ is $H$-conjugate
to an element in $(Z_L)^{\io\th}\reg$ and that $g\iv xg \in \eta_P\iv(\ol\vS)$, 
${g'}\iv xg' \in \eta_{P'}\iv(\ol\vS)$.  Then $Z_H(x_s)^0 \simeq L_H$.
We must have $g\iv xg \in \eta_P\iv(\vS)$ since if
$g\iv xg \notin \eta_P\iv(\vS)$, then $\dim Z_H(g\iv x_sg)> \dim L_H$, a contradiction.
Similarly, we have ${g'}\iv xg' \in \eta_{P'}\iv(\vS)$.  This shows that 
$(x, gP_H, g'P'_H) \in Z^a_{\b_0.\b_0}$.  (iv) is proved.    
\end{proof}

\begin{prop} 
Assume that $P = P', L = L', \vS = \vS'$. Then  
there exists an isomorphism of sheaves on $A_Y$; 
\begin{equation*}
\SF \simeq \bigoplus_{w \in \SW_1}\SF_{\OO(w)},
\end{equation*} 
where $\OO(w)$ is an $H$-orbit corresponding to $w \in \SW_1 \subset \SW = N_H(L_H)/L_H$. 
\end{prop}

\begin{proof}
Recall that $\SF = R^{2d_0}(\wt \s_0)_!(\ol\SE_1\boxtimes \ol\SE'_1)$. 
Here $\wt\s_0 : Z_0 \to A_Y$, where $Z_0 \simeq X_0\times_{X}X_0$. 
By a similar argument as in the proof of [L1, Proposition 5.11], we see that 
$\wt\s_0\iv((A_Y)\reg) \simeq \wt Y \times_{Y}\wt Y$. Thus the restriction of
$\SF$ on $(A_Y)\reg$ is the same as $R^{2d_0}(\wt \s_1)_!(\wt\SE_1 \boxtimes\wt\SE_1)$, 
where $\wt\s_1 : \wt Y \times_Y\wt Y \to (A_Y)\reg$ is defined similarly to $\wt\s_0$. 
Consider the partition $Z_0 = \bigcup_{\OO}Z_0^{\OO}$ as before. 
The pieces $Z_0^{\OO}$ are locally closed subsets in $Z_0$. 
But since $\wt Y$ is a principal $\SW_1$-bundle over $Y$, the intersection  
$Z_0^{\OO(w)} \cap (\wt\s_0)\iv((A_Y)\reg)$ is an open and closed subset. 
Since $\SF_{\OO(w)} = 0$ if $w \notin \SW_1$ by Lemma 4.6 (ii),  
we have a natural isomorphism 
$\SF \simeq \bigoplus_{w \in \SW_1}\SF_{\OO(w)}$ over $(A_Y)\reg$. Since $\SF$ and 
$\bigoplus_{w \in \SW_1}\SF_{\OO(w)}$ are both perfect sheaves, this isomorphism 
can be extended uniquely to the isomorphism over $A_Y$. The proposition is proved.   
\end{proof}

\para{4.11.}
Under the assumption of Proposition 4.10, 
we consider the sheaf $\wt\SF = R^{-2r}\wt\s_!(K_{\ol\SE_1}\boxtimes K_{\ol\SE'_1})$. 
Since $\wt\s = \s\circ p_1$ with $p_1 : Z = \wt X \times_X\wt X \to X$ the natural 
projection, one can write as
\begin{equation*}
\tag{4.11.1}
\wt\ST \simeq R^{-2r}\s_!(\pi_*K_{\ol\SE_1}\otimes\pi_*K_{\ol\SE'_1}).
\end{equation*}
By Proposition 3.8, $\pi_*K_{\ol\SE_1}\otimes \pi_*K_{\ol\SE'_1}$ has a natural structure
of a module over the algebra $\SA_{\SE_1}\otimes \SA_{\SE'_1}$.  
Hence $\wt\ST$ inherits a natural action of $\SA_{\SE_1}\otimes \SA_{\SE'_1}$. 
As in [L1, 3.4], $\SA_{\SE_1}$ has a natural decomposition 
$\SA_{\SE_1} = \bigoplus_{w \in \SW_{\SE_1}}\SA_{\SE_1,w}$ such that 
$\SA_{\SE_1,w}\cdot \SA_{\SE_1,w'} = \SA_{\SE_1, ww'}$. 
Then the action of $\SA_{\SE_1}\otimes \SA_{\SE'_1}$ on 
$\wt\ST \simeq \ST$ satisfies the relation (see [L1, 5.12])
\begin{equation*}
\tag{4.11.2}
(\SA_{\SE_1,w_1}\otimes\SA_{\SE_1', w_1'})\cdot \ST_{\OO(w)} = 
               \ST_{\OO(w_1w{w_1'}\iv)}
\end{equation*}
for $w \in \SW_1, w_1 \in \SW_{\SE_1}, w_1' \in \SW_{\SE_1'}$. 
\par
We consider the stalk $\ST_1$ of $\ST$ at $1 \in A_Y$. 
For $w \in \SW_1$, 
the stalk $\ST_{\OO(w),1}$ of $\ST_{\OO(w)}$ at $1 \in A_Y$
is given, by (4.8.1), 
\begin{equation*}
\ST_{\OO(w),1} \simeq H_c^{2\dim\SO_L}(\SO_L, \SE_1\otimes n^*\SE_1').
\end{equation*}
Hence $\ST_{\OO(w),1}$ is a one-dimensional $\Ql$-vector space if 
$w$ satisfies the relation $n^*\SE_1' \simeq \SE^*_1$, the dual local system 
of $\SE_1$, and $\ST_{\OO(w),1} = 0$ otherwise. 
Now assume that $\SE_1' = \SE_1^*$.  Then $n^*\SE'_1 \simeq \SE_1^*$ 
if and only if $w \in \SW_{\SE_1}$.  Thus we have a decomposition  
\begin{equation*}
\wt\ST_1 = \ST_1 = \bigoplus_{w \in \SW_{\SE_1}}\ST_{\OO(w),1}
\end{equation*}
into one-dimensional vector spaces, and (4.11.2) implies that 
\begin{equation*}
(\SA_{\SE_1,w_1}\otimes\SA_{\SE^*_1, w_1'})\cdot \ST_{\OO(w),1}
                 = \ST_{\OO(w_1w{w_1'}\iv)}
\end{equation*}
for $w, w_1, w_1' \in \SW_{\SE_1}$. This is nothing but the two-sided regular 
representation of $\SA_{\SE_1}$ if we identify $\SA_{\SE^*_1}$ with 
$\SA_{\SE_1}^0$ the opposed algebra of $\SA_{\SE_1}$. Thus we have proved

\begin{prop}  
\begin{enumerate}
\item 
$\wt\ST_1 = 0$ if $\SE'_1$ and $\SE^*_1$ are not conjugate under $N_H(L_H)$.  
\item 
Assume that $\SE_1' = \SE_1^*$.  Then the stalk $\wt\ST_1$ of $\wt\ST$ at $1 \in A_Y$ 
is isomorphic to the $\SA_{\SE_1}\otimes \SA_{\SE_1}^0$-module $\SA_{\SE_1}$ 
(two-sided regular representation of $\SA_{\SE_1}$).  
\end{enumerate}
\end{prop}

\par\bigskip
\section{ Generalized Springer correspondence}

\para{5.1.}
In this section, we shall establish the generalized Springer correspondence 
for $G^{\io\th}\uni$ following the discussion in [L1, \S 6]. 
Let $\SN_G$ be the set of pairs $(\SO, \SE)$, where $\SO$ is an $H$-orbit in 
$G^{\io\th}\uni$, and $\SE$ is an $H$-equivariant simple local system on $\SO$.  
We consider a triple $(L \subset P, \SO_L, \SE_1^{\dag})$, 
where $L$ is a $\th$-stable Levi subgroup of a $\th$-stable 
parabolic subgroup $P$ of $G$, and $\SO_L$ is an $L_H$-orbit in $L^{\io\th}\uni$, 
$\SE_1^{\dag}$ is an $L_H$-equivariant simple local system on $\SO_L$. 
Put $\vS = Z_L^{\io\th}\SO_L$, and $\SE_1 = \Ql\boxtimes \SE_1^{\dag}$. 
Let $\ScS_G$ be the set of triples $(L \subset P, \SO_L, \SE_1^{\dag})$, 
up to $H$-conjugate, 
such that $\SE_1$ is cuspidal on $\vS$. 
(Note that in the case where $N$ is odd, $P$ is determined by L, up to $H$-conjugate, 
hence we need not to write $P$ in the notation of $\ScS_G$ (see [L1, \S 6]).  
But in the case where $N$ is even and $L= T$, we need to distinguish 
the two cases $T \subset B$ and $T \subset B_1$).  
\par
For each $(L \subset P, \SO_L, \SE_1^{\dag}) \in \ScS_G$, 
one can construct a semisimple perverse sheaf 
$K = \pi_*K_{\ol\SE_1}$ on $X$ = $\ol Y$ with $Y = Y_{(L,\vS)}$.   
By Proposition 3.8, $K$ is equipped with an action of $\SA_{\SE_1}$, and is decomposed 
into simple perverse sheaves as follows;
\begin{equation*}
K \simeq \bigoplus_{\r \in \SA_{\SE_1}\wg}\r\otimes K_{\r}, 
\end{equation*} 
where $K_{\r} = \Hom (\r, K)$. 
We consider the restriction $K|_{G^{\io\th}\uni}$ of $K$ on $G^{\io\th}\uni$. 
Then $K|_{G^{\io\th}\uni}$ inherits the natural action of $\SA_{\SE_1}$.
The following result gives the generalized Springer 
correspondence for $G^{\io\th}\uni$, which is an analogue of 
[L1, Theorem 6.5]

\begin{thm} [generalized Springer correspondence]  
Under the notation above, 
\begin{enumerate}
\item 
Let $K = \pi_*K_{\ol\SE_1}$ for $(L \subset P, \SO_L, \SE_1^{\dag}) \in \ScS_G$, and 
put $r = \dim Z_L^{\io\th}$. 
Then $K[-r]|_{X\uni}$ is a semisimple perverse sheaf on 
$X\uni = X \cap G^{\io\th}\uni$, and is decomposed into simple perverse sheaves,
\begin{equation*}
\tag{5.2.1}
K[-r]|_{X\uni} \simeq \bigoplus_{(\SO, \SE) \in \SN_G}V_{(\SO,\SE)}\otimes
                         \IC(\ol\SO, \SE)[\dim \SO].
\end{equation*} 
The action of $\SA_{\SE_1}$ on $K[-r]|_{G^{\io\th}\uni}$ induces an action of 
$\SA_{\SE_1}$ on $V_{(\SO,\SE)}$, which makes $V_{(\SO,\SE)}$ an irreducible 
$\SA_{\SE_1}$ module if it is non-zero. 
The map $(\SO, \SE) \mapsto V_{(\SO,\SE)}$ gives a bijection 

\begin{equation*}
\{ (\SO,\SE) \in \SN_G \mid V_{(\SO,\SE)} \ne 0\} \simeq \SA_{\SE_1}\wg.
\end{equation*}
Moreover, under the above correspondence $(\SO,\SE) \lra \r$, we have

\begin{equation*}
\tag{5.2.2}
K_{\r}[-\dim X]|_{X\uni} \simeq \IC(\ol\SO, \SE)[-2d_{\SO}],  
\end{equation*}
where we put $d_{\SO} = (\nu_H - \dim \SO/2) - (\nu_{L_H} - \dim \SO_L/2) + \vD_P/2$.
\item
For each $(\SO,\SE) \in \SN_G$, there exists a unique 
$(L \subset P,\SO_L, \SE^{\dag}_1) \in \ScS_G$ 
such that $\IC(\ol\SO,\SE)[\dim \SO]$ is a direct summand of $\pi_*K_{\ol\SE_1}$. 
The correspondence $(\SO, \SE) \mapsto V_{(\SO,\SE)}$ gives a bijection 
\begin{equation*}
\SN_G \simeq \coprod_{(L \subset P,\SO_L, \SE_1^{\dag}) \in \ScS_G}\SA_{\SE_1}\wg.
\end{equation*} 

\item
For $(L \subset P,\SO_L,\SE_1^{\dag}) \in \ScS_G$, let $f : \wt X_0 \to X$ be the 
restriction of $\pi : \wt X \to X$.
Under the correspondence in (ii), the condition $(\SO,\SE)$ corresponds 
to $(L \subset P, \SO_L, \SE_1^{\dag})$ is that $\ol\SO \subset X$ and that 
$\SE$ appears as a direct summand in the local system $R^{2d_{\SO}}f_!\ol\SE_1|_{\SO}$. 
Moreover, in that case, the natural homomorphism, obtained from the embedding 
$\wt X_0 \subset \wt X$, 
\begin{equation*}
R^{2d_{\SO}}f_!\ol\SE_1|_{\SO} \to \SH^{2d_{\SO}}(\pi_*K_{\ol\SE_1}[-\dim X])|_{\SO} 
\end{equation*}
is an isomorphism. 
\end{enumerate}
\end{thm}

\para{5.3.}
The remainder of this section is devoted to the proof of the theorem. 
Let $P = LU_P$ be a $\th$-stable parabolic subgroup of $G$, with 
$\th$-stable Levi subgroup $L$. Let $\SO$ be an $H$-orbit in $G^{\io\th}\uni$, 
$\SO_L$ an $L_H$-orbit in $L^{\io\th}\uni$.  We consider the following diagram.

\begin{equation*}
\tag{5.3.1}
\begin{CD}
V = H \times^{P_H}(\SO \cap \eta_P\iv(\SO_L)) @>f_1>> \SO  \\
@Vf_2VV   \\
V' = H \times^{P_H}\SO_L.
\end{CD}
\end{equation*}
\par\medskip
\noindent
Here in $V'$, we consider the action of $P_H$ on $\SO_L$ 
such that $U_{P_H}$ acts trivially.
$f_1$ is the map induced from the map $(g,x) \mapsto gxg\iv$, 
and $f_2$ is the map induced from the map 
$\SO \cap \eta_P\iv(\SO_L) \to \SO_L$ which is the restriction of 
the projection $\eta_P : P^{\io\th} \to L^{\io\th}$. 
$H$ acts on $V, V', \SO$ naturally, and $f_1, f_2$ are $H$-equivariant. 
Moreover, the action of $H$ on $V'$ and on $\SO$ are transitive.  
Put 
\begin{align*}
d_1 &= (\nu_H - \dim \SO/2) - (\nu_{L_H} - \dim \SO_L/2) + \vD_P/2,  \\
d_2 &= (\dim \SO - \dim \SO_L) /2 + \vD_P/2.  
\end{align*}
By Proposition 2.2 (ii), (i), all the fibres of $f_1$ have dimension $\le d_1$ and 
all the fibres of $f_2$ have dimension $\le d_2$. 
Moreover, some (or all) fibre of $f_1$ has dimension $d_1$ if and only if some (or all) 
fibre of $f_2$ has dimension $d_2$ since this is equivalent to the condition that 
$\dim V = (\nu_H - \nu_{L_H}) + (\dim \SO  + \dim \SO_L)/2 + \vD_P/2$. 
Under this situation, the following result was proved by [L1, (6.1.1)]. 

\begin{lem}  
Let $\SF$ be an $H$-equivariant simple local system on $\SO$, and 
$\SF'$ be an $H$-equivariant simple local system on $V'$ (note that $V'$ is a single 
$H$-orbit).  Then the multiplicity of $\SF$ in the $H$-equivariant local system 
$R^{2d_1}(f_1)_!(f_2^*\SF')$ on $\SO$ is equal to the multiplicity of $\SF'$ in the
$H$-equivariant local system $R^{2d_2}(f_2)_!(f_1^*\SF)$ on $V'$. 
\end{lem}

\para{5.5.}  If $(\SO, \SE) \in \SN_G$ is a cuspidal pair, 
by Lemma 2.5, for any $\th$-stable parabolic subgroup $P \ne G$, and any 
$L_H$-orbit $\SO_L$ of $L^{\io\th}\uni$, we have 
\begin{equation*}
H_c^{2\d}(\SO \cap \eta_P\iv(v), \SE) = 0,
\end{equation*}
where $\d = (\dim \SO - \dim \SO_L)/2 + \vD_P/2$ and $v \in \SO_L$.
Now let $(\SO,\SE) \in \SN_G$ be an arbitrary element. One can find a 
$\th$-stable parabolic subgroup $P$ of $G$ with the $\th$-stable Levi subgroup $L$ 
and an $L_H$-orbit $\SO_L$ in $L^{\io\th}\uni$ 
satisfying the property
\begin{equation*}
\tag{5.5.1}
H_c^{2\d}(\SO \cap \eta_P\iv(v),\SE) \ne 0
\end{equation*} 
for $v \in \SO_L$.
We choose $P$ a $\th$-stable minimal parabolic subgroup of $G$ satisfying 
the property (5.5.1). Note that $P = G$ satisfies the condition (5.5.1).   
\par
Let $f_2 : \SO \cap \eta_P\iv(\SO_L) \to \SO_L$ be the map defined by 
$x \mapsto \eta_P(x)$, and  let $\SE^{\dag}_1$ be a 
$L_H$-equivariant simple local system on $\SO_L$ which is a direct summand of 
$R^{2\d}(f_2)_!(\SE)$. 

\begin{prop}  
The triple $(L \subset P, \SO_L, \SE^{\dag}_1)$ is uniquely determined by $(\SO,\SE)$. 
Moreover, $(\SO_L, \SE^{\dag}_1)$ is a cuspidal pair on $L^{\io\th}$, namely 
$(L \subset P,\SO_L, \SE_1^{\dag}) \in \ScS_G$.  
\end{prop}

\begin{proof}
Suppose that $(\SO_L, \SE_1^{\dag})$ is not cuspidal. Then there exists 
a $\th$-stable parabolic subgroup $P' \subsetneq P$ and a Levi subgroup $L'$ of 
$P'$ such that $L' \subset L$, and an $L'_H$-orbit $\SO_{L'}$ in ${L'}^{\io\th}\uni$
satisfying the property 

\begin{equation*}
H_c^{2\d'}(\SO_L \cap \eta_{L \cap P'}\iv(v'), \SE^{\dag}_1) \ne 0
\end{equation*} 
\par\medskip\noindent
for $v' \in \SO_{L'}$, where $\d' = (\dim \SO_L - \dim \SO_{L'})/2 + \vD_{L \cap P'}$. 
It follows that 

\begin{equation*}
H_c^{2\d'}(\SO_{L} \cap \eta_{L \cap P'}\iv(v'), R^{2\d}(f_2)_!(\SE)) \ne 0. 
\end{equation*}
The map $f_2$ defines a map 
$\SO \cap \eta_{P'}\iv(v') \mapsto \SO_L \cap \eta_{L \cap P'}\iv(v')$ by the restriction,
and all the fibres have dimension $\le \d$ by Proposition 2.2.  By the Leray 
spectral sequence, we have

\begin{equation*}
H_c^{2\d' + 2\d}(\SO \cap \eta_{P'}\iv(v'), \SE) \ne 0.
\end{equation*} 
Since $\vD_P + \vD_{L \cap P'} = \vD_{P'}$, we have 
$\d' + \d = (\dim \SO - \dim \SO_{L'})/2 + \vD_{P'}$.  
This contradicts the minimality of $P$.  Hence $\SE_1^{\dag}$ is cuspidal on $\SO_L$. 
\par
Let $\vS = Z_L^{\io\th}\SO_L$, and put $\SE_1 = \Ql\boxtimes \SE_1^{\dag}$.
Then $(\vS, \SE_1)$ is a cuspidal pair in $L$. 
By Lemma 2.7, $L_H$ is of the form $L_H \simeq (GL_1)^a \times SO_{N_0}$. 
Hence the discussion in Section 3 can be applied.
We consider $Y = Y_{(L,\vS)}$ and the map $\pi : \wt X \to X = \ol Y$.
Let $\wt X_0 \simeq H \times^{P_H}\eta_P\iv(\vS)$ be as in 3.7.  
The local system  $\ol\SE_1$ on $\wt X_0$ is constructed from $\SE_1$. 
Let $f : \wt X_0 \to X$ be the restriction of the map $\pi : \wt X \to X$.
Then the restriction 
of $f$ on $f\iv(\SO)$ coincides with $f_1$ in 5.3. 
By the definition of $(L \subset P, \SO_L, \SE^{\dag}_1)$ and by Lemma 5.4, we see that 
$\SE$ is a direct summand of $R^{2d}f_!(\ol\SE_1)|_{\SO}$, where 
$d = (\nu_H - \dim \SO/2) - (\nu_{L_H} - \dim \SO_L/2) + \vD_P/2$ 
(note that $d_1 = \d, d_2 = d$ in the notation in Lemma 5.4).
\par
We show that such a triple $(L \subset P, \SO_L, \SE^{\dag}_1) \in \ScS_G$ is uniquely determined from 
$(\SO,\SE)$. Suppose that there exists another $\th$-stable minimal 
parabolic subgroup $P'$ satisfying (5.5.1), and 
$(L'\subset P', \SO_{L'}, {\SE_1'}^{\dag}) \in \ScS_G$.
We can define a map $f' : \wt X_0' \to X'$ and a local system $\ol\SE_1'$ on $\wt X_0'$
as before.  We see that $\SE$ appears as a direct summand in 
$R^{2d'}f'_!(\ol\SE'_1)|_{\SO}$.  It follows that 
$R^{2d}f_!(\ol\SE_1)\otimes R^{2d'}f'_1({\ol\SE'}^*_1)|_{\SO}$ contains a constant sheaf 
$\Ql$ on $\SO$, where ${\ol\SE'}^*$ is the dual local system of $\ol\SE_1'$.  
Hence 
\begin{equation*}
\tag{5.6.1}
H_c^{2\dim \SO}(\SO, R^{2d}f_!(\ol\SE_1)\otimes R^{2d'}f'_!({\ol\SE'}^*_1)|_{\SO}) \ne 0.
\end{equation*}

Let $p : Z_0 = \wt X_0 \times_{G^{\io\th}} \wt X_0' \to G^{\io\th}$ be the natural projection, 
and put $Z_{0,\SO} = p\iv (\SO)$.   
Since for each $x \in \SO$, $\dim p\iv(x) \le d + d'$, 
(5.6.1) implies, by the Leray spectral sequence, that 

\begin{equation*}
\tag{5.6.2}
H_c^{2d_0}(Z_{0,\SO}, \ol\SE_1\boxtimes {\ol\SE'}^*_1) \ne 0,
\end{equation*}  
where $d_0 = d + d' + \dim \SO$ is as in (4.3.1). 
We consider 
$Z_0^1 = Z_{\b_0,\b'_0}^1$ for $a = 1 \in A_Y \cap A_{Y'}$ under the notation in 4.3. 
Then $Z_{0,\SO'}$ form a partition of $Z^1_0$ by locally closed pieces $Z_{0,\SO'}$ 
of dimension $\le d_0$ (by Lemma 4.4) if $\SO'$ varies the $H$-orbits in $G^{\io\th}\uni$.  
It follows from (5.6.2), that 
\begin{equation*}
H_c^{2d_0}(Z^1_0, \ol\SE_1\boxtimes {\ol\SE'}^*_1) \ne 0.
\end{equation*}
\par\medskip
\noindent
This means that $\ST_1 \ne 0$, where $\ST$ is defined with respect to 
$(L \subset P, \vS, \SE_1)$ and $(L'\subset P',\vS', {\SE'}^*_1)$. 
Then by Theorem 4.9, $P = P', L = L', \vS = \vS'$.
Moreover by Proposition 4.12, $\SE_1'$ is $N_H(L_H)$-conjugate to $\SE_1$.
Thus the triple $(L \subset P, \SO_L, \SE^{\dag}_1)$ is 
uniquely determined from $(\SO,\SE)$. 
\end{proof}

\para{5.7.}
By Proposition 5.6, for each $(\SO,\SE) \in \SN_G$, there exists 
a unique triple $(L \subset P, \SO_L, \SE_1^{\dag}) \in \ScS_G$. In this case, 
we say that $(\SO,\SE)$ belongs to the series $(L \subset P, \SO_L, \SE_1^{\dag})$.
Hence we have a partition
\begin{equation*}
\tag{5.7.1}
\SN_G = \coprod_{\xi \in \ScS_G}\SN_G^{(\xi)},
\end{equation*}
where $\SN_G^{(\xi)}$ is the set of all $(\SO,\SE)$ which belong to the series
$\xi = (L \subset P, \SO_L, \SE_1^{\dag})$. 
\par
For each $(L \subset P ,\SO_L, \SE_1^{\dag}) \in \ScS_G$, we consider 
$Y = Y_{(L,\vS)}$ and $X = \ol Y$.  Let $f : \wt X_0 \to X$ be the restriction 
of $\pi : \wt X \to X$. The following result gives a characterization of the 
set $\SN_G^{(\xi)}$.  

\begin{lem}  
$(\SO, \SE) \in \SN_G$ belong to $\xi = (L \subset P,\SO_L, \SE_1^{\dag})$ if and only if 
$\SO \subset X$ and $\SE$ is a direct summand of $R^{2d}f_!(\ol\SE_1)|_{\SO}$, 
where $d = (\nu_H - \dim\SO/2) - (\nu_{L_H} - \dim \SO_L/2) + \vD_P/2$.  
\end{lem}

\begin{proof}
In the course of the proof of Proposition 5.6, we have already shown that 
if $(\SO,\SE)$ belongs to $(L \subset P, \SO_L, \SE_1^{\dag})$, then $\SO \subset X$, 
and $\SE$ is a direct summand of $R^{2d}f_1(\ol\SE_1)|_{\SO}$.
Conversely, assume that $\SO \subset X$ and $\SE$ is a direct summand of 
$R^{2d}f_!(\ol\SE_1)|_{\SO}$. Then by Lemma 5.4, $\SE^{\dag}_1$ is a direct 
summand of $R^{2\d}(f_2)_!(\SE)$, where $\d$ and $f_2$ are defined as in 5.5.
The discussion in the proof of Proposition 5.6 shows that 
$(P, L, \SO_L)$ satisfies 
the condition in (5.5.1), and if such a choice of $(P,L, \SO_L)$ is not minimal, 
it contradicts that $(\SO_L, \SE_1^{\dag})$ is a cuspidal pair.  Hence 
$(\SO,\SE)$ belongs to $(L \subset P,\SO_L, \SE_1^{\dag})$.   
\end{proof}

\para{5.9.}
We fix $(L \subset P, \SO_L, \SE_1^{\dag}) \in \ScS_G$. 
Put $X\uni = X \cap G^{\io\th}\uni$, and $\wt X\uni = \pi\iv(X\uni)$. 
Thus
\begin{align*}
\wt X\uni &= \{ (x, gP_H) \in G^{\io\th} \times H/P_H \mid g\iv xg \in \e_P\iv(\ol\SO_L)\}, \\
    X\uni &= \bigcup_{g \in H}g(\eta_P\iv(\ol\SO_L))g.
\end{align*}
Let $\pi_1 : \wt X\uni \to X\uni$ be the restriction of $\pi$ on $\wt X\uni$. 
Since $\pi_1$ is proper, surjective, $X\uni$ is a closed subset of $G^{\io\th}\uni$. 

\begin{lem}  
$\wt X\uni, X\uni$ are irreducible varieties.  We have 
\begin{equation*}
\dim \wt X\uni = \dim X\uni = 2\nu_H - 2\nu_{L_H} + \dim \SO_L + \vD_P.
\end{equation*} 
\end{lem}

\begin{proof}
Since $\wt X\uni \simeq H \times^{P_H}\eta_P\iv(\ol\SO_L)$, $\wt X\uni$ is irreducible.
Also $\dim \wt X\uni = \dim H/P_H + \dim\eta_P\iv(\ol\SO_L) 
                   = 2\nu_H - 2\nu_{L_H} + \dim \SO_L + \vD_P$ (compare with Lemma 3.5). 
Thus $\dim \wt X\uni = d_0$.  
Since $\pi_1$ is surjective, $\dim \wt X\uni \ge \dim X\uni$, and $X\uni$ is irreducible. 
Suppose that $\dim X\uni = \d < \dim \wt X\uni$.  Then there exists an open dense subset 
$D$ of $X\uni$ such that $\dim \pi_1\iv(x) = \dim \wt X\uni - \d$ for $x \in D$. We consider 
the fibre product $\pi_1\iv(D) \times_D\pi_1\iv(D)$. It has the dimension 
$\dim D + 2(\dim \wt X\uni -\d) = 2\dim \wt X\uni - \d$.  On the other hand, 
$\pi_1\iv(D )\times_{D}\pi_1\iv(D) \subset \wt X\uni\times_{X\uni}\wt X\uni \simeq Z'$, 
where $Z'$ is as in 2.1. 
By Proposition 2.2 (iv), we have 
$\dim Z' \le d_0 = \dim \wt X\uni$.  
It follows that $\dim \wt X\uni \ge 2\dim \wt X\uni - \d$, and so $\d \ge \dim \wt X\uni$.
This is a contradiction, and the lemma follows. 
\end{proof}

\para{5.11.}
We now consider the restriction of $\pi_*K_{\ol\SE_1}$ on $X\uni$. 
Since $\ol\vS \simeq Z_L^{\io\th} \times \ol\SO_L$, 
$\IC(\ol\vS, \SE_1) \simeq \Ql \boxtimes \IC(\ol\SO_L, \SE_1^{\dag})$. 
It follows that the restriction of $K_{\ol\SE_1} = \IC(\wt X, \ol\SE_1)[\dim \wt X]$ 
on $\wt X\uni$ coincides with $K^1_{\ol\SE_1}[r]$, where 
$K^1_{\ol\SE_1} = \IC(\wt X\uni, \ol\SE_1^{\dag})[\dim \wt X\uni]$. Here  
$\ol\SE_1^{\dag}$ is defined from $\SE_1^{\dag}$ in a similar way as $\ol\SE_1$ 
is defined from $\SE_1$.
(Note that $\dim \wt X - \dim \wt X\uni = r$ by Lemma 5.10 and (3.2.1).)
Hence the restriction of $\pi_*K_{\ol\SE_1}$ on $X\uni$ coincides with 
$(\pi_1)_*K^1_{\ol\SE_1}[r]$. 
The following lemma can be proved in a similar way as Proposition 3.8, by using 
Proposition 2.2 (iv) (see also [L1, 6.6]). 

\begin{lem}  
$\pi_*K_{\ol\SE_1}[-r]|_{X\uni} \simeq (\pi_1)_*K^1_{\ol\SE_1}$ 
is a perverse sheaf on $X\uni$.  
\end{lem}

\para{5.13.}
Since $\pi_1$ is proper, by Deligne-Gabber's decomposition theorem, 
$(\pi_1)_*K^1_{\ol\SE_1}$ is a semisimple perverse sheaf on $X\uni$. 
Since $\pi_1$ is $H$-equivariant, each simple component in 
$(\pi_1)_*K^1_{\ol\SE_1}$ is an $H$-equivariant simple perverse sheaf, 
hence it is of the form  $K(\SE) = \IC(\ol\SO, \SE)[\dim \SO]$ for some 
$(\SO, \SE) \in \SN_G$. Here we prepare a lemma.

\begin{lem}  
Take $(\SO,\SE), (\SO',\SE') \in \SN_G$, and assume that $\SO, \SO' \subset X\uni$. 
Then 
\begin{equation*}
\dim \BH^0_c(X\uni,K(\SE)\otimes K(\SE')) = 
         \begin{cases}
                1  &\quad\text{ if $\SO = \SO'$ and $\SE' \simeq \SE^*$, } \\
                0  &\quad\text{ otherwise. }
         \end{cases}
\end{equation*}
\end{lem}

\begin{proof}
First, by assuming that $\SO \ne \SO'$, we show that 
\begin{equation*}
\tag{5.14.1}
\BH^0_c(X\uni, K(\SE)\otimes K(\SE')) = 0.
\end{equation*}
In order to show (5.14.1), it is enough to see, by using the hypercohomology 
spectral sequence, that
\par\medskip\noindent
(5.14.2) \ If $H^i_c(X\uni, \SH^jK(\SE)\otimes \SH^{j'}K(\SE'))\ne 0$, then 
$i + j + j' < 0$. 
\par\medskip
Put $D_{j,j'} = \supp\SH^jK(\SE) \cap \supp\SH^{j'}K(\SE')$. 
Since the cohomology in (5.14.2) is not changed if we replace $X\uni$ by $D_{j,j'}$.
Thus we may assume that 
\begin{equation*}
H^i_c(D_{j,j'}, \SH^jK(\SE)\otimes \SH^{j'}K(\SE')) \ne 0.
\end{equation*} 
It follows that $i \le 2\dim D_{j,j'}$. 
By the condition on the intersection cohomology for $K(\SE)$ and $K(\SE')$, we have
\begin{align*}
\dim D_{j,j'} &\le \dim\supp\SH^jK(\SE) \le -j, \\
\dim D_{j,j'} &\le \dim\supp\SH^{j'}K(\SE') \le -j'.
\end{align*}
Hence $j,j' \le -\dim D_{j,j'}$. If we assume both are equalities, 
then we have $\SO = D_{j,j'} = \SO'$, a contradiction.  Hence we have
$i + j + j' < 0$, as asserted. 
\par
Next assume that $\SO = \SO'$. Put $D = \ol\SO - \SO$. We show the following.
\par\medskip\noindent
(5.14.3) \ $\BH^0_c(D, K(\SE)\otimes K(\SE')) = 0$ and 
          $\BH^{-1}_c(D, K(\SE)\otimes K(\SE')) = 0$. 
\par\medskip
As in the previous argument, we consider the condition that 
\begin{equation*}
H^i_c(D_{j,j'}, \SH^jK(\SE)\otimes \SH^{j'}K(\SE')) \ne 0,
\end{equation*}
where $D_{j,j'} = D \cap \supp\SH^jK(\SE) \cap \supp\SH^{j'}K(\SE')$.
Again we have $i \le 2\dim D_{j,j'}$. 
But since $\dim D < \dim \SO$, we have inequalities 
$j < -\dim D_{j,j'}, j' < -\dim D_{j,j'}$, and so 
$i + j + j' < -1$. This proves (5.14.3). By using the cohomology long exact sequence
with respect to $\SO \subset \ol\SO$, (5.14.3) implies that
\begin{equation*}
\BH^0_c(X\uni, K(\SE)\otimes K(\SE')) \simeq \BH^0_c(\SO, K(\SE)\otimes K(\SE')).
\end{equation*}
Since the cohomology in the right hand side coincides with 
$H_c^{2\dim \SO}(\SO, \SE\otimes \SE')$, we obtain the required formula.
The lemma is proved.
\end{proof}

\para{5.15.}
We are now ready to prove (i) of Theorem 5.2.
Put $K = \pi_*K_{\ol\SE_1}, K_1 = K[-r]|_{X\uni} = (\pi_1)_*K^1_{\ol\SE_1}$. 
We have $\End K \simeq \SA_{\SE_1}$. Let 
\begin{equation*}
\a: \End K \to \End K_1
\end{equation*}
be the natural homomorphism. We show that $\a$ gives rise to an isomorphism. 
Let us consider the sheaf $\wt\ST$ defined in (4.3.2) with $\SE_1' = \SE_1^*$.
Note that $K_{\ol\SE^*_1} = D(K_{\ol\SE_1})$ and $\pi_*K_{\ol\SE_1^*} = D(K)$, 
where $D$ is the Verdier dual operator. 
Using the expression (4.11.1), we have

\begin{align*}
\wt\ST_1 &\simeq \BH_c^{-2r}(Z', K_{\ol\SE_1}\boxtimes D(K_{\ol\SE_1})) \\ 
         &\simeq \BH_c^{-2r}(X\uni, K|_{X\uni}\otimes D(K)|_{X\uni}) \\
         &= \BH_c^0(X\uni, K_1\otimes D(K_1)).
\end{align*}
Let $n_{\SE}$ be the multiplicity of $K(\SE)$ appearing in the decomposition 
of the semisimple perverse sheaf $K_1$. 
Since $n_{\SE}$ coincides with the multiplicity of 
$K(\SE^*) = D(K(\SE))$ in the decomposition of $DK_1$, by using Lemma 5.14, we have
\begin{equation*}
\dim \wt\ST_1 = \dim \BH_c^0(X\uni, K_1\otimes D(K_1)) 
              = \sum_{(\SO,\SE) \in \SN_G}n_{\SE}^2.
\end{equation*}

On the other hand, by Proposition 4.12,
$\dim \wt\ST_1 = \dim \SA_{\SE_1} = \dim \End K$.
Since $\dim \End K_1 = \sum n_{\SE}^2$, we have  $\dim \End K = \dim \End K_1$.
Hence it is enough to show that $\a$ is injective.    
Now $K_1$ can be decomposed as $K_1 \simeq \bigoplus_{\r \in \SA_{\SE_1}}\r 
             \otimes (K_{\r}|_{X\uni})$, up to shift, and the 
$\SA_{\SE_1}\otimes\SA_{\SE_1}^0$-module structure 
of $\wt\ST_1 = \BH^0_c(X\uni, K_1\otimes D(K_1))$ is determined from this 
decomposition.  By Proposition 4.12, we know that $\wt\ST_1$ is the two-sided regular  
representation of $\SA_{\SE_1}$. In particular, $K_{\r}|_{X\uni} \ne 0$ for any 
$\r \in \SA_{\SE_1}\wg$.  Since $K_{\r}|_{X\uni}$ is a direct summand of $K_1$, 
$K_{\r}|_{X\uni}$ is a sum of various $K(\SE)$, at least one summand.
Thus $\a$ is injective, and so $\a$ gives an isomorphism.  
(5.2.1) follows from this. 
From the above discussion, we have 
$K_{\r}|_{X\uni} \simeq \IC(\ol\SO,\SE)[\dim \SO + r]$, which 
is equivalent to (5.2.2).
Thus (i) of Theorem 5.2 is proved.  

\para{5.16.}
We prove  (ii) and (iii) of Theorem 5.2. Assume that $\SO \subset X$. 
First we show that the homomorphism in (iii) is surjective.  For this, 
it is enough to see that 
$\BH_c^{2d_{\SO} - \dim X}(\pi\iv(x) - \pi\iv(x)_{\b_0}, K_{\ol\SE_1}) = 0$ 
for any $x \in \SO$, thus enough to see that 
$\BH_c^{2d_{\SO} - \dim X}(\pi\iv(x)_{\b}, K_{\ol\SE_1}) = 0$ for any $\b \ne \b_0$. 
(Here we put $\pi\iv(x)_{\b} = \pi\iv(x) \cap \wt X_{\b}$ under the notation in 4.3.)
By using the hypercohomology spectral sequence, we have only to show 
that if $H_c^i(\pi\iv(x)_{\b}, \SH^jK_{\ol\SE_1}) \ne 0$, then $i + j < 2d_{\SO} - \dim X$. 
If the cohomology is non-zero, Proposition 2.2 (ii) implies that 
\begin{equation*}
i \le 2\dim \pi\iv(x)_{\b} \le 2d_{\SO} - (\dim \SO_L - \dim \SO_{\b}),  
\end{equation*}
and $j < -\dim \wt X_{\b} = -\dim X  + (\dim \SO_L - \dim \SO_{\b})$ 
since $\b \ne \b_0$.  
Thus $i + j < 2d_{\SO} - \dim X$ as asserted. Hence the homomorphism is surjective.
\par
For an $H$-equivariant simple local system $\SE$ on $\SO$, we denote by 
$m_{\SE}$ the multiplicity of $\SE$ in $R^{2d_{\SO}}f_!\ol\SE_1|_{\SO}$, and 
by $\wt m_{\SE}$ the multiplicity of $\SE$ in $\SH^{2d_{\SO}-\dim X}K|_{\SO}$. 
Since the homomorphism in (iii) is surjective, we have 
\begin{equation*}
\tag{5.16.1}
m_{\SE} \ge \wt m_{\SE}.
\end{equation*}
By (5.2.2), we have
$\SH^{2d_{\SO}}(K_{\r}[-\dim X])|_{\SO} \simeq 
           \SH^0(\IC(\ol\SO,\SE))|_{\SO} \simeq \SE$, 
and so $\wt m_{\SE} = n_{\SE}$. 
Hence by the discussion in 5.15, we have
\begin{equation*}
\tag{5.16.2}
\sum_{(\SO,\SE) \in \SN_G}\wt m_{\SE}^2 = \dim \SA_{\SE_1}. 
\end{equation*}
We shall prove that 
\begin{equation*}
\tag{5.16.3}
\sum_{(\SO,\SE) \in \SN_G}m_{\SE}^2 = \dim \SA_{\SE_1}.
\end{equation*}
From the definition of $m_{\SE}$, we have
\begin{equation*}
\dim H^{2\dim \SO}_c(\SO, R^{2d_{\SO}}f_!(\ol\SE_1)\otimes R^{2d_{\SO}}f_!(\ol\SE_1^*))
       = \sum_{\SE}m_{\SE}^2,
\end{equation*}
where in the sum, $\SE$ runs over all the $H$-equivariant local systems on 
a fixed $\SO$.
Let $Z_0 = \wt X_0 \times_{X}\wt X_0$ and 
$p_1 : Z_0 \to X$ be the first projection.
For an $H$-orbit $\SO \subset X\uni$, put $Z_{0,\SO} = p_1\iv(\SO)$. 
We have a natural map $Z_{0,\SO} \to \SO$, and all fibres have dimension 
$\le 2d_{\SO}$. 
Since $2d_{\SO} + \dim \SO = d_0$, where 
$d_0$ is as in (4.3.1), 
by using the Leray spectral sequence, we have
\begin{equation*}
H_c^{2d_0}(Z_{0,\SO}, \ol\SE_1\boxtimes \ol\SE_1^*) 
           \simeq H_c^{2\dim \SO}(\SO, R^{2d_{\SO}}f_!(\ol\SE_1)\otimes 
                   R^{2d_{\SO}}f_!(\ol\SE_1^*)). 
\end{equation*}
Now $Z^1_{\b_0,\b_0} = \coprod_{\SO}Z_{0,\SO}$ gives a partition of 
$Z^1_{\b_0,\b_0} = p_1\iv(X\uni)$ by locally closed pieces $Z_{0,\SO}$, where 
$\dim Z_{0,\SO} \le d_0$. Hence
\begin{align*}
\sum_{(\SO,\SE)}m_{\SE}^2 &= 
   \sum_{\SO}\dim H^{2d_0}_c(Z_{0,\SO}, \ol\SE_1\boxtimes \ol\SE_1^*)  \\
     &= \dim H^{2d_0}_c(Z^1_{\b_0,\b_0}, \ol\SE_1\boxtimes\ol\SE_1^*)  \\
     &= \dim \ST_1 = \dim \SA_{\SE_1}. 
\end{align*}
Hence (5.16.3) holds. Comparing (5.16.2) and (5.16.3), we see, 
by (5.16.1), that $m_{\SE} = \wt m_{\SE}$.  This shows that the natural 
homomorphism in (iii) is an isomorphism.  
\par
By Lemma 5.8, the condition $(\SO,\SE) \in \SN_G$ belongs to 
$\xi = (L\subset P,\SO_L,\SE_1^{\dag}) \in \ScS_G$, i.e., $(\SO,\SE) \in \SN_G^{(\xi)}$, 
 is equivalent to the condition that 
$m_{\SE} \ne 0$.  On the other hand, let $\wt\SN_G^{(\xi)}$ be the set of 
$(\SO,\SE) \in \SN_G$ such that $\IC(\ol\SO,\SE)[\dim \SO]$ appears as a 
direct summand of $\pi_*K_{\ol\SE_1}$ for $\xi = (L \subset P,\SO_L, \SE_1^{\dag})$. 
Then by Theorem 5.2 (i), the condition 
$(\SO,\SE) \in \wt\SN_G^{(\xi)}$  is equivalent to the condition 
that $\wt m_{\SE} \ne 0$.   
Since $m_{\SE} = \wt m_{\SE}$, we conclude that $\SN_G^{(\xi)} = \wt\SN_G^{(\xi)}$.
Now (ii) and (iii) follows from  (5.7.1) and Lemma 5.8.  This completes 
the proof of the theorem. 

\remark{5.17.} 
It is likely that the previous discussion for establishing the 
generalized Springer correspondence will work for a symmetric space of 
general type if it satisfies the condition that $\vD^+_0 = \emptyset$. 

\par\bigskip
\section{Restriction Theorem}

\para{6.1.}
Let $P = LU_P$ be as in 3.1. Hence $L^{\th} \simeq (GL_1)^a \times GL^{\th}_{N_0}$ 
with $N_0 = N - 2a$.  We consider a $\th$-stable parabolic subgroup $Q$ of $G$ 
containing $P$, and the $\th$-stable Levi subgroup $M$ of $Q$ containing $L$.  
Here we assume that $M^{\th} \simeq (GL_1)^{a'} \times GL^{\th}_{N_0'}$, where 
$N_0' = N - 2a'$ with $a \ge a'$.    
Let $\SO_L$ be an $L_H$-orbit in $L^{\io\th}\uni$, and $\SE_1^{\dag}$ a cuspidal 
local system on $\SO_L$. Put $\vS = Z_L^{\io\th}\SO_L$, and 
$\SE_1 = \Ql\boxtimes\SE_1^{\dag}$. 
We consider the complex $\pi_*K_{\ol\SE_1}$ on $X$ obtained from the triple 
$(L \subset P, \SO_L, \SE_1^{\dag}) \in \ScS_G$.  Then the endomorphism algebra 
$\End(\pi_*K_{\ol\SE_1})$ is isomorphic to $\SA_{\SE_1}$, which is a twisted group   
algebra of $\SW_{\SE_1} \subset \SW_1 \simeq S_a$.
By applying a similar discussion for 
$(L \subset M \cap P, \SO_L, \SE_1^{\dag}) \in \ScS_M$, 
one can obtain a complex $\pi'_*K'_{\ol\SE_1}$
on $X'$.  Then the endomorphism algebra 
$\End(\pi'_*K'_{\ol\SE_1}) \simeq \SA'_{\SE_1}$, which is isomoprhic to a twisted 
group algebra of $\SW '_{\SE_1} \subset \SW '_1 \simeq S_{a'}$.    
(Here we denote by $X', \pi'$, etc. by attaching the primes to express the objects 
for $M$ corresponding to $X, \pi$, etc. for $G$).
$\SA'_{\SE_1}$ is canonically identified with the subalgebra of $\SA_{\SE_1}$. 
\par
We now apply the generalized Springer correspondence for $H$ and $M_H = (M^{\th})^0$. 
By Theorem 5.2 (ii), if $(\SO',\SE') \in \SN_M$ belongs to the series 
$(L \subset M \cap P, \SO_L, \SE_1^{\dag}) \in \ScS_M$, then an irreducible representation 
$\r' \in (\SA'_{\SE_1})\wg$ is determined from $(\SO',\SE')$. 
On the other hand, if $(\SO,\SE) \in \SN_G$ belongs to the series 
$(L \subset P, \SO_L, \SE_1^{\dag}) \in\ScS_G$, then an irreducible representation 
$\r \in \SA\wg_{\SE_1}$ is determined.  
We will describe the relationship between $\r'$ and $\r$.  
\par
Take $\SO$ and $\SO'$ as above. Let
$f_{\SO, \SO'} : \eta_{Q}\iv (\SO') \cap \SO \to \SO'$ be the restriction of 
the map $\eta_Q : Q^{\io\th} \to M^{\io\th}$. 
We define an integer $m_{\SE, \SE'}$ as the multiplicity of $\SE'$ in the 
local system $R^{2d_{\SO,\SO'}}(f_{\SO,\SO'})_!(\SE)$, where we put    
$d_{\SO,\SO'} = (\dim \SO - \dim \SO')/2 + \vD_Q/2$. The following result is 
an analogue of Lusztig's restriction theorem [L1, Theorem 8.3]. 

\begin{thm}  
Assume that $(\SO',\SE') \in \SN_M$ belongs to 
$(L \subset M \cap P,\SO_L, \SE_1^{\dag}) \in \ScS_M$. 
\begin{enumerate}
\item
If $m_{\SE,\SE'} \ne 0$, then $(\SO,\SE) \in \SN_G$ belongs to 
$(L \subset P,\SO_L, \SE_1^{\dag})$. 
\item
Assume that $(\SO,\SE) \in \SN_G$ belongs to $(L \subset P,\SO_L,\SE_1^{\dag})$, and let 
$\r \in \SA_{\SE_1}\wg$ $($resp. $\r' \in (\SA'_{\SE_1})\wg$ $)$ 
be the irreducible representation corresponding to $(\SO,\SE)$ $($resp. 
$(\SO',\SE')$ $)$.
Then we have
\begin{equation*}
\lp \r|_{\SA'_{\SE_1}}, \r' \rp_{\SA'_{\SE_1}} = m_{\SE, \SE'},
\end{equation*} 
where the left hand side means the multiplicity of $\r'$ in the restriction of 
$\r$ on $\SA'_{\SE_1}$.
\end{enumerate}
\end{thm}

\para{6.3.}
The remainder of this section is devoted to the proof of the theorem. 
The proof is done by a similar strategy as in [L1]. 
Take $(L \subset P,\SO_L, \SE_1^{\dag}) \in \ScS_G$, and let 
$Y = \bigcup_{g \in H}g(\vS\reg \SD^0_P)g\iv$, 
$\ol Y = X = \bigcup_{g \in H}g\eta_P\iv(\ol\vS)g\iv$
be as in Section 3.  
By replacing $G, P$ by $M, M \cap P$, we obtain similar 
varieties,
\begin{equation*}
Y^M = \bigcup_{g_1 \in M_H}g_1(\vS\reg \SD^0_{M \cap P})g_1\iv, \quad 
\ol{Y^M} = X^M = \bigcup_{g_1 \in M_H}g_1\eta_{M \cap P}\iv(\ol\vS)g_1\iv.
\end{equation*} 
Consider the following varieties,

\begin{align*}
\wt Y &= \{ (x, gL_H) \in G^{\io\th} \times H/L_H \mid g\iv xg \in \vS\reg\SD^0_P \}, \\ 
\wt Y_M &= \{ (x, gM_H) \in G^{\io\th} \times H/M_H 
               \mid g\iv xg \in Y^M \SD^0_Q \}, \\
\wt X &= \{ (x, gP_H) \in G^{\io\th} \times H/P_H \mid g\iv xg \in \eta_P\iv(\ol\vS) \}, \\
\wt X_Q &= \{ (x, gQ_H) \in G^{\io\th} \times H/Q_H \mid g\iv xg \in \eta_Q\iv(X^M)\}.
\end{align*}
Note that $Y^M\SD^0_Q \subset Q^{\io\th}$ since 
\begin{equation*}
g_1(\vS\reg \SD^0_{M \cap P})g_1\iv \cdot \SD^0_Q = 
    g_1(\vS\reg \SD^0_{M \cap P}\SD^0_Q)g_1\iv = 
    g_1(\vS\reg \SD^0_P)g_1\iv \subset Q^{\io\th}
\end{equation*} 
for $g_1 \in M_H$. 
Let $\pi_0 : \wt Y \to Y, \pi : \wt X \to X$ be the first projections as before.
We have the following commutative diagram,

\begin{equation*}
\tag{6.3.1}
\begin{CD}
\wt X @>\pi'>> \wt X_Q @>\pi''>> X  \\
@Aj_LAA         @AAj_MA       @AAjA   \\
\wt Y  @>\pi_0'>> \wt Y_M  @>\pi_0''>>  Y,
\end{CD}
\end{equation*}
where $\pi': (x,gP_H) \mapsto (x, gQ_H), \pi'': (x, gQ_H) \mapsto x$, 
$\pi_0': (x, gL_H) \mapsto (x, gM_H), \pi_0'' : (x, gM_H) \mapsto x$. 
(Note that $\pi_0'$ is well-defined since  
$g_1\iv xg_1 \in Y^M\SD^0_Q$ for $x \in \vS\reg \SD^0_P$, $g_1 \in M_H$ by 
$\SD^0_P = \SD^0_{P \cap M} \times \SD^0_Q$.)
Moreover, $j_L : (x, gL_H) \mapsto (x, gP_H), j_M : (x, gM_H) \mapsto (x, gQ_H)$,
$j : x \mapsto x$ are natural inclusions. 
By definition, $\pi = \pi''\circ \pi', \pi_0 = \pi''_0\circ \pi'_0$. 
By Lemma 3.5, $j_L, j$ give embeddings of $\wt Y, Y$ into $\wt X, X$ as open dense 
smooth subsets.  Since $\wt Y_M = (\pi'')\iv(Y)$, $\wt Y_M$ is open dense in $\wt X_Q$.
Here $\wt Y_M \simeq H \times^{M_H}(Y^M \SD^0_Q)$. Since 
$Y^M$ is smooth by Lemma 3.5, and so $Y^M \SD^0_Q$ is smooth, thus $\wt Y_M$ is smooth. 
\par
Recall that $\pi_0 : \wt Y \to Y$ is a finite Galois covering with Galois group 
$\SW_1 \simeq S_a$. $\SL = (\pi_0)_*\wt\SE_1$ is a semisimple 
local system on $Y$ and we have $\End \SL \simeq \SA_{\SE_1}$.  Thus $\SL$ is decomposed as 
\begin{equation*}
\tag{6.3.2}
\SL \simeq \bigoplus_{\r \in \SA_{\SE_1}\wg}\r \otimes \SL_{\r},
\end{equation*}  
where $\SL_{\r} = \Hom_{\SA_{\SE_1}}(\r, \SL)$. 
Similarly, $\pi_0' : \wt Y \to \wt Y_M$ is a finite Galois covering 
with group $\SW_1' \simeq S_{a'}$. Thus $\SL' = (\pi_0')_*\wt\SE_1$ is a semisimple local 
system on $\wt Y_M$, and we have $\End \SL' \simeq \SA'_{\SE_1}$. 
We have a decomposition into simple local systems, 
\begin{equation*}
\tag{6.3.3}
\SL' \simeq \bigoplus_{\r' \in (\SA'_{\SE_1})\wg}\r' \otimes \SL'_{\r'},
\end{equation*}
where $\SL'_{\r'} = \Hom_{\SA'_{\SE_1}}(\r', \SL')$. 
Since $\SL \simeq (\pi''_0)_*(\pi'_0)_*\wt\SE_1 \simeq (\pi''_0)_*\SL'$,
by applying $(\pi''_0)_*$ on both sides of (6.3.3), we have
\begin{equation*}
\tag{6.3.4}
\SL \simeq \bigoplus_{\r' \in (\SA'_{\SE_1})\wg}\r' \otimes (\pi''_0)_*\SL'_{\r'}.
\end{equation*} 
\par
By comparing (6.3.2) and (6.3.4), we see that $(\pi''_0)_*\SL'_{\r'}$ is a semisimple 
local system whose simple components are of the form $\SL_{\r}$ for some 
$\r \in \SA_{\SE_1}\wg$.  Since the action of $\SA'_{\SE_1}$ on $\SL$ in (6.3.4) is 
the restriction to $\SA'_{\SE_1}$ of the action of $\SA_{\SE_1}$ on $\SL$ in (6.3.2), 
we obtain the following decomposition into simple local systems.

\begin{equation*}
\tag{6.3.5}
(\pi''_0)_*\SL'_{\r'} \simeq \bigoplus_{\r \in \SA_{\SE_1}\wg}
             \SL_{\r}\otimes \Ql^{\lp \r,\r'\rp},
\end{equation*}  
where $\lp \r, \r'\rp = \lp \r|_{\SA'_{\SE_1}}, \r'\rp_{\SA'_{\SE_1}}$ is the multiplicity 
of $\r'$ in the restriction of $\r$ on $\SA'_{\SE_1}$. 

\para{6.4.}
We now consider $K_{\ol\SE_1} = \IC(\wt X, \ol\SE_1)[\dim \wt X]$.
By Proposition 3.8, 
$\pi_*K_{\ol\SE_1}$ is a semisimple perverse sheaf on $X$, and is decomposed 
into simple perverse sheaves

\begin{equation*}
\tag{6.4.1}
\pi_*K_{\ol\SE_1} \simeq \bigoplus_{\r \in \SA_{\SE_1}\wg}\r \otimes A_{\r},
\end{equation*}
where $A_{\r} = \IC(X, \SL_{\r})[\dim X]$. 
\par
Next we shall show that 
\begin{equation*}
\tag{6.4.2}
\pi'_*K_{\ol\SE_1} \simeq \IC(\wt X_Q,\SL')[\dim \wt X_Q]. 
\end{equation*}
Under the isomorphisms

\begin{equation*}
\tag{6.4.3}
\begin{aligned}
H \times^{P_H}(\eta_P\iv(\ol\vS)) &\isom \wt X, \\
H \times^{Q_H}(\eta_Q\iv(X^M)) &\isom \wt X_Q,
\end{aligned}
\end{equation*}
consider the following commutative diagram

\begin{equation*}
\tag{6.4.4}
\begin{CD}
H \times^{P_H}(\eta_P\iv(\ol\vS)) @<\wt\xi<<  H \times (Q_H \times^{P_H}(\eta_P\iv(\ol\vS))
               @>\wt\eta>> \wt X^M   \\
@V\pi'VV            @VV \wh\pi V      @VV\pi^M V    \\
H \times^{Q_H}(\eta_Q\iv(X^M))   @<\xi<<   H \times \eta_Q\iv(X^M) @>\eta>>  X^M,  
\end{CD}
\end{equation*}
where $\xi$ is the quotient map under the free action of $Q_H$ on 
$H \times \eta_Q\iv(X^M)$ by $q: (g,y) \mapsto (gq\iv, qyq\iv)$,  
$\wt\xi$ is the quotient map under the free action of $Q_H$ on 
$H \times (Q_H \times^{P_H}(\eta_P\iv(\ol\vS))$ as above, under the natural
isomorphism 
\begin{equation*}
H \times^{Q_H}(Q_H \times^{P_H}(\eta_P\iv(\ol\vS))) \simeq 
              H \times^{P_H}(\eta_P\iv(\ol\vS)).
\end{equation*}
$\eta$ is the composite of the projection $H \times \eta_Q\iv(X^M) \to \eta_Q\iv(X^M)$
and the map $\eta_Q$. 
The map $\pi^M : \wt X^M \to X^M$ is the map analogous to $\pi : \wt X \to X$ obtained 
by replacing $G$ by $M$. 
$\wt\eta$ is the composite of the projection 
$H \times (Q_H \times^{P_H}(\eta_P\iv(\ol\vS)) \to Q_H \times^{P_H}(\eta_P\iv(\ol\vS))$
and the natural map 
\begin{equation*}
Q_H \times^{P_H}(\eta_P\iv(\ol\vS)) \to 
    M_H \times^{(P_H \cap M_H)}(\eta_{P\cap M}\iv(\ol\vS)) \simeq \wt X^M 
\end{equation*}
obtained from the projection $Q \to M$. 
Under the isomorphisms in (6.4.3), $\pi'$ coincides with the map induced from 
the natural inclusion $H \times \eta_P\iv(\ol\vS) \hra H \times \eta_Q\iv(X^M)$.
$\wh \pi = \id \times \wh\pi_1$, where 
$\wh\pi_1 : Q_H\times^{P_H}(\eta_P\iv(\ol\vS)) \to \eta_Q\iv(X^M)$ induced from  
$(g,x) \mapsto gxg\iv$. 
\par
Here $\eta, \wt\eta$ are $Q_H$-equivariant with respect to the natural action of 
$M_H$ (and the trivial action of $U_{Q_H}$) on $\wt X^M$ and on $X^M$.  
Note that two squares are cartesian squares. 
\par
By applying the construction of $K_{\ol\SE_1} = \IC(\wt X, \ol\SE_1)[\dim \wt X]$ 
to $M$, one can define the perverse sheaf 
$K_{\ol\SE^M_1} = \IC(\wt X^M, \ol\SE^M_1)[\dim \wt X^M]$ on $\wt X^M$, and we obtain 
an $M_H$-equivariant  semisimple perverse sheaf $(\pi^M)_*K_{\ol\SE_1^M}$ on $X^M$.  
Here $\eta$ is a smooth morphism with connected fibre isomorphic to 
$H \times U_Q^{\io\th}$, and is $Q_H$-equivariant.  Then 
$\eta^*(\pi^M)_*K_{\ol\SE^M_1}[\a]$ is a $Q_H$-equivariant perverse sheaf on 
$H \times \eta_Q\iv(X^M)$, where $\a = \dim H + \dim U_Q^{\io\th}$. 
On the other hand, since $\xi$ is a $Q_H$-principal bundle, there exists 
a unique perverse sheaf $K_1$ on $H \times^{Q_H}(\eta_Q\iv(X^M))$ such that
\begin{equation*}
\tag{6.4.5}
\xi^*K_1[\b] \simeq \e^*(\pi^M)_*K_{\ol\SE^M_1}[\a],
\end{equation*}  
where $\b = \dim Q_H$. 
\par
A similar construction works also for the $M_H$-equivariant perverse sheaf
$K_{\ol\SE^M_1}$ on $\wt X^M$, and one can find a perverse sheaf $K_2$ on 
$H \times^{P_H}(\eta_P\iv(\ol\vS))$, unique up to isomorphism, such that

\begin{equation*}
\tag{6.4.6}
\wt\xi^*K_2[\b] \simeq \wt\eta^* K_{\ol\SE^M_1}[\a].
\end{equation*}    
Here, if we note that the two squares in the diagram (6.4.4) are cartesian squares, 
by using the proper base change theorem, (6.4.5) and (6.4.6) imply that 
\begin{equation*}
\tag{6.4.7}
K_1 \simeq \pi'_*K_2.
\end{equation*}
Thus, in order to prove (6.4.2), it is enough to see that 
\begin{equation*}
\tag{6.4.8}
K_2 \simeq \IC(\wt X, \ol\SE_1)[\dim \wt X], \quad 
     K_1 \simeq \IC(\wt X_Q, (\pi'_0)_*\wt\SE_1)[\dim \wt X_Q]. 
\end{equation*}
\par
It follows from the construction that $K_1$ (resp. $K_2$) is an 
intersection cohomology complex on $\wt X_Q$ (resp. $\wt X$) whose 
support is $\wt X_Q$ (resp. $\wt X$). Hence to prove (6.4.8), 
we have only to show that

\begin{equation*}
\tag{6.4.9}
K_2|_{\wt Y} \simeq \wt\SE_1[\dim \wt Y], \quad 
K_1|_{\wt Y_M} \simeq (\pi_0')_*\wt\SE_1[\dim \wt Y_M].
\end{equation*}

Under the isomorphisms
\begin{align*}
\wt Y &\simeq H \times^{L_H}(\vS\reg\SD^0_P), \\
\wt Y_M &\simeq H \times^{M_H}(Y^M \SD^0_Q),
\end{align*}
we have the following commutative diagram.

\begin{equation*}
\tag{6.4.10}
\begin{CD}
H \times^{L_H}(\vS\reg\SD^0_P) @<\wt\xi_0<<  H \times (M_H \times^{L_H}(\vS\reg\SD^0_P))
                    @>\wt\eta_0>> \wt Y^M   \\
@V\pi_0'VV    @VV\wh\pi_0 V     @VV\pi_0^MV   \\
H \times^{M_H}(Y^M \SD^0_Q)   @<\xi_0<<  H \times (Y^M\SD^0_Q)   @>\eta_0>>  Y^M.
\end{CD}
\end{equation*}
The maps are defined similarly to (6.4.4), and the diagram (6.4.10) enjoys 
similar properties as (6.4.4).   In particular, 
the squares in the diagram are both cartesian. 
We consider the local system $\wt\SE_1^M$ on $\wt Y^M$.  Then one can check 
that the local system on $H \times^{L_H}(\vS\reg\SD^0_P)$, obtained from $\wt\SE^M_1$
by a similar discussion as before, coincides with $\wt\SE_1$ on $\wt Y$.   
It follows from (6.4.10) that the local system on $H \times^{M_H}(Y^M\SD^0_Q)$, 
obtained from  the local system $(\pi_0^M)_*\wt\SE^M_1$ on $Y^M$ by a similar discussion 
as before, coincides with $(\pi'_0)_*\wt\SE_1$. (6.4.9) now follows from this by applying 
the proper base change theorem to (6.4.4) and (6.4.10). Thus (6.4.2) is proved. 

\para{6.5.}
By applying the decomposition of $\SL'$ in (6.3.3) to the isomorphism in (6.4.2), 
we obtain the decomposition of $\pi'_*K_{\ol\SE_1}$ into simple perverse sheaves, 
\begin{equation*}
\tag{6.5.1}
\pi'_*K_{\ol\SE_1} \simeq \bigoplus_{\r' \in (\SA'_{\SE_1})\wg}\r'\otimes B_{\r'}, 
\end{equation*} 
where $B_{\r'} \simeq \IC(\wt X_Q, \SL'_{\r'})[\dim \wt X_Q]$. 
By applying $\pi''_*$ on both sides of (6.5.1), we obtain 
\begin{equation*}
\tag{6.5.2}
\pi_*K_{\ol\SE_1} \simeq \bigoplus_{\r' \in (\SA'_{\SE_1})\wg}\r' \otimes \pi''_*B_{\r'}.
\end{equation*}
Since $\pi''$ is proper, $\pi''_*B_{\r'}$ is a semisimple complex by the theorem 
of Deligne-Gabber. Since $\pi''_*B_{\r'}$ is a direct summand of a semisimple 
perverse sheaf $\pi_*K_{\ol\SE_1}$, $\pi''_*B_{\r'}$ is a semisimple perverse sheaf, 
and all of its simple components have support on $X$. 
Since $\pi''_*B_{\r'}|_Y \simeq (\pi_0'')_*\SL'_{\r'}[\dim X]$, we have
\begin{equation*}
\pi''_*B_{\r'} \simeq \IC(X, (\pi_0'')_*\SL'_{\r'})[\dim X].
\end{equation*}   
By applying the decomposition of $(\pi''_0)_*\SL'_{\r'}$ in (6.3.5), we obtain 
the following lemma, which is a perverse sheaf version of (6.3.5).

\begin{lem}  
For each $\r' \in (\SA'_{\SE_1})\wg$, $\r \in \SA_{\SE_1}\wg$, put 
$A_{\r} = \IC(X, \SL_{\r})[\dim X]$, 
$B_{\r'} = \IC(\wt X_Q, \SL'_{\r'})[\dim \wt X_Q]$.  Then 
we have an isomorphism of perverse sheaves on $X$,
\begin{equation*}
\tag{6.6.1}
\pi''_*B_{\r'} \simeq \bigoplus_{\r \in \SA_{\SE_1}\wg}A_{\r}\otimes \Ql^{\lp \r,\r'\rp}. 
\end{equation*}
\end{lem}

\para{6.7.}
Let $(\SO_{\r}, \SE_{\r}) \in \SN_G$ be the pair which belongs to 
$(L \subset P, \SO_L, \SE^{\dag}_1) \in \ScS_G$ and corresponds to $\r \in \SA\wg_{\SE_1}$. 
Put $d = \dim X$, and let 
$d_{\r} = d_{\SO_{\r}}$ be as in Theorem 5.2. 
By (5.2.2), we have $A_{\r}|_{X\uni} \simeq \IC(\ol\SO_{\r}, \SE_{\r})[d - 2d_{\r}]$, hence
the following property holds.
For any $\wt \r \in \SA_{\SE_1}\wg$, we have
\par\medskip\noindent
(6.7.1) \ 
$\SH^{2d_{\r} - d}A_{\wt \r}|_{\SO_{\r}}$ is isomorphic to $\SE_{\wt\r}$ 
if $\SO_{\wt\r} = \SO_{\r}$. 
\par\medskip\noindent
Comparing (6.7.1) with Lemma 6.6, we have the following.

\par\medskip\noindent
(6.7.2) \ $\lp \r, \r'\rp \le $ the multiplicity of $\SE_{\r}$ 
in the local system $\SH^{2d_{\r} - d}(\pi''_*B_{\r'})|_{\SO_{\r}}$. 
\par\medskip

Put $(\wt X_Q)\uni = \wt X_Q \cap (G^{\io\th}\uni \times H/Q_H)$. 
We will consider the restriction of $B_{\r'}$ on $(\wt X_Q)\uni$. 
Let $(\SO'_{\r'}, \SE'_{\r'}) \in \SN_M$ be the pair which belongs to 
$(L \subset M \cap P, ,\SO_L, \SE^{\dag}_1) \in \ScS_M$ and 
which corresponds to $\r' \in (\SA'_{\SE_1})\wg$.  
Put

\begin{equation*}
D = \{ (x,gQ_H) \in G^{\io\th}\uni 
           \times H/Q_H \mid g\iv xg \in \eta_Q\iv(\ol\SO'_{\r'}) \}.
\end{equation*}
$D$ is a closed subvariety of $(\wt X_Q)\uni$. We note that 

\begin{equation*}
\tag{6.7.3}
\supp B_{\r'}|_{(\wt X_Q)\uni} \subset D.
\end{equation*}
In fact, as in(6.4.1), we denote by $A'_{\r'}$ the simple perverse sheaf on $X^M$ 
appearing in the decomposition of $(\pi^M)_*K_{\ol\SE^M_1}$, which corresponds 
to $\r' \in (\SA'_{\SE_1})\wg$. Then $B_{\r'}$ can be constructed from $A'_{\r'}$ 
similarly to $K_1, K_2$, by using the diagram (6.4.4). 
On the other hand,  by the generalized Springer correspondence for $M$, 
the restriction of $A'_{\r'}$ on $M^{\io\th}\uni$ coincides with 
$\IC(\ol\SO'_{\r'},\SE_{\r'})$, up to shift. (6.7.3) follows from this.
\par
Let $\pi''_D$ be the restriction of $\pi''$ on $D$. 
In view of (6.7.3), (6.7.2) can be rewritten as follows.  
\par\medskip
\noindent
(6.7.4) \ $\lp \r, \r'\rp \le $ the multiplicity of $\SE_{\r}$ 
in the local system 
$\SH^{2d_{\r} - d}((\pi''_D)_!(B_{\r'}|_D))|_{\SO_{\r}}$. 
\par\medskip

Put $D_0 = \{ (x, gQ_H) \in G^{\io\th}\uni 
          \times H/Q_H \mid g\iv xg \in \eta_Q\iv(\SO'_{\r'})\}$.  
$D_0$ is an open subset of $D$. Let $\pi''_{D_0}$ be the restriction of $\pi''$ 
on $D_0$. 
We define an integer $x_{\r,\r'}$ as the multiplicity of $\SE_{\r}$ in the 
local system  $\SH^{2d_{\r} - d}((\pi''_{D_0})_!(B_{\r'}|_{D_0}))|_{\SO_{\r}}$. 
We show

\begin{prop}  
$\lp \r,\r'\rp$ coincides with $x_{\r,\r'}$. 
\end{prop}

\begin{proof}
First we show that 
\begin{equation*}
\tag{6.8.1}
x_{r,\r'} \ge \lp \r, \r'\rp. 
\end{equation*}

Consider the long exact sequence of local systems obtained from the open embedding 
$D_0 \hra D$, 
\begin{equation*}
\begin{CD}
\cdots @>>> \SH^{2d_{\r} - d}((\pi''_{D_0})_!(B_{\r'}|_{D_0}))|_{\SO_{\r}} @>\vf >>  
     \SH^{2d_{\r}- d}((\pi''_D)_!(B_{\r'}|_D)|_{\SO_{\r}}  \\
@>>>  \SH^{2d_{\r}-d}((\pi''_{D - D_0})_!(B_{\r'}|_{D - D_0}))|_{\SO_{\r}} @>>> 
           \cdots.\phantom{**********}   
\end{CD}
\end{equation*}
If $\vf$ is surjective, then (6.8.1) holds by (6.7.4).  Thus it is enough to see that
\begin{equation*}
\SH^{2d_{\r}-d}((\pi''_{D - D_0})_!(B_{\r'}|_{D - D_0}))|_{\SO_{\r}} = 0. 
\end{equation*}
For this we have only to show that the stalk vanishes 
at any $x \in \SO_{\r}$.  Hence we will show that 
\begin{equation*}
\tag{6.8.2}
\BH^{2d_{\r} - d}_c(({\pi''}\iv(x) \cap (D - D_0), B_{\r'}) = 0 
\end{equation*} 
for any $x \in \SO_{\r}$. 

\par
We consider the decomposition of $\ol\SO'_{\r'}$ into $M_H$-orbits, 
$\ol\SO'_{\r'} - \SO'_{\r'} = \coprod_{\a}\SO'_{\a}$.  This implies 
the decomposition of $D - D_0$ into locally closed subsets, 
$D - D_0 = \coprod_{\a}D_{\a}$, where 

\begin{equation*}
D_{\a} = \{ (x, gQ_H) \in G^{\io\th} \times H/Q_H \mid g\iv xg \in \eta_Q\iv(\SO'_{\a}) \}.
\end{equation*}
In order to show (6.8.2), it is enough to see, for any $x \in \SO_{\r}$, and for 
any $D_{\a}$, that

\begin{equation*}
\tag{6.8.3}
\BH^{2d_{\r}-d}_c({\pi''}\iv(x) \cap D_{\a}, B_{\r'}) = 0.
\end{equation*} 
We remark that the following holds.

\par\medskip\noindent
(6.8.4) \ If $H^i_c({\pi''}\iv(x) \cap D_{\a}, \SH^jB_{\r'}) \ne 0$, then 
$i + j < 2d_{\r} - d$. 
\par\medskip

In fact, since the cohomology is non-zero, we have

\begin{equation*}
i \le 2\dim ({\pi''}\iv(x) \cap D_{\a}) \le 
    (2\nu_H - \dim \SO_{\r}) - (2\nu_{M_H} - \dim \SO'_{\a}) + \vD_Q, 
\end{equation*}
where the second formula is obtained from Proposition 2.2 (ii), applied to 
$Q_H \subset H$.
On the other hand, the condition $\SH^jB_{\r'}|_{D_{\a}} \ne 0$ is 
equivalent to the condition that $\SH^{j + (d-d')}A'_{\r'}|_{\SO'_{\a}} \ne 0$ 
under the correspondence $B_{\r'} \lra A'_{\r'}$ through the diagram (6.4.4), 
where $d' = \dim X^M$. 
Hence by applying (5.2.2) for $M$, we have 

\begin{align*}
j + d &\le (2\nu_{M_H} - \dim \SO'_{\r'}) - (2\nu_{L_H} - \dim \SO_L) + \vD_{P \cap M}  \\
      &< (2\nu_{M_H} - \dim \SO'_{\a}) - (2\nu_{L_H} - \dim \SO_L) + \vD_{P \cap M}. 
\end{align*}
Since $\vD_Q + \vD_{P \cap M} = \vD_P$, we obtain (6.8.4). 
\par
Now by applying the hypercohomology spectral sequence, we obtain (6.8.3) from 
(6.8.4). Thus (6.8.2) holds, and so (6.8.1) is proved. 
\par
Since $\sum_{\r' \in (\SA'_{\SE_1})\wg}(\dim \r') \lp \r, \r'\rp = \dim \r$, 
in view of (6.8.1), 
in order to prove the proposition, it is enough to show the following formula. 

\begin{equation*}
\tag{6.8.5}
\sum_{\r' \in (\SA'_{\SE_1})\wg}(\dim \r')x_{\r,\r'} = \dim \r.
\end{equation*}

For an $M_H$-orbit $\SO'$ in $M^{\io\th}\uni$, put
$x_{\SO'} = \sum_{\r'}(\dim \r')x_{\r,\r'}$, where the sum  is taken over all 
$\r' \in (\SA'_{\SE_1})\wg$ such that $\SO'_{\r'} = \SO'$. Then the sum in the 
left hand side of (6.8.5) coincides with $\sum_{\SO'}x_{\SO'}$, where the sum is taken 
over all $M_H$-orbits $\SO'$ in $M^{\io\th}\uni$.  
We consider $x_{\SO'}$ separately.  It follows from the decomposition (6.5.1) and 
from the definition of $x_{\r,\r'}$, we see that $x_{\SO'}$ coincides with 
the multiplicity of $\SE_{\r}$ in the local system
\begin{equation*}
\tag{6.8.6}
\SH^{2d_{\r} - d}\bigl((\pi''_{D_{\SO'}})_!(\pi'_*K_{\ol\SE_1}|_{D_{\SO'}})\bigr)|_{\SO_{\r}},
\end{equation*}
where $D_{\SO'}$ is defined similarly to $D_0$ by replacing $\SO'_{\r'}$ by $\SO'$. 

Here we note that the local system given in (6.8.6) coincides with the local system

\begin{equation*}
\tag{6.8.7}
\SH^{2d_{\r}- 2d'_{\r'}}\biggl((\pi''_{D_{\SO'}})_!
       \bigl(\SH^{2d'_{\r'}-d}(\pi'_*K_{\ol\SE_1})|_{D_{\SO'}}\bigr)\biggr)|_{\SO_{\r}}.
\end{equation*} 
since the discussion in the proof of (6.8.4) shows that 
if $\SH^i\bigl((\pi''_{D_{\SO'}})_!\SF\bigr)|_{\SO_{\r}} \ne 0$ for a sheaf $\SF$, then 
$i \le (2\nu_H - \dim \SO_{\r}) - (2\nu_{M_H} - \dim \SO'_{\r'}) + \vD_Q = 
    2d_{\r} - 2d'_{\r'}$ and 
if $\SH^j(\pi'_*K_{\ol\SE_1})|_{D_{\SO'}} \ne 0$, then $j \le 2d'_{\r'} - d$. 
\par
By applying Theorem 5.2 (iii) to $M$, and by converting it to 
the sheaves on $\wt X_0$ by making use of the diagram (6.4.4), we obtain 

\begin{equation*}
\SH^{2d'_{\r'} - d}(\pi'_*K_{\ol\SE_1})|_{D_{\SO'}} \simeq 
               \SH^{2d'_{\r'}}(\pi'|_{\wt X_0})_!\ol\SE_1|_{D_{\SO'}},
\end{equation*}
where $\wt X_0 \simeq H \times^{P_H} \eta_P\iv(\vS)$. 
Hence (6.8.7) can be rewritten as 

\begin{equation*}
\SH^{2d_{\r}- 2d'_{\r'}}\biggl((\pi''_{D_{\SO'}})_!
       \bigl(\SH^{2d'_{\r'}}(\pi'|_{\wt X_0})_!\ol\SE_1|_{D_{\SO'}}\bigr)\biggr)|_{\SO_{\r}}.
\end{equation*}
By using the spectral sequence associated to the composite $\pi = \pi''\circ \pi'$, 
we see that the last formula is equal to 
\begin{equation*}
\tag{6.8.8}
\SH^{2d_{\r}}\bigl((\pi''\circ\pi'|_{\wt X_{0,\SO'}})_!\ol\SE_1\bigr)|_{\SO_{\r}},
\end{equation*}
where $\wt X_{0,\SO'} = \wt X_0 \cap {\pi'}\iv(D_{\SO'})$. 
In particular, $x_{\SO'}$ coincides with the multiplicity of $\SE_{\r}$ 
in the local system given in (6.8.8). 
\par
Here we note that $\wt X_0 = \coprod_{\SO'}\wt X_{0,\SO'}$ gives a partition of 
$\wt X_0$ by locally closed pieces $\wt X_{0,\SO'}$, 
where $\SO'$ runs over all the $M_H$-orbits 
in $M^{\io\th}\uni$. 
By Proposition 2.2 (ii), $\dim (\wt X_{0, \SO'} \cap \pi\iv (x)) \le d_{\r}$ 
for any $x \in \SO_{\r}$. Then by applying the cohomology long exact sequence 
of $\pi\iv(x)$, we see that $\sum_{\SO'}x_{\SO'}$ coincides with the multiplicity 
of $\SE_{\r}$ in the local system 

\begin{equation*}
\tag{6.8.9}
\SH^{2d_{\r}}((\pi|_{\wt X_0})_!\ol\SE_1)|_{\SO_{\r}}.
\end{equation*}

On the other hand, the stalk at $x \in \SO_{\r}$ of the local system in (6.8.9) 
coincides with $H^{2d_{\r}}_c(\SP_{\SO_L,x}, \ol\SE_1)$, where 
$\SP_{\SO_L,x} = \{ gP_H \in H/P_H \mid g\iv xg \in \eta_P\iv(\SO_L)\}$.  
By Theorem 5.2, $H^{2d_{\r}}_c(\SP_{\SO_L,x}, \ol\SE_1)$ has a natural structure of 
$\SA_{\SE_1} \times A_H(x)$-module, and is decomposed as

\begin{equation*}
\tag{6.8.10}
H^{2d_{\r}}_c(\SP_{\SO_L,x}, \ol\SE_1) \simeq 
     \bigoplus_{\wt\r \in \SA_{\SE_1}\wg} V_{\wt\r} \otimes \tau_{\wt\r}, 
\end{equation*}
where $V_{\wt\r}$ is an $\SA_{\SE_1}$-module, which is isomorphic to $\wt\r$ if it is non-zero, 
and $\t_{\wt\r}$ is an irreducible representation of $A_H(x)$ corresponding to the local 
system $\SE_{\wt\r}$ on $\SO_{\r} = \SO_{\wt\r}$. 
In particular, the multiplicity of $\SE_{\r}$ in the local system in (6.8.9) 
coincides with $\dim \r$ if it is non-zero.  
But by (6.8.1), we know that $\sum_{\SO'}x_{\SO'} \ge \dim \r$.  It follows that 
$\sum_{\SO'}x_{\SO'} = \dim \r$. Thus we have proved (6.8.5), and the proposition 
follows. 
\end{proof}

\para{6.9.}
We are now ready to prove the theorem.
For an $H$-orbit $\SO$ in $G^{\io\th}\uni$ and an $M_H$-orbit $\SO'$ in 
$M^{\io\th}\uni$, consider the diagram as in (5.3.1), but by replacing 
$P$ by $Q$, namely, 
\begin{equation*}
\begin{CD}
V = H \times^{Q_H}(\SO \cap \eta_Q\iv(\SO')) @>f_1>> \SO  \\
@Vf_2VV   \\
V' = H \times^{Q_H}\SO'.
\end{CD}
\end{equation*}
By putting  
\begin{align*}
d_1 &= (\nu_H - \dim \SO/2) - (\nu_{M_H} - \dim \SO'/2) + \vD_Q/2,  \\
d_2 &= (\dim \SO - \dim \SO') /2 + \vD_Q/2.  
\end{align*}
as in 5.3, we can apply Lemma 5.4 for $Q$. 
\par
Since $\SO' \subset X^M$, we have $V \subset H \times^{Q_H}(\eta_Q\iv(X^M))$.
The lower row of the diagram (6.4.4) gives a diagram, where 
$\xi_1, \eta_1$ are restrictions of $\xi, \eta$ on $H \times (\SO \cap \eta_Q\iv(\SO'))$.  
\begin{equation*}
\begin{CD}
H \times^{Q_H}(\SO \cap \eta_Q\iv(\SO')) @<\xi_1<<  H \times (\SO \cap \eta_Q\iv(\SO'))
                @>\eta_1>>  \SO'. 
\end{CD}
\end{equation*}  
By using a similar argument as in 6.4, we obtain an $H$-equivariant 
simple local system $\wt\SE'$ on $V$ from  an $H$-equivariant simple local system $\SE'$ 
on $\SO'$. 
On the other hand, $V'$ is a single $H$-orbit, and its stabilizer is  
isomorphic to $Z_{M_H}(v)U_{Q_H}$ for $v \in \SO'$. Hence 
the set of $H$-equivariant simple local systems on $V'$ is in bijection with 
the set of $M_H$-equivariant simple local systems on $\SO'$.
We denote by $\wh\SE'$ the local system on $V'$ corresponding to $\SE'$ on $\SO'$. 
Here we note that

\begin{equation*}
\tag{6.9.1}
f_2^*\wh\SE' \simeq \wt\SE'.
\end{equation*}

In fact, since both of $V, V'$ are locally trivial fibration over $H/Q_H$, 
we have the following commutative diagram from the embedding of fibres.

\begin{equation*}
\tag{6.9.2}
\begin{CD}
\SO \cap \eta_Q\iv(\SO') @>i>>  H \times^{Q_H}(\SO \cap \eta_Q\iv(\SO')) \\
@Vf_{\SO,\SO'}VV                            @VVf_2V     \\
\SO'   @ >i'>>   H \times^{Q_H}\SO',
\end{CD}
\end{equation*}
where $f_{\SO,\SO'}$ is the map defined in 6.1, and $i,  i'$ are 
inclusions of fibres. 
Since $f_2^*\wh\SE', \wt\SE'$ are both $H$-equivariant local systems, 
it is enough to show that their restrictions on $\SO \cap \eta_Q\iv(\SO')$ coincide.  
Since the restriction of $\wh\SE'$ on $\SO'$ is equal to $\SE'$, the restriction of 
$f_2^*\wh\SE'$ coincides with $f_{\SO,\SO'}^*\SE'$. 
On the other hand, it follows from the construction of $\wt\SE'$, one can check that 
the restriction of $\wt\SE'$ on $\SO \cap \eta_Q\iv(\SO')$ coincides with $\SE'$.    
Thus (6.9.1) holds. 
\par
Recall that $m_{\SE,\SE'}$ is the multiplicity of $\SE'$ in the local system 
$R^{2d_{\SO,\SO'}}(f_{\SO,\SO'})_!\SE$.
Here $d_2 = d_{\SO,\SO'}$. Since (6.9.2) is the fibre product, we see that 
the restriction of $R^{2d_2}(f_2)_!(f_1^*\SE)$ on $\SO'$ coincides with 
$R^{2d_{\SO,\SO'}}(f_{\SO,\SO'})_!\SE$. 
Hence $m_{\SE,\SE'}$ coincides with the multiplicity of $\wh\SE'$ in the local system 
$R^{2d_2}(f_2)_!(f_1^*\SE)$. 
Then by using (6.9.1) and Lemma 5.4, we obtain 

\par\medskip\noindent
(6.9.3) \ $m_{\SE,\SE'}$ coincides with the multiplicity of $\SE$ in the local system
$R^{2d_1}(f_1)_!\wt\SE'$.  
\par\medskip

Here the perverse sheaf $B_{\r'}$ on $\wt X_Q$ is constructed from $A'_{\r'}$ on $X^M$
by using the diagram (6.4.4). The local system $\wt\SE'$ on $V$ is also constructed from 
$\SE'$ on $\SO'$ by (6.4.4).  
If we write $(\SO',\SE') = (\SO'_{\r'}, \SE'_{\r'})$, 
one can check that $\SH^{2d'_{\r'} - d}B_{\r'}|_V \simeq \wt\SE'$.  
Hence by (6.9.3), $m_{\SE,\SE'}$ coincides with the multiplicity of $\SE$ 
in the local system $R^{2d_1}(\pi''|_V)_!(\SH^{2d'_{\r'}-d}B_{\r'}|_V)|_{\SO}$. 
If we write $(\SO,\SE) = (\SO_{\r}, \SE_{\r})$, we have 
$d_1 = d_{\r} - d'_{\r'}$.  
It follows from the definition of $x_{\r,\r'}$ and (6.8.7), 
that this multiplicity of $\SE$ actually coincides with $x_{\r,\r'}$.  
Hence by Proposition 6.8, we conclude that $m_{\SE,\SE'} = \lp\r, \r'\rp$. 
This proves (ii) of the theorem.
\par
We show (i).
Take $(\SO',\SE') \in \SN_M$ which belongs to $(L \subset M \cap P,\SO_L, \SE^{\dag}_1)$.
Assume that $m_{\SE,\SE'} \ne 0$.   Since 
$\SH^{2d'_{\r'} - d}\pi'_*K_{\ol\SE_1}|_V$ contains $\wt\SE'$,  the above discussion 
implies that
$\SE$ appears in the local system 
$R^{2d_1}(\pi''|_V)_!R^{2d'_{\r'}}(\pi'|_{\wt X_{0,\SO}})_!(\ol\SE_1)|_V$  
on $\SO$ with non-zero multiplicity.
This is equivalent to saying that $\SE$ appears in the local system 
$R^{2d_1 + 2d'_{\r'}}(\pi|_{\wt X_{0,\SO}})(\ol\SE_1)|_{\SO}$ with non-zero 
multiplicity (since $\dim \pi\iv(x) \cap \wt X_{0,\SO} \le d_1 + d'_{\r'}$ 
for $x \in \SO$).  It follows that $\SE$ appears in the local system 
$R^{2d_1 + 2d'_{\r'}}(\pi|_{\wt X_0})(\ol\SE_1)|_{\SO}$ with non-zero 
multiplicity, by considering the partition of $\wt X_0$ 
by locally closed pieces $\wt X_{0,\SO'}$. 
Now (i) follows from Theorem 5.2 (iii) since $d_1 + d'_{\r'} = d_{\SO}$ in the notation 
there. 
This complete the proof of Theorem 6.2. 

\para{6.10.}
By making use of the discussion in 2.3, we reformulate 
the restriction theorem in a more convenient form for the 
application. 
We keep the notation in 6.1.  Let $(\SO, \SE) \in \SN_G$ and 
$(\SO',\SE') \in \SN_M$. We fix $u \in \SO$, $v \in \SO'$. 
Let $\t \in A_H(u)\wg$ be the irreducible representation corresponding to $\SE$, 
and $\t' \in A_{M_H}(v)\wg$ the irreducible representation corresponding 
to $\SE'$. 
As in 2.3, but by replacing $L$ by $M$, we define varieties

\begin{align*}
Y_{u,v} &= \{ gZ_{M_H}^0(v)U_Q^{\th} \mid g \in H, g\iv ug \in \eta_Q\iv(v) \}, \\
\wt Y_{u,v} &= \{ g \in H \mid g\iv ug \in \eta_Q\iv(v) \}. 
\end{align*}
$Z_H(u) \times Z_{M_H}(v)$ acts on $Y_{u,v}$ by 
$(g_1, g_2) : gZ_{M_H}^0(v)U_Q^{\th} \mapsto g_1gg_2\iv Z_{M_H}^0U_Q^{\th}$. 
By the discussion in 2.3, we have 

\begin{equation*}
\dim Y_{u,v} \le s = (\dim Z_H(u) - \dim Z_{M_H}(v))/2 + \vD_Q/2.
\end{equation*}

\par\medskip\noindent
Let $I(Y_{u,v})$ be the set of irreducible components in $Y_{u,v}$ of 
dimension $s$. Then $A_H(u) \times A_{M_H}(v)$ acts on $I(Y_{u,v})$ 
as the permutations.  We denote by $\ve_{u,v}$ the corresponding 
permutation representation of $A_H(u) \times A_{M_H}(v)$. 
As a corollary to the restriction theorem, we have the following result
(cf [LS, 0.4]). 

\begin{cor}   
Assume that $(\SO,\SE) \in \SN_G$ corresponds to $\r \in \SA\wg_{\SE_1}$ and 
$(\SO',\SE') \in \SN_M$ corresponds to $\r' \in (\SA'_{\SE_1})\wg$. Then 
\begin{equation*}
\tag{6.11.1}
\lp \r|_{\SA'_{\SE_1}}, \r'\rp_{\SA'_{\SE_1}} = 
          \lp \t \otimes \t'^*, \ve_{u,v} \rp_{A_H(u) \times A_{M_H}(v)}.
\end{equation*}
\end{cor}

\begin{proof}
Assume that $I(Y_{u,v}) \ne \emptyset$.  
Then it follows from the discussion in 2.3 that 
$\dim (\SO \cap \eta_Q\iv(v)) = d_{\SO,\SO'} = (\dim \SO - \dim \SO')/2 + \vD_Q/2$. 
As in 6.1, let $m_{\SE,\SE'}$ be the multiplicity of $\SE'$ in the local system 
$R^{2d_{\SO,\SO'}}(f_{\SO,\SO'})_!\SE$.  Then $m_{\SE,\SE'}$ coincides with 
the multiplicity of ${\t'}^*$ in the $A_{M_H}(v)$-module 
$H_c^{2d_{\SO,\SO'}}(\SO \cap \eta_Q\iv(v),\SE)$.
We consider a similar diagram as in (2.3.1). We use the same notation there. 
$Z_H(u) \times Z_{M_H}(v)$ acts on $\wt Y_{u,v}$, $Z^0_H(u)\backslash \wt Y_{u,v}$, 
and $\f, \psi$ are $Z_H(u) \times Z_{M_H}(v)$-equivariant. On the other hand, 
let $\xi : \wt\SO = H/Z^0_H(u) \to \SO = H/Z_H(u)$ be the natural map. Then    
$Z_{M_H}(v)$ acts on $\xi\iv(\SO \cap \eta_Q\iv(v))$ under the isomorphism 
$\xi\iv(\SO \cap \eta_Q\iv(v)) \simeq Z_H^0(u)\wt Y_{u,v}$, and the map 
$\xi\iv(\SO \cap \eta_Q\iv(v)) \to (\SO \cap \eta_Q\iv(v))$ becomes 
$Z_{M_H}$-equivariant. Under this situation, 
we can identify $H^{2d_{\SE,\SE'}}_c(\SO \cap \eta_Q\iv(v), \SE)$ with 
the multiplicity space of $\t$ in the $A_H(u)$-module 
$H^{2d_{\SO,\SO'}}_c(\xi\iv(\SO \cap \eta_Q\iv(v)), \Ql)$. 
Moreover, this identification is compatible with 
the action of $A_{M_H}(v)$ on both cohomologies.  
It follows that $m_{\SE,\SE'}$ coincides with the multiplicity of $\t\otimes \t'^*$ 
in the $A_H(u) \times A_{M_H}(v)$-module 
$H^{2d_{\SO,\SO'}}_c(\xi\iv(\SO \cap \eta_Q\iv(v)), \Ql)$. 
From the discussion in 2,3, this $A_H(u) \times A_{M_H}(v)$-module 
is isomorphic to the permutation representation $\ve_{u,v}$ on the set $I(Y_{u,v})$.  
Thus the corollary follows from Theorem 6.2.
\end{proof}

\par\bigskip
\section{ Unipotent orbits}

\para{7.1.}
We follow the notation in 1.7.  Since $\Fg^-$ is the set of 
self-adjoint matrices in $\Fg = \Fg\Fl_N$, and $G^{\io\th}$ is 
the set of non-degenerate self-adjoint matrices in $\Fg$, the map 
$x \mapsto x-1$ gives an isomorphism between $G^{\io\th}\uni$ and $\Fg^-\nil$.
This map is clearly $H$-equivariant.  In this way, we have a natural bijection 
between the set of $H$-orbits in $G^{\io\th}\uni$ and the set of $H$-orbits in 
$\Fg^-\nil$.  
\par
For an integer $m \ge 0$, let $\SP_m$ be the set of partitions 
$\la = (\la_1, \dots, \la_r)$ with $\la_r \ge 0$ such that $\sum_i \la_i = m$. 
$m = |\la|$ is called the size of $\la$, and the maximal number $k = \ell(\la)$ 
such that $\la_k \ne 0$ is called the length of 
$\la$.  
For $\la = (\la_1, \dots, \la_r)$, put 
\begin{equation*}
\tag{7.1.1}
n(\la) = \sum_{i = 1}^r(i-1)\la_i.
\end{equation*}

It is known that the set of unipotent conjugacy classes in $G = GL_N$ is in bijection 
with $\SP_N$, via Jordan normal form. We denote by $\OO_{\la}$ the unipotent class 
in $G$ corresponding to $\la \in \SP_N$.  Similarly, the set of nilpotent orbits 
in $\Fg\nil$ is in bijection with $\SP_N$.  We also denote by $\OO_{\la}$ the 
corresponding nilpotent orbit in $\Fg\nil$, if there is no fear of confusion. 
For $x \in \OO_{\la} \subset \Fg\nil$, we define a Jordan basis 
$\{ v_{i,j} \mid  1 \le i \le \ell(\la), 1 \le j \le \la_i \}$ of $V$ with respect to $x$ by 
$xv_{i,j} = v_{i, j-1}$ if $j \ge 2$, and $xv_{i,1} = 0$.  
The following result gives a parametrization of $G^{\th}$-orbits in $G^{\io\th}\uni$. 
\begin{lem}  
For any $\la \in \SP_N$, $\OO_{\la} \cap G^{\io\th}\uni \neq \emptyset$. Any two elements 
in $\OO_{\la} \cap G^{\io\th}\uni$ is conjugate under $G^{\th}$. 
Thus the set of $G^{\th}$-orbits in $G^{\io\th}\uni$ is in bijection 
with $\SP_N$. 
\end{lem}   

\begin{proof}
It is enough to prove the corresponding fact for $\Fg^-\nil$. 
The argument below is in part due to [CVX1, Lemma 2.2]. 
First assume that $\OO_{\la} \cap \Fg^-\nil \ne \emptyset$. 
Take $x \in \OO_{\la} \cap \Fg^-\nil$, with $\la = (\la_1, \dots, \la_r)$.  
We show that there exists a Jordan basis $\{ v_{i,j} \}$ of $V$ with respect to $x$
satisfying the following property;
\begin{equation*}
\tag{7.2.1}
\lp v_{i,j}, v_{i',j'} \rp = \begin{cases}
                       1   &\quad\text{ if $i = i'$ and $j + j' = \la_i + 1$, } \\
                       0   &\quad\text{ otherwise.}
                              \end{cases} 
\end{equation*}
Since $\lp \Ker x^{\la_1-1}, \Im x^{\la_1-1}\rp = 0$, 
there exists $v_1 \in V$ such that $\lp v_1, x^{\la_1-1}v_1 \rp \ne 0$. 
We may assume that $\lp v_1, x^{\la_1-1}v_1\rp = 1$.  Then we have 
$\lp x^iv_1, x^jv_1 \rp = 0$ unless $i + j = \la_1 -1$, in which case it is equal to 1. 
Let $W_1$ be the subspace of 
$V$ spanned by $v_1, xv_1, \dots, x^{\la_1-1}v_1$.  Then $\dim W_1 = \la_1$, and the 
restriction of the bilinear form $\lp \ , \ \rp$ on $W_1$ is non-degenerate. 
Hence we can write as $V = W_1 \oplus W_1^{\perp}$, and $W_1^{\perp}$ is stable by $x$. 
The Jordan type of $x|_{W_1^{\perp}}$ is $\la' = (\la_2, \dots, \la_r)$.  
Thus by induction on $N = \dim V$, one can find a Jordan basis $\{ v_{i,j}\}$ 
satisfying the property (7.2.1). 
\par
If we take another $x' \in \OO_{\la} \cap \Fg^-\nil$, we can find a Jordan basis 
$\{ v'_{i,j}\}$ satisfying (7.2.1). The map 
$g : v_{ij} \mapsto v'_{i,j}$ determines $g \in G^{\th}$, and we have $gx = x'$.
Thus $\OO_{\la} \cap \Fg^-\nil$ consists of a single $G^{\th}$-orbit.   
\par
Next we show that $\OO_{\la} \cap \Fg^-\nil \neq \emptyset$.  
Take $x \in \OO_{\la} \subset \Fg\nil$, and choose a Jordan basis 
$\{ v_{ij}\}$ of $V$ with respect to $x$. We define a symmetric bilinear form  
$\lp \ , \ \rp'$ on $V$ 
by a similar condition as in (7.2.1).   
Then $x$ is self-adjoint with respect to this bilinear form, namely, we have
$\lp xv, w\rp' = \lp v, xw\rp'$ for any $v,w \in V$. 
By 1.7, this implies that some $G$-conjugate of $x$ is contained in $\Fg^-\nil$.  
Hence $\OO_{\la} \cap \Fg^-\nil \neq \emptyset$. 
The lemma is proved. 
\end{proof}

\para{7.3.}
We denote by $\SO_{\la}$ the $G^{\th}$-orbit in $G^{\io\th}\uni$ corresponding to 
$\la \in \SP_N$. Also by the same symbol $\SO_{\la}$, we denote the $G^{\th}$-orbit
in $\Fg^-\nil$.  For $x \in \SO_{\la} \subset \Fg^-\nil$, a Jordan basis 
$\{ v_{i,j} \}$ of $V$ with respect to $x$ satisfying the property in (7.2.1) 
is called a normal basis of $x$. 
\par
We shall compute $\dim Z_{G^{\th}}(x)$ for $x \in \SO_{\la}$. 
First we recall the well-known result in the case of $Z_G(x)$ for $x \in \OO_{\la}$. 
Take $x \in \OO_{\la} \subset \Fg\nil$, 
and put $E^x = \{ y \in \End(V) \mid xy = yx \}$.
Note that $\dim Z_G(x) = \dim E^x$.  Let $\{ v_{i,j} \}$ be a Jordan basis of $V$ with 
respect to $x$. Following [AH, Proposition 2.8], we define $y_{i_1, i_2, s} \in \End(V)$
as follows;

\begin{equation*}
y_{i_1,i_2,s}(v_{i,j}) = \begin{cases}
                         v_{i_2, j-s}  &\quad\text{ if } i = i_1, s + 1 \le j \le \la_i, \\
                         0             &\quad\text{ otherwise.}
                        \end{cases}
\end{equation*}
Then 
\begin{equation*}
\bigl\{ y_{i_1, i_2, s}\mid i_1, i_2 \le \ell(\la), 
          \max\{0, \la_{i_1} - \la_{i_2}\} \le s \le \la_{i_1} - 1\bigr\} 
\end{equation*}
gives rise to a basis of the $\Bk$-vector space $E^x$. 
Hence we have

\begin{equation*}
\tag{7.3.1}
\dim E^x = \sum_{i_1 \ge i_2}\la_{i_2} + \sum_{i_1 < i_2}\la_{i_1}
         = N + 2n(\la). 
\end{equation*}

We show a lemma. 

\begin{lem}  
For $x \in \SO_{\la} \subset G^{\io\th}\uni$, we have $\dim Z_{G^{\th}}(x) = n(\la)$. 
\end{lem}

\begin{proof}
It is enough to show the corresponding fact for $x \in \SO_{\la} \subset \Fg^-\nil$. 
Thus we have only to show that $\dim Z_{\Fg^+}(x) = n(\la)$. 
Here $E^x = Z_{\Fg}(x)$ is stable under $\th$, and we can decompose 
$E^x = (E^x)^+ \oplus (E^x)^-$, where $(E^x)^+ = Z_{\Fg^+}(x)$. 
We choose a normal basis of $x$, and define $y_{i_1,i_2,s}$ 
by using this basis. 
Then by 1.7,
\begin{equation*}
\{ y_{i_1, i_2, s} + y_{i_2, i_1, s + (\la_{i_2} - \la_{i_1})}
         \mid i_1 \le i_2, 0 \le s \le \la_{i_1} -1 \}
\end{equation*}
gives a basis of $(E^x)^-$ and 
\begin{equation*}
\{ y_{i_1, i_2, s} - y_{i_2, i_1, s + (\la_{i_2} - \la_{i_1})}
         \mid i_1 < i_2, 0 \le s \le \la_{i_1} -1 \}
\end{equation*}
gives a basis of $(E^x)^+$. It follows that 
\begin{equation*}
\dim (E^x)^+ = n(\la), \qquad \dim (E^x)^- = n(\la) + N.
\end{equation*}
The lemma follows from this. 
\end{proof}

\remark{7.5.}  
In the case where $\Bk = \BC$, the lemma follows from the general result 
of Kostant-Rallis [KR].  In fact, by [KR, Proposition 5], we have

\begin{equation*}
\tag{7.5.1}
\dim [\Fg^+, x] = \dim [\Fg, x]/2
\end{equation*} 
for $x \in \Fg^-$. Now assume that $x \in \SO_{\la}$. 
By (7.3.1), $\dim [\Fg, x] = \dim \OO_{\la} = N^2 - (N  + 2n(\la))$. 
If we notice that $\dim G^{\th} = (N^2 - N)/2$, (7.5.1) implies that 
$\dim [\Fg^+,x] = \dim \SO_{\la} = \dim G^{\th} - n(\la)$, and 
the lemma follows.   

\para{7.6.}
Take $x \in \SO_{\la} \subset \Fg^-\nil$. 
We write $\la \in \SP_N$ as $\la = (a_1^{m_1}, a_2^{m_2}, \dots, a_h^{m_h})$
with $a_1 > a_2 > \dots > a_h > 0$.
Then by using a normal basis of $x$, we have a decomposition 
of $V$ into $x$-stable subspaces
\begin{equation*}
\tag{7.6.1}
V = \bigoplus_{1 \le i \le h}\wt J_i,
\end{equation*}
where $\wt J_i$'s are mutually orthogonal, and the restriction of $x$ on $\wt J_i$ 
has the Jordan type $(a_i^{m_i})$.  
By this decomposition, we have an embedding of groups

\begin{equation*}
\tag{7.6.2}
\prod_{1 \le i \le h}G_i^{\th} \subset G^{\th},
\end{equation*}
where $G_i = GL(\wt J_i)$ is a $\th$-stable subgroup of $G$. 
One can find a subspace $J_i$ of $\wt J_i$ such that $\dim J_i = m_i$ and that 
$\wt J_i = J_i \oplus xJ_i \oplus \cdots \oplus x^{a_i-1}J_i$. 
We define a quadratic form $Q_i$ on $J_i$ by 

\begin{equation*}
\tag{7.6.3}
Q_i(v) = \lp v, x^{a_i-1}v\rp, \qquad (v \in J_i).
\end{equation*} 
Then $Q_i$ is non-degenerate on $J_i$.
Let $O_i = O(J_i)$ be the orthogonal group on $J_i$ with respect to the 
quadratic form $Q_i$. 
For each $g \in O_i$, we define a map $\wt g: \wt J_i \to \wt J_i$ by 
$\wt g(x^kv) = x^k(gv)$ for $0 \le k \le a_i-1, v \in J_i$. Then $\wt g$ commutes 
with the action of $x$ on $\wt J_i$, and $\wt g$ preserves the bilinear form 
$\lp \ , \, \rp$ on $\wt J_i$. Hence $\wt g \in Z_{G_i^{\th}}(x)$.  
It is clear that $g \mapsto \wt g$ gives a homomorphism 
$O_i \to Z_{G_i^{\th}}(x)$ whose image is isomorphic to $O_i$. 
Thus we can construct a closed subgroup of $Z_{G^{\th}}(x)$ which is isomorphic to 
$\prod_{1 \le i \le h}O_i$. 
Actually, this subgroup is the reductive part of $Z_{G^{\th}}(x)$, and, as in 
the classical case, one can show that 
\begin{equation*}
\tag{7.6.4}
Z_{G^{\th}}(x) \simeq \bigl(\prod_{1 \le i \le h}O_i\bigr) \ltimes U_1,
\end{equation*}    
where $U_1$ is a connected unipotent normal subgroup of $Z_{G^{\th}}(x)$.

\para{7.7}
Put $A_{G^{\th}}(x) = Z_{G^{\th}}(x)/Z^0_{G^{\th}}(x)$ for $x \in \SO_{\la}$. 
Then by (7.6.4), $A_{G^{\th}}(x) \simeq (\BZ/2\BZ)^h$. 
We write an element $\a \in A_{G^{\th}}(x)$ as $\a = (\a_1, \dots, \a_h)$
with $\a_i = \pm 1$, where $\a_i = -1$ is the generator of the $i$-th component 
$\BZ/2\BZ$. 
Let $A_H(x) = Z_H(x)/Z^0_H(x)$ as before.  Since $Z^0_{G^{\th}}(x) = Z^0_H(x)$, 
$A_H(x)$ is a subgroup of $A_{G^{\th}}(x)$.  It follows from the discussion in 7.6, 
we see easily that 

\begin{equation*}
\tag{7.7.1}
A_H(x) = \{ \a = (\a_1, \dots, \a_h) \in A_{G^{\th}}(x) 
                    \mid \prod_{a_i : \text{ odd }}\a_i = 1 \}.
\end{equation*} 
A partition $\la = (\la_1, \dots, \la_r)$ is called an even partition if 
$\la_i$ is even for each $i$. 
Then (7.7.1) can be written as 
\begin{equation*}
\tag{7.7.2}
A_H(x) \simeq \begin{cases}
                (\BZ/2\BZ)^h       &\quad\text{ if $\la$ is even,}  \\
                (\BZ/2\BZ)^{h-1}   &\quad\text{ if $\la$ is not even}. \\
              \end{cases}
\end{equation*}

\begin{prop}  
\begin{enumerate}
\item  Assume that $N$ is odd.  Then each $G^{\th}$-orbit $\SO_{\la}$ 
is a single $H$-orbit. 
\item
Assume that $N$ is even.  If $\la$ is not an even partition, $\SO_{\la}$ 
is a single $H$-orbit.  If $\la$ is an even partition, 
$\SO_{\la}$ splits into two $H$-orbits. 
\end{enumerate}
\end{prop}

\begin{proof}
Since $G^{\th} = H \cup H\z$ for some $\z \in G^{\th} - H$, we have 
$G^{\th}x = H x \cup H (\z x)$.
Thus $G^{\th}x = Hx$ if $\z \in Z_{G^{\th}}(x)$, 
namely, if $Z_{G^{\th}}(x) \supsetneq Z_H(x)$.  
If $N$ is odd, $-1 \in Z_{G^{\th}}(x) - Z_H(x)$ satisfies this condition. 
If $N$ is even and $\la$ is not an even partition, then 
$Z_{G^{\th}}(x) \supsetneq Z_H(x)$ by (7.7.2). Hence $\SO_{\la}$ is a single 
$H$-orbit.  Now assume that $N$ is even and $\la$ is an even partition. 
In this case, $Z_{G^{\th}}(x) = Z_H(x)$.  We have

\begin{equation*}
G^{\th}/Z_{G^{\th}}(x) \simeq (H \coprod \z H)/Z_H(x)
\simeq H/Z_H(x) \coprod \z H/Z_H(x).
\end{equation*} 
Hence $\SO_{\la}$ splits into two $H$-orbits. 
\end{proof}

\para{7.9.}
Let $\OO_{\la}$ be the regular unipotent class in $G\uni$, where $\la = (N) \in \SP_N$. 
Then $\OO_{\la}$ is open dense in $G\uni$.  Since 
$\OO_{\la} \cap G^{\io\th}\uni \neq \emptyset$, $\OO_{\la} \cap G^{\io\th}\uni$ is 
an open dense subset of $G^{\io\th}\uni$. In the case where $N$ is odd, 
$\OO_{\la} \cap G^{\io\th}\uni = \SO_{\la}$ is a single $H$-orbit. 
Thus $\SO_{\la}$ is the unique $H$-orbit in $G^{\io\th}\uni$ such that 
$G^{\io\th}\uni = \ol\SO_{\la}$.  In particular, $G^{\io\th}\uni$ is an irreducible 
variety. In the case where $N$ is even, $\SO_{\la} = \OO_{\la} \cap G^{\io\th}\uni$
splits into two $H$-orbits since $(N)$ is an even partition. Thus if we decompose 
$\SO_{\la}$ into two $H$-orbits, 
$\SO_{\la} = \SO_{\la}' \cup \SO_{\la}''$, 
$G^{\io\th}\uni = \ol\SO'_{\la} \cup \ol\SO''_{\la}$ gives a decomposition of 
$G^{\io\th}\uni$ into irreducible components. Sunning up the above arguments, we have

\begin{lem}  
Let $\la = (N) \in \SP_N$.
\begin{enumerate}
\item
Assume that $N$ is odd.  Then $G^{\io\th}\uni$ is irreducible, which is the closure 
of the $H$-orbit $\SO_{\la}$.
\item 
Assume that $N$ is even. Then $G^{\io\th}\uni$ consists of two irreducible components, 
which are the closures of $H$-orbits $\SO'_{\la}$ and $\SO''_{\la}$. 
We have $\ol\SO'_{\la} \cap \ol\SO''_{\la} = \ol\SO_{\la'}$, where $\SO_{\la'}$ is 
the $H$-orbit corresponding to $\la' = (N-1, 1) \in \SP_N$, and 
$\ol\SO_{\la'}$ is irreducible.  
In particular, we have

\begin{equation*}
\ol\SO'_{\la} - \ol \SO_{\la'} = \SO'_{\la}, \qquad
\ol\SO''_{\la} - \ol\SO_{\la'} = \SO''_{\la}. 
\end{equation*}
\end{enumerate}
\end{lem} 

\begin{proof}
It is enough to show the latter statement of (ii). 
It is well-known that $\OO_{\la'}$ is the subregular unipotent class in $G\uni$
and that $G\uni - \OO_{\la} = \ol\OO_{\la'}$.  Since 
$\ol\OO_{\la'} = \bigcup_{\mu \le \la'}\OO_{\mu}$ (here $\mu \le \la'$ is the dominance 
order on $\SP_N)$, we have $G^{\io\th}\uni - \SO_{\la} = 
\ol \SO_{\la'} = \bigcup_{\mu \le \la'}\SO_{\mu}$. (ii) follows from this. 
Note that $\SO_{\la'}$ is a single $H$-orbit since $\la'$ is not an even partition, 
so $\ol\SO_{\la'}$ is irreducible. 
\end{proof}

\para{7.11.}
Let $\la = (N) \in \SP_N$. We discuss the relationship between the regular unipotent 
orbit $\SO_{\la}$ and the varieties $X\uni$ appeared in 5.9.  We consider the case 
where $P = B, L = T$, and $\vS = T$.  Let $Y = Y_{(L,\vS)}$ and $X = \ol Y$. Then 
$X\uni$ is given by

\begin{equation*}
\tag{7.11.1}
X\uni = \bigcup_{g \in H}gU^{\io\th}g\iv,
\end{equation*}
where $U$ is the unipotent radical of $B$. 
Under the notation in 1.7, we define $x -1 \in \Fg\nil$ by 
\begin{equation*}
(x-1) : f_1 \mapsto f_2 \mapsto \cdots \mapsto f_n \mapsto e_0 \mapsto 
                 e_n \mapsto \cdots \mapsto e_2 \mapsto e_1 \mapsto 0
\end{equation*}
in the case where $N$ is odd, and by 
\begin{equation*}
(x-1) : f_1 \mapsto f_2 \mapsto \cdots \mapsto f_n \mapsto 
                 e_n \mapsto \cdots \mapsto e_2 \mapsto e_1 \mapsto 0
\end{equation*}
in the case where $N$ is even.
Then $x \in U^{\io\th} \cap \SO_{\la}$. 
In particular, $X\uni$ contains an $H$-orbit containing $x$. 
\par
Assume that $N$ is odd.  In this case $\SO_{\la}$ is a single $H$-orbit, 
and is contained in $X\uni$. Since $X\uni$ is irreducible, $X\uni$ 
is a closed subset of $G^{\io\th}\uni$, and $\ol\SO_{\la} = G^{\io\th}\uni$, we conclude 
that 
\begin{equation*}
\tag{7.11.2}
X\uni = G^{\io\th}\uni.
\end{equation*}  
\par
Assume that $N$ is even. In this case, we need to consider the pairs 
$(B, T)$ and $(B_1, T)$, where $B_1 = t_n B t_n\iv$ as in 1.9.
Put $U_1 = t_n U t_n\iv$. 
Then $X\uni$ associated to $(B,T)$ is defined as in (7.11.1), which we denote by $X^+\uni$. 
A similar variety is defined by replacing $U^{\io\th}$ by $U_1^{\io\th}$ in (7.11.1), 
which we denote by $X^-\uni$.  Note that $X^-\uni = \z X^+\uni \z\iv$ for 
$\z \in G^{\th} - H$. 
Let $\SO^+_{\la}$ be the $H$-orbit containing $x$.  Then another $H$-orbit contained in
$\SO_{\la}$ is given by $\SO^-_{\la} = \z \SO^+_{\la} \z\iv$.  
We have $\SO^+_{\la} \subset X^+\uni$ and $\SO^-_{\la} \subset X^-\uni$. 
Note that $X^{\pm}\uni$ are irreducible, closed subsets of $G^{\io\th}\uni$. 
By Lemma 7.10 (ii), $\ol\SO^+_{\la} \cup \ol\SO^-_{\la}$ gives a decomposition of 
$G^{\io\th}\uni$ into irreducible components. This implies that
\begin{equation*}
\tag{7.11.3}
X^+\uni = \ol\SO^+_{\la}, \qquad X^-\uni = \ol\SO^-_{\la},
\end{equation*}
and $G^{\io\th}\uni = X^+\uni \cup X^-\uni$ gives a decomposition into irreducible
components. 

\par\bigskip\noindent
\section{Structure of the algebra $\SA_{\SE_1}$}

\para{8.1.}
We follow the notation in 6.1. 
Let $Q$ be as in 6.1. Here we assume that 
$M_H \simeq GL_1 \times SO_{N-2}$. Hence $Q_H$ is a maximal parabolic subgroup of 
$H$, and $H/Q_H$ can be identified with the set of isotropic lines in $\BP(V)$. 
Let $\SO$ be an $H$-orbit in $G^{\io\th}\uni$ and $\SO'$ an $M_H$-orbit in $M^{\io\th}\uni$.
We fix $u \in \SO, v \in \SO'$, and consider the varieties, $Y_{u,v}, \wt Y_{u,v}$ 
as in 6.10.
Let $\wt\SO'$ be the $M^{\th}$-orbit in $M^{\io\th}$ containing $\SO'$. 
Put
\begin{equation*}
\tag{8.1.1}
\wt Y_{u,\SO'} = \{g \in H \mid g\iv ug \in \eta_Q\iv(\wt\SO')\}. 
\end{equation*}
Then $Q_H$ acts on $\wt Y_{u,\SO'}$ by $x : g \mapsto gx\iv$, and 
$\wt Y_{u,\SO'}/Q_H$ is a locally closed subset of $H/Q_H$. 
Let $\tau : \wt Y_{u,\SO'} \to \wt\SO'$ be the map defined by 
$g \mapsto \eta_Q(g\iv ug)$. Then 
$\t$ is $M_H$-equivariant, and for any $v' \in \wt\SO'$, $\tau\iv(v') \simeq \wt Y_{u,v}$.
Hence $\dim \wt Y_{u,\SO'} = \dim \wt Y_{u,v} + \dim \SO'$, and we have
\begin{equation*}
\tag{8.1.2}
\dim \wt Y_{u,\SO'}/Q_H = \dim \wt Y_{u,v} + \dim \SO' - \dim Q_H.
\end{equation*} 

\para{8.2.}
Let $\Fm = \Lie M$, and consider the subvariety $\Fm^-\nil$ of $\Fm$ on which $M_H$ acts. 
We denote by the same symbol 
$\SO$ (resp. $\SO'$) the $H$-orbit in $\Fg^-\nil$ 
(resp. $M_H$-orbit in $\Fm^-\nil$)
corresponding to  $\SO$ (resp. $\SO'$) as in Section 7. 
In the Lie algebra setting, the map 
$\eta_Q : \Fqq^{-\th} \to \Fm^{-\th}$ is defined similarly, where $\Fqq = \Lie Q$,  
and $\wt Y_{x,\SO'}$ is defined, for $x \in \SO \subset \Fg^{-\th}\nil$,  similarly 
to (8.1.1). 
We have $\wt Y_{u,\SO'} = \wt Y_{x, \SO'}$ for $x = u-1$. 
\par
By Lemma 7.2, the set of $G^{\th}$-orbits in $G^{\io\th}\uni$ is 
parametrized by $\SP_N$, $\SO$ corresponds to $\la \in \SP_N$ if $u \in \SO$ 
has Jordan type $\la$ as an element in $G\uni$.  
Similarly,  the set of $M^{\th}$-orbits in $M^{\io\th}\uni$ 
is parametrized by $\SP_{N-2}$. 
\par
Now assume that $\SO $ corresponds to $\la \in \SP_N$, and $\SO'$ corresponds to 
$\la' \in \SP_{N-2}$.  
Take $x \in \SO \subset \Fg^-\nil$. 
Put $W = \Ker x$, and let $\BP(W)^0$ be the set of isotropic lines in $\BP(W)$, 
namely, $\BP(W)^0 = \{ \lp v \rp  \in \BP(W) \mid \lp v, v \rp = 0\}$.
Note that the form $\lp\ ,\ \rp$ induces a non-degenerate symmetric bilinear form 
on $\ol V = V_1^{\perp}/V_1$ for $V_1 \in \BP(W)^0$, and $x|_{\ol V}$ is self-adjoint 
with respect to this form.  Here $\Fg\Fl(\ol V)$ is $H$-conjugate to $\Fm$, 
and $x|_{\ol V}$ gives an element in $\Fm^-\nil$ under this isomorphism.
It is easy to see that 
\begin{equation*}
\tag{8.2.1}
\wt Y_{x,\SO'}/Q_H \simeq \{ V_1 \in \BP(W )^0 \mid x|_{V_1^{\perp}/V_1} 
      \text{ : Jordan type } \la' \},
\end{equation*}    
  
\para{8.3.}
We will see the variety on the right hand side on (8.2.1) more precisely.   
Write $\la \in \SP_N$ as $\la = (a_1^{m_1}, a_2^{m_2}, \dots, a_h^{m_h})$ 
as in 7.6, and put $W^i = \Ker x \cap \Im x^{a_i-1}$. 
We have a filtration of $W$ by subspaces

\begin{equation*}
W = W^h  \supsetneq W^{h-1} \supsetneq \cdots \supsetneq W^1 \supsetneq W^0 = \{ 0\}, 
\end{equation*}
where $\dim W^i/W^{i-1} = m_i$ for $i = 1, \dots, h$. 
For $v, v' \in W^i$, put $(v,v') = \lp v, v_1'\rp$, where $v_1' \in V$ is such that 
$x^{a_i-1}v_1' = v'$. By using a normal basis of $x \in \Fg^{-}\nil$, one can  
check that this gives a well-defined symmetric bilinear form on $W^i$, which is 
identically zero on $W^{i-1}$. The induced form $( \ , \ )$ on $W^i/W^{i-1}$ gives 
a non-degenerate symmetric bilinear form on $W^i/W^{i-1}$. 
For each non-zero $v \in W^i$, let $V_1 = \lp v \rp$ be the line spanned by $v$. 
Since $v \in \Ker x$, $x$ induces a linear map $\ol x = x|_{\ol V}$ 
on $\ol V = V_1^{\perp}/V_1$.  We have 
$\ol x \in \Fg\Fl(\ol V)^{-}\nil \simeq \Fm^{-}\nil$. 
Define a variety $S_i$ of $V$ such that $W^{i-1} \subset S_i \subset W^i$ by 
\begin{equation*}
S_i = \{ v \in W^i \mid (v,v) = 0\}.
\end{equation*}

We consider the following conditions on $\la \in \SP_N$ and $\la'\in \SP_{N-2}$.
\par\medskip
\noindent
($\text{A}_i$) \ The Young diagram of 
$\la'$ is obtained from that of $\la$ by replacing one row of \\
\phantom{*****} length $a_i$ by 
a row of length $a_i-2$.
\par\medskip\noindent
($\text{B}_i$) \  The Young diagram of 
$\la'$ is obtained from that of $\la$ by replacing two rows of \\
\phantom{*****} length $a_i$ by 
two rows of length $a_i-1$ ( in this case, we assume that $m_i \ge 2$). 
\par\medskip\noindent
Moreover, we divide ($\text{A}_i$) into two cases $(\text{A}'_i)$ and $(\text{A}''_i)$. 
Here $(\text{A}'_i)$ is the case where $a_{i+1} \le a_i-2$, and $(\text{A}''_i)$ is 
the case where $a_{i+1} = a_i -1$. 
\par\medskip
The following lemma can be checked easily, by using a normal basis of $x$. 
(It is reduced to the case where $\la = (a^m)$ with $m = 1$ or 2). 

\begin{lem}   
Under the notation above, let $\la'$ be the Jordan type of 
$\ol x \in \Fg\Fl(\ol V)^{-}\nil$. 
Then we have 
\begin{enumerate}
\item  If $v \notin S_i$, $\la'$ is obtained from  $\la$ 
by the procedure $(\text{\rm A}_i)$. 
\item  If $v \in S_i$ (in this case $m_i \ge 2)$, $\la'$ is 
obtained from  $\la$ by the procedure $(\text{\rm B}_i)$. 
\end{enumerate}
\end{lem}

The following corollary is immediate from Lemma 8.4.

\begin{cor}  
$\wt Y_{x, \SO'} = \emptyset$ unless $\la'$ is obtained from $\la$ 
by $(\text{\rm A}_i)$, $(\text{\rm B}_i)$ for some $i$. 
If $\wt Y_{x,\SO'} \ne \emptyset$, then 

\begin{equation*}
\dim \wt Y_{x,\SO'}/Q_H = \begin{cases}
                    m_1 + \cdots + m_i - 1,  &\quad\text{ {\rm case} $(\text{\rm A}_i)$, } \\
                    m_1 + \cdots + m_i - 2,  &\quad\text{ {\rm case} $(\text{\rm B}_i)$. }
                  \end{cases}
\end{equation*}
\end{cor}

\para{8.6.}
Recall that $\dim Y_{u,v} \le s$, where 
$s = (\dim Z_H(u) - \dim Z_{M_H}(v))/2 + \vD_Q/2$ by 6.10.
Since we know by Lemma 7.4 that $\dim Z_H(u) = n(\la)$ if $u \in G^{\io\th}\uni$ has 
Jordan type $\la$,  
the number $s$ can be computed explicitly. 
We have the following.

\begin{prop}   
Assume that $u \in \SO$ has type $\la$, $v \in \SO'$ has type $\la'$.  
Then $Y_{u,v}$ is 
non-empty if and only if $\la'$ is obtained from $\la$ by the procedure
$(\text{\rm A}_i)$, or $(\text{\rm B}_i)$. We have 
$\dim Y_{u,v} = s$ if and only if $\la'$ is obtained from $\la$ by $(\text{\rm A}'_i)$. 
\end{prop}

\begin{proof}
Since $Y_{u,v} \ne \emptyset$ if and only if $\wt Y_{x,\SO'} \ne \emptyset$, 
the first assertion follows from Corollary 8.5. Now assume 
$Y_{u,v} \ne \emptyset$.  Since $\dim Y_{u,v} = \dim \wt Y_{x,\SO'}/Q_H$, 
$\dim Y_{u,v}$ can be computed from the formula in Corollary 8.5.
Since $M_H \simeq GL_1 \times SO_{N-2}$, $\dim Z_{M_H}(v) = n(\la') + 1$ by 
Lemma 7.4. Then we have

\begin{equation*}
\tag{8.7.1}
\dim Z_H(u) - \dim Z_{M_H}(v) = 
                 \begin{cases}
                   2(m_1 + \cdots + m_i -1) -1  &\quad\text{ case ($\text{A}'_i$), } \\
                   2(m_1 + \cdots + m_i - 1)    &\quad\text{ case ($\text{A}''_i$), } \\
                   2(m_1 + \cdots + m_i -2)     &\quad\text{ case ($\text{B}_i$). } 
                 \end{cases}
\end{equation*}
Since $\vD_Q = 1$, we have

\begin{equation*}
   s = \begin{cases}
          m_1 + \cdots + m_i -1  &\quad\text{ case ($\text{A}'_i$), } \\
          (m_1 + \cdots + m_i - 1) + 1/2  &\quad\text{ case ($\text{A}''_i$), } \\
          (m_1 + \cdots + m_i -2) + 1/2  &\quad\text{ case ($\text{B}_i$). }
       \end{cases}  
\end{equation*}

\par\medskip
The proposition follows. 
\end{proof}

\para{8.8.}
For $x \in \Fg^{-}\nil$, let $\SB_x = \{ gB_H \in H/B_H \mid 
g\iv x \in \Lie B_H \}$ be the Springer fibre of $x$. We also 
define $\SP_x = \{ gQ_H \in H/Q_H \mid g\iv x \in \Lie Q_H \}$
and a natural map $\pi_Q : \SB_x \to \SP_x$, $gB_H \mapsto gQ_H$. 
$\SP_x$ can be identified with $\BP(W)^0$.  
Thus $\SP_x$ is partitioned into locally closed pieces 
$\SP_{x,\la'} \simeq \wt Y_{x,\SO'}/Q_H$, where $\wt\SO'$ is the 
$M^{\th}$-orbit corresponding to $\la' \in \SP_{N-2}$,
and we  have a partition
of $\SB_x$ into locally closed pieces
\begin{equation*}
\tag{8.8.1}
\SB_x = \coprod_{\la' \in \SP_{N-2}}\pi_Q\iv(\SP_{x,\la'}).
\end{equation*}  

Assume that $x \in \SO$ has type $\la$.  By Lemma 8.4, 
if $\pi_Q\iv(\SP_{x,\la'}) \ne \emptyset$, then 
$\la'$ is obtained from $\la$ by the procedure given in (i) or (ii) in Lemma 8.4.
If $v \in W $ is given, $\ol x \in \Fg\Fl(\ol V)^{-}\nil$ is defined as above.  
We denote by $\ol\SB_{\ol x}$ the corresponding Springer fibre for $SO(\ol V)$.   
By applying Proposition 2.2  (ii) for the case $P = B, L = T$, we have

\begin{align*}
\tag{8.8.2}
\dim \SB_x &\le \nu_H - \dim \SO/2 + \vD_B/2 \\
           &= \dim Z_H(x)/2,
\end{align*}
where $\vD_B = [N/2] = \dim T_H$. 

The following result was proved by [CVX1, Proposition 3.1] in the case where $N$ is odd, and 
$\Bk = \BC$. The discussion below  seems to be simpler. 
(But this result will not be used in later discussions.) 

\begin{prop}  
Assume that $x \in \SO$ has type $\la$.  
Then $\dim \SB_x = \dim Z_H(x)/2$ if and only if 
$\la = (\la_1, \dots, \la_r) \in \SP_N$ satisfies the condition that 
$\la_i$ is even for $i \ge 2$ $($hence $\la_1$ is odd $($resp. even$)$ if 
$N$ is odd $($resp. even$))$. 
\end{prop}

\begin{proof}
In the notation of 8.8, the restriction of the map $\pi_Q : \SB_x \to \SP_x$ 
on $\pi_Q\iv(\SP_{x,\la'})$ is surjective, and each fibre is isomorphic to 
$\ol\SB_{\ol x}$.  
Hence 
\begin{equation*}
\dim \pi_Q\iv(\SP_{x,\la'}) = \dim \SP_{x,\la'} + \dim \ol\SB_{\ol x},
\end{equation*}  
and $\dim \SB_x = \max \{ \dim \pi_Q\iv(\SP_{x,\la'}) \mid \la' \in \SP_{N-2} \}$.
$\dim \SP_{x,\la'}$ is given in Corollary 8.5. If we denote $\ol H = SO(\ol V)$, then  
$\dim \ol\SB_{\ol x} \le \dim Z_{\ol H}(\ol x)/2$, and 
$\dim Z_{\ol H}(\ol x) = n(\la')$. 
Thus by using a similar argument as in the proof of Proposition 8.7, 
we see that $\dim \SB_x = \dim Z_H(x)/2$ if and only if there exists $i$ such that    
$\la' \in \SP_{N-2}$ is obtained from $\la$ by ($\text{A}'_i$) and that 
$\dim \ol\SB_{\ol x} = \dim Z_{\ol H}(\ol x)/2$.
The proposition follows from this by induction on $N$.   
\end{proof}

\para{8.10.}
$Z_H(x)$ acts on $\wt Y_{x, \SO'}$ by $h : g \mapsto hg$, hence induces an 
action of $Z_H(x)$ on $\wt Y_{x,\SO'}/Q_H$. 
Put $C_i = W^i - S_i$, and let $p(C_i), p(S^*_i)$ be the image of 
$C_i, S_i^* = S_i - \{0\}$ under the map $p : W^i - \{0\} \to \BP(W)$. Then 
by Lemma 8.4, the variety $\wt Y_{x,\SO'}/Q_H$ is isomorphic to 
$p(C_i)$ or $p(S^*_i)$ if $\SO'$ corresponds to $\la'$, which is obtained from $\la$ 
as in Lemma 8.4 for some $i$.  
The corresponding action of $Z_H(x)$ on $p(C_i)$ or $p(S^*_i)$ is 
given as follows; 
The natural action of $Z_H(x)$ on $V$ leaves $W^i$ invariant.  Moreover, 
the bilinear form $( \ , \ )$ on $W^i$ is $Z_H(x)$-invariant.  Hence 
$S_i$ and $C_i$ are stable by the action of $Z_H(x)$.  The induced action on 
$p(S^*_i)$ or $p(C_i)$ coincides 
with the action of $Z_H(x)$ on $\wt Y_{x,\SO'}/Q_H$. We show the following lemma.

\begin{lem}  
\begin{enumerate}
\item
The action of $Z_H(u)$ on $\wt Y_{u,\SO'}/Q_H$ is transitive.  
\item 
The action of $Z_H(u) \times Z_{M_H}(v)$ on $Y_{u,v}$ is transitive.
\end{enumerate}
\end{lem}  

\begin{proof}
(i)  It is enough to show that $Z_H(x)$ acts transitively on $p(C_i)$ or $p(S^*_i)$.
Here we only show the case of $p(C_i)$, since this is the case needed in  later 
discussions.  The case of $p(S^*_i)$ is proved similarly. 
Put $a = \dim W^i, b = m_i = \dim W^i/W^{i-1}$ and $r = [b/2]$.  
Take $v \in C_i$.  Then there exists a basis $\{ v_1, \cdots, v_a\}$ of $W^i$ such that 
$\{ v_{b +1}, \cdots, v_a\}$ gives a basis of $W^{i-1}$, satisfying the condition 
that $v = v_1 + v_b$ and that, for any $w = \sum_jx_jv_j \in W^i$, 
\begin{equation*}
(w,w) = x_1x_b + x_2x_{b -1} + \cdots + x_rx_{b -r} + \d x_{r+1}^2,  
\end{equation*} 
where $\d = 1$ if $b$ is odd and $\d = 0$ if $b$ is even.  
Let $C_i^0$ be the set of $w = \sum_jx_jv_j \in C_i$ such that $x_1 \ne 0$. Thus 
$C_i^0$ is an open dense subset of $C_i$ containing $v$. We show 
\par\medskip\noindent
(8.11.1) \ For any $w \in C_i^0$, there exists $\f_w \in Z_H(x)$ such that 
$\f_w(v) = \g\iv w$, where we put $(w,w) = \g^2 \in \Bk^*$.  
\par\medskip
In fact, for a given $w \in C_i^0$, we choose $\g$ such that $(w,w) = \g^2$, 
and define a map $\f'_w : W^i \to W^i$ by 

\begin{equation*}
\begin{cases}
v \mapsto \g\iv w,  \\
v_b \mapsto x_1\iv \g v_b, \\ 
v_j \mapsto v_j - x_{b-j+1}v_b  \quad\text{ if $1 < j < b$, } \\
v_j \mapsto v_j                 \quad\text{ if $j > b$}.
\end{cases}
\end{equation*} 

It is easy to check that $\f'_w$ preserves the symmetric bilinear form 
$( \ ,\ )$ on $W^i$, and the restriction of $\f'_w$ on $W^{i+1}$ is identity. 
Moreover $\det \f'_w = 1$. 
For $\la = (\la_1, \dots, \la_r)$, let $\{ w_{i,j} \}$
be a normal basis of $x$. 
For any $i$, we define 
$V^i$ as the subspace of $V$ spanned by $\{ w_{k,j} \mid 1 \le k \le i \}$. Then 
$V^i$ is an $x$-stable subspace of $V$ such that $V^i \cap \Ker x = W^i$. 
The linear map $\f'_w$ on $W^i$ can be extended in a canonical way to a linear
automorphism  $\f''_w$ on $V^i$ which commutes with $x$, and preserves the form 
$\lp \ , \ \rp$. We extend $\f''_w$ to a linear map $\f_w$ on $V$ by defining 
$w_{k,j} \mapsto w_{k,j}$ for any $k > a$.  Then $\f_w \in Z_H(x)$, and satisfies 
the condition in (8.11.1).  Thus (8.11.1) was proved. 
\par
Now (8.11.1) shows that the action of $Z_H(x)$ on $p(C_i)$ is transitive on 
$p(C_i^0)$. For any $v \in C_i$, we can find such $C_i^0$ containing $v$.  Hence 
such $p(C_i^0)$ covers whole $p(C_i)$.  Since $p(C_i^0)$ is open dense in 
$p(C_i)$, we conclude that $Z_H(x)$ acts transitively on $p(C_i)$.  Thus 
(i) is proved. 
\par
(ii) \ It follows from (i) that the action of $Z_H(u)$ on $Y_{u,v}$ is transitive 
modulo $Q_H$, namely, for any $gZ_{M_H}(v)^0U_Q^{\th} \in Y_{u,v}$, there exists 
$z \in Z_H(u)$ and $q \in Q_H$ such that 
$zgqZ_{M_H}(v)^0U_Q^{\th} = Z_{M_H}(v)^0U_Q^{\th}$. Here we may replace $q$ 
by $m \in M_H$. Then the definition  of $Y_{u,v}$ implies that 
$m\iv(g\iv ug)m \in \eta_Q\iv(v)$.  But this implies that $m \in M_H(v)$.  Hence 
$Z_H(u) \times Z_{M_H}(v)$ acts transitively on $Y_{u,v}$.  (ii) is proved.  
\end{proof}

\para{8.12.}
Take $x \in \SO \subset \Fg^-\nil$, where the Jordan type of $x$ is $\la$. 
The structure of the group $A_H(x)$ is described in 7.7. 
Now assume that $x \in \Fqq\nil$ and let $\ol x \in \Fm^-\nil$ 
be the image of $x$ under the map $\eta_Q : \Fqq \to \Fm$.  Assume that 
the Jordan type of $\ol x$ is $\la'$. 
$A_{M_H}(\ol x)$ is described similarly.  
Put $Z_{Q_H}(x) = Z_H(x) \cap Q_H$, and let $A_{Q_H}(x)$ be the image 
of $Z_{Q_H}(x)$ under the natural map $Z_H(x) \to A_H(x)$. The projection 
$Q_H \to M_H$ induces a map $Z_{Q_H}(x) \to Z_{M_H}(\ol x)$, which gives  
a natural homomorphism  $\vf : A_{Q_H}(x) \to A_{M_H}(\ol x)$.  
\par 
Now assume that $\la'$ is obtained from $\la$ by $(\text{\rm A}_i')$.
In this case, one can check that $A_{Q_H}(x) = A_H(x)$.  Hence 
$\vf$ gives rise to a map $\vf : A_H(x) \to A_{M_H}(\ol x)$. 
The image $\Im \vf \subset A_{M_H}(\ol x)$ is described as follows; 
we consider two cases, according to the case where 
$a_{i+1} = a_i - 2$ or $a_{i+1} < a_i -2$ (note that $a_{i+1} \le a_i -2$ 
by the assumption $(\text{A}'_i)$.)
\par\medskip\noindent
Case I. \ $a_{i+1} = a_i-2$.  In this case, 
$\la' = (a_1^{m_1}, \dots, a_i^{m_i-1}, a_{i+1}^{m_{i+1}+1}, \dots, a_h^{m_h})$.   
Then $A_{M_H}(\ol x) = \{ \a' = (\a_1', \dots \a'_i, \a'_{i+1}, \dots, \a'_h) \}$ 
(here we ignore $\a_i'$ if $m_i = 1$).
$\vf$ is given by $\a_j \mapsto \a_j'$ for $j \ne i$, and $\a_i \mapsto \a'_{i+1}$. 
We have $\Im \vf = \{ \a' \in A_{M_H}(\ol x) \mid \a'_i = 1\}$, 
which is an index two subgroup of $A_{M_H}(\ol x)$ if $m_i \ge 2$, and $\vf$ is 
surjective if $m_i = 1$.
\par\medskip\noindent
Case II. \ $a_{i+1} < a_i -2$.  In this case, 
$\la' = (a_1^{m_1}, \dots, a_i^{m_i-1}, a_i-2, a_{i+1}^{m_{i+1}}, \dots, a_h^{m_h})$.  
Then $A_{M_H}(\ol x) = 
   \{ \a' = (\a_1', \dots, \a'_i, \a'_*, \a'_{i+1}, \dots, \a'_h)\}$.
The map $\vf$ is given by $\a_j \mapsto \a_j'$ if $j \ne i$, and $\a_i \mapsto \a'_*$. 
We have $\Im \vf = \{ \a' \in A_{M_H}(\ol x) \mid \a'_i = 1\}$, which is 
an index 2 subgroup of $A_{M_H}(\ol x)$ if $m_i \ge 2$, and $\vf$ is an isomorphism 
if $m_i = 1$.

\para{8.13.}
We now  consider the variety 
$Y_{u,v} = Y_{x, \ol x}$ under the assumption that  
$\la'$ is obtained from $\la$ by $(\text{\rm A}_i')$. 
Recall that $I(Y_{x,\ol x})$ is the set of 
irreducible components of $Y_{x,\ol x}$ of dimension $s$.
By Proposition 8.7, we have $\dim Y_{x,\ol x} = s$, hence $I(Y_{x,\ol x}) \ne \emptyset$. 
By Lemma 8.11, $Z_H(x) \times Z_{M_H}(\ol x)$ acts transitively 
on $Y_{x,\ol x}$, hence $A_H(x) \times A_{M_H}(\ol x)$ acts transitively 
on $I(Y_{x,\ol x})$. 
In order to determine the permutation representation $\ve_{x,\ol x}$ of 
$A_H(x) \times A_{M_H}(\ol x)$ on $I(Y_{x,\ol x})$, it is enough to determine
an isotropy subgroup of $A_H(x) \times A_{M_H}(\ol x)$.  
By the definition of $Y_{x,\ol x}$, we can write as 
\begin{equation*}
Y_{x, \ol x} = \bigcup_{\a \in A_H(x), \a' \in A_{M_H}(\ol x)}
\a Z^0_H(x)U_Q^{\th}Z^0_{M_H}(\ol x)\a'.
\end{equation*}
Hence $Z_H^0(x)U_Q^{\th}Z_{M_H}^0(\ol x)$ is an irreducible component in $Y_{x,\ol x}$.
We denote by $E$ the stabilizer of the component $Z_H(x)^0U_Q^{\th}Z_{M_H}(\ol x)^0$ in 
$A_H(x) \times A_{M_H}(\ol x)$.  We have the following lemma. 

\begin{lem}  
Under the notation as above, we have

\begin{equation*}
E = \begin{cases}
       \{ (\a, \a') \in A_H(x) \times A_{M_H}(\ol x) \mid 
                \a_j = \a'_j \ (j \ne i), \ \a_i = \a'_{i+1} \},
          &\quad\text{ \rm {case (I)}, } \\
       \{ (\a, \a') \in A_H(x) \times A_{M_H}(\ol x) \mid 
                \a_j = \a'_j \ (j \ne i), \ \a_i = \a'_* \},
          &\quad\text{ \rm{case (II)}. } \\
    \end{cases}
\end{equation*}
\end{lem}

\begin{proof}
Since $A_H(x) = A_{Q_H}(x)$, we have, for 
$(\a, \a') \in A_H(x) \times A_{M_H}(\ol x)$,  
\begin{equation*}
\a A_H^0(x)U_Q^{\th}A_{M_H}^0(\ol x)\a' = A_H^0(x)U_Q^{\th}A_{M_H}^0(\ol x)\vf(\a)\a'.
\end{equation*} 
Hence the lemma follows from the discussion in 8.12.
\end{proof}

\para{8.15.}
We denote an element $\tau$ of $A_{G^{\th}}(x)\wg \simeq (\BZ/2\BZ)^h$ as
$\tau = (\tau_1, \dots, \tau_h) \in (\BZ/2\BZ)^h$, with $\tau_i = \pm 1$. 
Let $\tau_i = (-1)_i$ be the character of $A_{G^{\th}}(x)$ corresponding to 
the generator of the $i$-th component $\BZ/2\BZ$.  Then 
$A_H(x)\wg$ is given as 
\begin{equation*}
A_H(x)\wg = A_{G^{\th}}(x)\wg /\prod_{a_i \text{ : odd}}(-1)_i,
\end{equation*}   
hence, we may identify $A_H(x)\wg$ with the following subset of $A_{G^{\th}}(x)\wg$, 
\begin{equation*}
\tag{8.15.1}
A_H(x)\wg = \{ \tau = (\tau_1, \dots, \tau_h) \in (\BZ/2\BZ)^h \mid \tau_{i_0} = 1\},
\end{equation*}
where $i_0$ is the index such that $a_{i_0}$ is the largest odd number among 
$a_1, \dots, a_h$. 
\par
Let $\la, \la'$ be as in 8.12, and we return to the setting that 
$u \in \SO \subset G^{\io\th}\uni$, $v \in \SO' \subset M^{\io\th}\uni$. 
Then we can write 
$\tau \in A_H(u)\wg$ as $\tau = (\tau_1, \dots, \tau_h)$. As in 8.12, according to the 
case (I) or (II), we can write $\tau' \in A_{M_H}(v)\wg$ as 
$\tau' = (\tau_1',\dots, \tau'_i, \tau'_{i+1}, \dots, \tau'_h)$ or 
$\tau' = (\tau'_1, \dots, \tau'_i, \tau'_*, \tau'_{i+1}, \dots, \tau'_h)$. 
Note that this notation is compatible with the identification in (8.15.1). 
We define a subset $D$ 
of $A_H(u)\wg \times A_{M_H}(v)\wg$ by  

\begin{equation*}
   D = \begin{cases}
         \{(\tau, \tau') \in A_H(u)\wg \times A_{M_H}(v)\wg \mid 
                 \tau_j = \tau'_j \ (j \ne i), \ \tau_i = \tau'_{i+1} \}, 
             &\quad\text{ case (I), } \\
         \{(\tau , \tau') \in A_H(u)\wg \times A_{M_H}(v)\wg \mid
                \tau_j = \tau'_j \ (j \ne i), \ \tau_i = \tau'_* \}, 
              &\quad\text{ case (II). }
        \end{cases} 
\end{equation*}
\par
Let $\ve_{u,v}$ be the permutation representation of $A_H(u) \times A_{M_H}(v)$ on 
$I(Y_{u,v})$ as before.  It follows from the previous discussions, we have

\begin{prop}  

$A_H(u) \times A_{M_H}(v)$-module $\ve_{u,v}$ can be decomposed into 
irreducible modules as 

\begin{equation*}
\ve_{u,v} \simeq \bigoplus_{(\tau, \tau') \in D} \tau \otimes \tau'.
\end{equation*}
\end{prop}

\para{8.17.}
We generalize the notation of $\t \in A_H(u)\wg$ as follows;
let $u \in \SO$ with type $\la$.  
Write $\la \in \SP_N$ as $\la = (\la_1, \dots, \la_N)$ with 
$\la_1 \ge \la_2 \ge \cdots \ge \la_N \ge 0$.  Correspondingly, 
we consider the symbol $\tau = (\tau_1, \dots, \tau_N)$ satisfying 
the following properties, 
\par\medskip\noindent
(i) \ $\tau_i = \pm 1$ for any $i$, and $\tau_i = 1$ if $\la_i = 0$, \\
(ii) \ $\tau_i = \tau_j$ if $\la_i = \la_j$,  \\
(iii) \ $\tau_i = 1$ for $i$ such that $\la_i$ is the largest odd number 
among $\la_1, \dots, \la_N$.  
\par\medskip
The set of such symbols $\tau$ is in bijection with $A_H(u)\wg$.     
If $\SE$ is a local system on $\SO$ corresponding to $\tau \in A_H(u)\wg$, 
we denote it by $\SE = \SE_{\tau}$. 

\para{8.18.}
Take $(L \subset P, \SO_L, \SE^{\dag}_1) \in \ScS_G$. We assume that 
$L_H \simeq (GL_1)^a \times SO_{N_0}$ with $N_0 + 2a = N$. 
Assume that the Jordan type of $\SO_L$ is $\nu = (\nu_1, \dots, \nu_{N_0})$, 
and $\SE^{\dag}_1 = \SE^{\dag}_{\s}$ with 
$\s = (\s_1, \dots, \s_{N_0})$.  
For $i = 0, 1, \dots, a$, let $P^{(i)}$ be the $\th$-stable parabolic subgroup of $G$ containing 
$P$, and $L^{(i)}$ the $\th$-stable Levi subgroup of $P^{(i)}$ containing $L$ such that 
$L_H^{(i)} \simeq (GL_1)^{a-i} \times SO_{N_0 + 2i}$. 
Hence $P^{(a-1)}_H = Q_H$ is the maximal parabolic subgroup of $H$, and $P^{(0)}_H = P_H$.
We consider $\pi: \wt X \to X$ with respect to $(L \subset P,\SO_L, \SE^{\dag}_1)$ and 
consider the semisimple perverse sheaf $K = \pi_*K_{\ol\SE_1}$ on $X$. 
We can define a similar complex $K^{(i)}$ on $X^{(i)} \subset (L^{(i)})^{\io\th}$, 
by replacing $G$ by $L^{(i)}$. 
The following lemma is a generalization of Proposition 8.9.

\begin{lem}  
Let $(\SO,\SE) \in \SN_G$, and assume
that $(\SO,\SE)$ belongs to $(L \subset P, \SO_L, \SE^{\dag}_1)$. 
Then the Jordan type $\la = (\la_1, \dots, \la_N)$ of $\SO$ satisfies the condition 
\begin{equation*}
\tag{8.19.1} 
\la_i - \nu_i \in 2\BZ_{\ge 0} \quad\text{ for each $i$}.   
\end{equation*}
(Here we write $\nu = (\nu_1, \dots, \nu_N)$ by putting $\nu_i = 0$ for $i > N_0$.)
\end{lem}

\begin{proof}
We can formulate a similar property as in the lemma by replacing $G$ by $L^{(i)}$.
We shall prove the claim  of the lemma by induction on $i$.  So we may assume that 
the claim holds for $i = a-1$, i.e., in the case where $P^{(i)}_H$ is the 
maximal parabolic subgroup $P^{(a-1)}_H = Q_H$. Let $\r \in \SA_{\SE_1}\wg$ be 
the irreducible character corresponding to $(\SO, \SE)$ under the generalized 
Springer correspondence. Let $\SA'_{\SE_1}$ be the subalgebra of $\SA_{\SE_1}$ 
associated to $Q_H$, and take $\r' \in (\SA'_{\SE_1})\wg$ such that 
$\lp \r, \r' \rp \ne 0$.  Let $(\SO', \SE') \in \SN_{M}$ be the 
pair corresponding to $\r'$ under the generalized Springer correspondence. 
Take $u \in \SO, v \in \SO'$, and write $\SE = \SE_{\tau}, \SE' = \SE'_{\tau'}$ 
for $\tau \in A_H(u)\wg, \tau' \in A_{M_H}(v)\wg$. Then by Corollary 6.11, 
$\tau\otimes \tau'$ appears in the decomposition of $\ve_{u,v}$.  In particular, 
$I(Y_{u,v})$ is non-empty. Hence by Proposition 8.7, the Jordan type $\la'$ of 
$\SO'$ is obtained from $\la$ 
by the procedure $(A'_i)$ for some $i$.    
By induction hypothesis, $(\SO',\SE')$ satisfies the claim of the lemma. Hence 
$\la$ also satisfies the claim.  The lemma is proved. 
\end{proof}

\begin{prop}   
Let $\SO_0$ be a unique $H$-orbit in $X\uni$ such that 
$\SO_0 \cap \eta_P\iv(\SO_L)$ is open dense in $\eta_P\iv(\SO_L)$.  
\begin{enumerate}
\item 
$\SO_0$ is the unique open dense orbit contained in $X\uni$. 
\item 
The Jordan type of $\SO_0$ is given by 
$\la = (2a + \nu_1, \nu_2, \dots, \nu_N)$.
\end{enumerate}
\end{prop}

\begin{proof}
(i) \ Put $\SO = \SO_0$.  
By the assumption, $\ol{\eta_P\iv(\SO_L)} \subset \ol\SO$. 
Since $X\uni$ is a union of $H$-conjugates of $\eta_P\iv(\ol\SO_L)$, 
$X\uni \subset \ol\SO$.  As $\SO \subset X\uni$, we have $\ol\SO = X\uni$. 
Since $X\uni$ is irreducible, $\SO$ is uniquely determined. 
\par
(ii) \ By Lemma 5.10, $\dim X\uni = 2\nu_H - 2\nu_{L_H} + \dim \SO_L + a$ 
(here $a = \vD_P$). 
Hence $\dim \SO = \dim H - \dim L_H + \dim \SO_L + a$.
But $\dim L_H - \dim \SO_L -a = \dim Z_{\ol H}(v)$ for $v \in \SO_L$, where 
$\ol H = SO_{N_0}$. 
Take $u \in \SO$, and put $\la = (2a + \nu_1, \nu_2, \dots, \nu_N)$.  
If we note that $n(\la) = n(\nu)$, by using Lemma 7.4 we have
\begin{equation*}
\tag{8.20.1}
\dim Z_H(u) = \dim Z_{\ol H}(v) = n(\nu) = n(\la).
\end{equation*} 
On the other hand, let $d_{\SO}$ be as in (5.2.2).  The previous computation 
shows that $d_{\SO} = 0$.  Thus $R^{2d_{\SO}}f_!\ol\SE_1 = f_!\ol\SE_1$.  
Since $\eta_P\iv(\SO_L) \cap \SO$ is open dense in $\eta\iv(\SO_L)$, we see that 
$f_!\ol\SE_1|_{\SO} \ne 0$.  Hence by Theorem 5.2 (iii), the pair $(\SO, \SE)$ 
belongs to $(L \subset P,\SO_L, \SE^{\dag}_1)$ for some local system $\SE$ on $\SO$. 
Then by Lemma 8.19, the Jordan type of $\SO$ satisfies the condition (8.19.1).
One can check that if $\la' \ne \la$ satisfies the condition (8.19.1), 
then $n(\la') > n(\la)$.  This implies, by (8.20.1), that the Jordan type of $\SO$ 
is equal to $\la$.  The proposition is proved.  
\end{proof}

We can now prove the following theorem , which is a counter-part of 
[L1, Theorem 9.2] in the symmetric space case.  

\begin{thm}  
Let the notations be as in 8.18.
\begin{enumerate}
\item 
The algebra $\SA_{\SE_1}$ is isomorphic to the group algebra 
$\Ql[S_a]$. 
\item
Let $\SE_0 = \SE_{\t}$ be the local system on $\SO_0$ defined by 
$\tau = (\tau_1, \dots, \tau_N)$ with 
$\tau_i = \s_i$ for $i = 1, \dots, N_0$, and $\tau_i = 1$ for $i > N_0$.  
Then $(\SO_0,\SE_0)$ belongs to $(L \subset P,\SO_L, \SE^{\dag}_1)$.  
$\SE_0$ is the unique local system on $\SO_0$ such that 
$(\SO_0,\SE_0)$ belongs to $(L \subset P,\SO_L, \SE^{\dag}_1)$.  
\item 
Under the isomorphism $\SA_{\SE_1} \simeq \Ql[S_a]$ in {\rm (i)}, $(\SO_0,\SE_0)$ corresponds 
to the unit representation of $S_a$. 
\end{enumerate}
\end{thm} 

\begin{proof}
The statement of the theorem can be formulated by replacing $G$ by $L^{(i)}$. 
In the case where $i = 0$ the claim of the theorem is trivial.  
By induction on $i$, we may assume that 
the claim holds for 
$i = a-1$, i.e., for $L_H^{(a-1)} = M_H$.   
Let $(\SO_0',\SE_0')$ be the pair in $\SN_M$ defined similarly to $(\SO_0, \SE_0)$ 
for $H$. Let $\la'$ be the Jordan type of $\SO'_0$.  
From the proof of Proposition 8.20, we know that there exists 
a pair $(\SO_0, \SE)$ which belongs to $(L \subset P, \SO_L, \SE^{\dag}_1)$ for some 
local system $\SE$ on $\SO_0$.   
Let $\r \in \SA\wg_{\SE_1}$ be the irreducible character corresponding 
to $(\SO_0,\SE)$. Also define $\r' \in (\SA'_{\SE_1})\wg$ as the character 
corresponding to $(\SO'_0,\SE'_0)$. Let $\r'' \in (\SA'_{\SE_1})\wg$ be a character 
appearing in the restriction of $\r$ on $\SA'_{\SE_1}$, and let $(\SO'', \SE'')$
be the corresponding pair in $\SN_{M_H}$. 
By Corollary 6.11 and Proposition 8.7, the Jordan type $\la''$ of $\SO''$ is 
obtained from $\la$ by the procedure $(\text{A}'_i)$ for some $i$.  
On the other hand, by Lemma 8.19, $\la''$ satisfies the condition (8.19.1). 
It follows that the multiplicity $\lp \r, \r''\rp = 0$ unless 
$\la'' = \la'$. Now assume that $\la'' = \la'$, i.e., $\SO'' = \SO'_0$.
In this case, by our assumption, we have $(\SO'',\SE'') = (\SO'_0, \SE'_0)$, 
hence $\r'' = \r'$.  
We write 
$\SE = \SE_{\xi}$ for $\xi \in A_H(u)\wg$, and $\SE'_0 = \SE'_{\tau'}$ for 
$\tau' \in A_{M_H}(v)$ with $v \in \SO'_0$. 
By Corollary 6.11, $\lp\r, \r'\rp $ coincides with the multiplicity of $\xi\otimes \tau'$
in $\ve_{u,v}$. 
Since $\lp\r, \r'\rp \ne 0$, by Proposition 8.16, 
we see that $\xi = \tau$ and that $\lp \r, \r'\rp = 1$. 
This shows that $\SE = \SE_0$, and (ii) holds. 
\par
The above discussion shows that the restriction of $\r$ on $\SA'_{\SE_1}$ coincides with 
$\r'$.  Since $\r'$ is a one-dimensional representation by our assumption, 
we see that $\r$ is one-dimensional.  
Recall that $\SA_{\SE_1}$ is isomorphic to a twisted group algebra $\Ql[\SW_{\SE_1}]$
(see 3.6). By making use of the one-dimensional representation $\r$, in a similar way 
as in the proof of Theorem 9.2 in [L1], we can construct an algebra isomorphism 
$\SA_{\SE_1} \isom \Ql[\SW_{\SE_1}]$, where $\r$ corresponds to the unit representation 
on $\SW_{\SE_1}$.  But in the definition of $\SW_{\SE_1}$ in 3.6, $n \in N_H(L_H)$ 
induces a trivial automorphism  $\ad(n)$ on  $L_H$, 
hence acts trivially on $\SO_L$ if the image of $n$ in $\SW$ is contained in $\SW_1$.  
It follows that $\SW_{\SE_1} \simeq S_a$. 
Thus (i) and (iii) holds.   The theorem is proved.
\end{proof}

\remark{8.22.} 
In [L1, Proposition 9.5], the unipotent class corresponding to the sign representation
of $N_G(L)/L$ was determined.  In the symmetric space case, however, the behavior of the 
$H$-orbit corresponding to the sign representation of $S_a$ is more complicated. 
For the description of this $H$-orbit, one has to wait for the determination of the whole
generalized Springer correspondence.  

\par\bigskip
\section{Determination of the generalized Springer correspondence}

\para{9.1.}
In order to obtain an exact parametrization of $H$-orbits in $G^{\io\th}\uni$ in the 
case where $N$ is even, 
we need some preliminaries.  Assume that $N = 2n$.  
Let $P^{+}$ be a $\th$-stable parabolic subgroup
of $G$ containing $B$ and $L^{+}$ the $\th$-stable Levi subgroup of $G$ containing 
$T$ such that $L^{+}_H \simeq GL_n$. 
Then we can write as 
\begin{equation*}
L^{+}_H = \biggl\{ \begin{pmatrix}
               a  &   0   \\
               0  &   {}^ta\iv
         \end{pmatrix}  \mid a \in GL_n \biggr\}, 
\qquad
L^{+ \io\th} = \bigg\{ \begin{pmatrix}
          a   &   0  \\
          0   &   {}^ta
       \end{pmatrix}  \mid a \in GL_n \biggr\}.  
\end{equation*}
Thus $L^{+ \io\th}$ is in natural bijection with $GL_n$, and 
the conjugation action of $L^+_H$ on $L^{+ \io\th}$
coincides with the conjugation action of $GL_n$ on $GL_n$.  
Let $\SO^{\dag}_{\la}$ be the $L^+_H$-orbit in $L^{+ \io\th}$
corresponding to the unipotent class of $GL_n$ corresponding to $\la \in \SP_n$. 
Let $\SO$ be the unique $H$-orbit in $G^{\io\th}$ such that
$\eta_{P^{+}}\iv(\SO^{\dag}_{\la}) \cap \SO $ is open dense in 
$\eta_{P^{+}}\iv(\SO^{\dag}_{\la})$, where 
$\eta_{P^{+}} : P^{+ \io\th} \to L^{+ \io\th}$ 
is defined similarly as before.
Similarly to the map $\pi : \wt X \to X$ in 3.4, we define 
\begin{align*}
\wt X^+_{\la} &= \{ (x, gP_H^{+}) \in G^{\io\th}\uni \times H/P^{+}_H, 
          \mid g\iv xg \in \eta_{P^{+}}\iv(\ol\SO_{\la}^{\dag}) \}, \\ 
X^+_{\la} &= 
    \bigcup_{g \in H}g\eta_{P^{+}}\iv(\ol\SO^{\dag}_{\la})g\iv,
\end{align*}
and let $\pi_{\la} : \wt X^+_{\la} \to X^+_{\la}$ be the first projection. 
Then $\pi_{\la}$ is proper surjective, 
and $X^+_{\la}$ is an irreducible closed subset of $G^{\io\th}$.  
The following result is an analogue of Proposition 8.20.

\begin{lem}  
For $\la = (\la_1, \dots, \la_k) \in \SP_n$, put 
$2\la = (2\la_1, \dots, 2\la_k) \in \SP_N$. 
\begin{enumerate}
\item 
$\SO$ is the unique open dense orbit in $X^+_{\la}$.
\item 
The Jordan type of $\SO$ is equal to $2\la$. 
\item 
$\dim \wt X^+_{\la} = \dim X^+_{\la}$. 
\end{enumerate}
\end{lem}  
                                   
\begin{proof}
The proof of (i) is similar to the proof of Proposition 8.20 (i).  
We show (ii) and (iii). First compute the dimension of $\wt X^+_{\la}$.  Since 
$\wt X^+_{\la} \simeq H \times^{P^{+}_H}\eta_{P^+}\iv(\ol\SO_{\la}^{\dag})$, 
\begin{align*}
\dim \wt X^+_{\la} &= \dim U^{\th}_{P^{+}} + \dim \SO_{\la}^{\dag} 
                          + \dim U_{P^{+}}^{\io\th}  \\
                     &= \dim U_{P^{+}} + \dim \SO_{\la}^{\dag} \\
                     &= 2n^2 - n - 2n(\la)
\end{align*}
since $\dim U_{P^{+}} = n^2, \dim \SO_{\la}^{\dag} = n^2 - n - 2n(\la)$ 
(known result for $GL_n$, see (7.3.1)). 
Since $\dim H = 2n^2 - n$, and $\pi_{\la}$ is 
surjective, we have
\begin{equation*}
\dim X^+_{\la} \le \dim H - 2n(\la). 
\end{equation*} 
If $\SO'$ is an $H$-orbit of type $2\la$, then 
$\dim \SO' = \dim H - 2n(\la)$ by Lemma 7.4.
Hence in order to prove (ii) and (iii), it is enough to show that $X^+_{\la}$ contains 
an element of Jordan type $2\la$. 
Choose a basis $e_1, \dots, e_n, f_1, \dots, f_n$ of $V$ as in 1.7. 
Consider the following map $x : V \to V$;
\begin{align*}
\tag{9.2.1}
&f_1 \mapsto f_2 \mapsto \cdots \mapsto f_{\la_1} 
         \mapsto e_{\la_1} \mapsto e_{\la_1 -1} \mapsto \cdots \mapsto e_1 \mapsto 0, \\
&f_{\la_1 + 1} \mapsto f_{\la_1 + 2} \mapsto \cdots \mapsto  f_{\la_1 +\la_2} 
         \mapsto e_{\la_1 + \la_2} \mapsto e_{\la_1 + \la_2 -1} 
              \mapsto \cdots \mapsto e_{\la_1 + 1} \mapsto 0,  \\
&\cdots\cdots  \\
&f_{\la_1 + \cdots + \la_{k-1}+1} \mapsto \cdots \mapsto f_n \mapsto e_n 
              \mapsto \cdots \mapsto e_{\la_1 + \cdots +\la_{k-1} + 1} \mapsto 0.  
\end{align*}
Then $u = x+1 \in \eta_{P^{+}}\iv(\SO_{\la}^{\dag})$ and     
the Jordan type of $u$ is $2\la$. Thus (ii) and (iii) hold.  The lemma is proved. 
\end{proof}

\para{9.3.}
For an even partition $2\la$ of $N$, we denote by $\SO^+_{2\la}$ the unique 
open dense orbit in $X^+_{\la}$ given in Lemma 9.2. We denote by $\SO^-_{2\la}$ 
another $H$-orbit contained in $\SO_{2\la}$.  Note that (9.2.1) gives an explicit 
representative of the $H$-orbit $\SO^+_{2\la}$.  In particular, all the $\SO^+_{2\la}$ 
are contained in $X\uni = \bigcup_{g \in H}gU^{\io\th}g\iv$.  Hence $\SO^+_{(2n)}$ coincides 
with the unique open dense orbit $\SO_0$ in $X\uni$ given in Proposition 8.20.   
\par
Let $t_n \in N_{G^{\th}}(T) - N_H(T)$ be as in 1.9. 
Then  $t_n\SO_{2\la}^+t_n\iv = \SO_{2\la}^-$.  
Put $P^- = t_nP^+t_n\iv$ and $L^- = t_nL^+t_n\iv$. 
Replacing $P^+, L^+$ by $P^-, L^-$, we can define a variety $X_{\la}^-$ similarly 
to $X_{\la}^+$.  Then $\SO^-_{2\la}$ is characterized as the unique open dense orbit
contained in $X_{\la}^-$. The representatives of $\SO^-_{2\la}$ are obtained by applying 
$\ad(t_n)$ on $x$ in (9.2.1). 
\par
We have a refinement of Proposition 8.7.

\begin{lem}  
Assume that $\la, \la'$ are even partitions. 
Let $Y_{u,v}$ be as in  Proposition 8.7, where $u \in \SO^{\ve}_{\la}$, 
$v \in \SO^{\ve'}_{\la'}$ with $\ve, \ve' \in \{ 1,-1\}$.  If 
$I(Y_{u,v}) \ne \emptyset$, then we have $\ve = \ve'$.  
\end{lem}

\begin{proof}
Assume that $u \in \SO^+_{\la}$.  We choose a representative $u$ such that 
$x = u-1$ is given as in (9.2.1).  Then $\ol x \in \SO^{\pm}_{\la'}$ is 
obtained from $x$, for example, by replacing 
\begin{equation*}
f_1\mapsto f_2 \mapsto \cdots f_{\la_1} \mapsto e_{\la_1} \mapsto \cdots 
  \mapsto e_2 \mapsto e_1 \mapsto 0
\end{equation*} 
by 
\begin{equation*}
f_2 \mapsto \cdots \mapsto f_{\la_1} \mapsto e_{\la_1} \mapsto \cdots \mapsto e_2 \mapsto 0.
\end{equation*}
Thus $v = \ol x + 1$ is contained in a similar variety $X^+_{\la'}$ as 
$X^+_{\la}$ defined for $\la'$.  Hence $v \in \SO^+_{\la'}$.  If $u \in \SO^-_{\la}$, 
we can apply the same argument by replacing $X^+_{\la}$ by $X^-_{\la}$, and 
obtain that $v \in \SO^-_{\la'}$. 
The lemma is proved.    
\end{proof}

\para{9.5.}
Assume that $N \ge 1$ is an odd integer.  A pair $(\la, \tau)$ is called 
a signed partition if $\la = (\la_1, \dots, \la_N)$ is a partition of $N$, 
and if $\tau = (\tau_1, \dots, \tau_N)$ satisfies the condition (cf. 8.17) that 
\begin{enumerate}
\item
$\tau_i = \pm 1$, and $\tau_i = 1$ if $\la_i = 0$, 
\item 
$\tau_i = \tau_j$ if $\la_i = \la_j$, 
\item
$\tau_{i_0} = 1$ where $i_0$ is the index such that $\la_{i_0}$ is the largest 
odd number among $\la_1, \dots, \la_N$.
\end{enumerate}  
\par
We denote by $\Psi_N$ the set of signed partitions of $N$. 
Note that, by Proposition 7.8 (i), the set of $H$-orbits 
in $G^{\io\th}\uni$ is parametrized by $\SP_N$. 
We denote by $\SO_{\la}$ the $H$-orbit corresponding to $\la \in \SP_N$.  
Then as in 8.17, $H$-equivariant simple local system on $\SO_{\la}$ can 
be expressed as $\SE_{\tau}$ for $(\la, \tau) \in \Psi_N$, and the map 
$(\la, \tau) \mapsto (\SO_{\la}, \SE_{\tau})$ gives a bijection 
$\Psi_N \isom \SN_G$ if $N \ge 3$. 
\par
Next assume that $N \ge 0$ is an even integer. 
In this case, for each partition $\la \in \SP_N$, we prepare 
two copies $\la^{\pm}$, and assume that $\la^+ = \la^-$ if $\la$ is 
not an even partition. A signed partition $(\la^{\pm}, \tau)$ is defined 
similarly as above, for each $\la^+$ and $\la^-$.   We denote by 
$\Psi_N$ the set of signed partitions of $N$. Note that if $N = 0$, we 
regard the empty partition $\emptyset$ as an even partition, so we consider 
$(\emptyset^+, 1)$ and $(\emptyset^-,1)$.  
By Proposition 7.8 (ii), the set of $G^{\th}$-orbits in $G^{\io\th}\uni$ 
is parametrized by $\SP_N$. $G^{\th}$-orbit $\SO_{\la}$ is a single $H$-orbit unless 
$\la$ is an even partition, in which case, $\SO_{\la}$ splits into two $H$-orbits.
By 9.3, we denote those two $H$-orbits by $\SO^+_{\la}$ and $\SO^-_{\la}$.  
By abuse of the notation, we denote $\SO_{\la}$ by $\SO^+_{\la} = \SO^{-}_{\la}$ 
if $\la$ is not an even partition.  
Thus we have a bijective map  $\Psi_N \isom  \SN_G$ by 
$(\la^{\pm}, \tau) \mapsto (\SO^{\pm}_{\la}, \SE_{\tau})$, 
if $N \ge 2$.  

\para{9.6.}
Assume that $N$ is odd.  
Take integers $N_0 \ge 1, a \ge 0$ such that $N = 2a + N_0$.  For a fixed 
$\xi = (\nu, \s) \in \Psi_{N_0}$, we define a map 
$\vG_{\xi} : \SP_a \to \Psi_N$ as follows; write 
$\nu = (\nu_1, \dots, \nu_N), \s = (\s_1, \dots, \s_N)$ by putting 
$\nu_i = 0, \s_i = 1$ for $i > N_0$.  For each 
$\mu = (\mu_1, \dots, \mu_N) \in \SP_a$, define integers $\la_1, \dots, \la_N$ by 

\begin{equation*}
\tag{9.6.1}
\la_i = \nu_i + 2\mu_i \quad \text{ for $i = 1, \dots, N$.}
\end{equation*}   
Then $\la = (\la_1, \dots, \la_N) \in \SP_N$, and the pair $(\la, \s)$ gives 
a well-defined element in $\Psi_N$.  By definition, $\vG_{\xi}(\mu) = (\la, \s)$.
In the case where $N$ is even, the map $\vG_{\xi} : \SP_a \to \Psi_N$ is defined
similarly, but we consider $\xi^{\pm} = (\nu^{\pm}, \s)$, and put 
$\vG_{\xi^{\pm}}(\mu) = (\la^{\pm}, \s)$. 
\par
Let $N \ge 0$ be an integer. 
An element $(\la, \tau) \in \Psi_N$ (or $(\la^{\pm},\tau) \in \Psi_N$) 
is called cuspidal if 
\par\medskip
\begin{enumerate}
\item \
$\la_i - \la_{i+1} \le 2$ for $i = 1, \dots, N$ (here we put $\la_{N+1} = 0$),  
\item \
If $\la_i - \la_{i+1} = 2$, then $\tau_i \ne \tau_{i+1}$. 
\end{enumerate}
\par\medskip
We denote by $\SC_N$ the set of triples $(N_0, \nu, \s)$ such that 
$N - N_0 \in 2\BZ_{\ge 0}$ and $\xi = (\nu,\s) \in \Psi_{N_0}$ is a cuspidal element. 
For $c \in \SC_N$, we denote by $\Psi_N^{(c)}$ the image of 
$\vG_{\xi} : \SP_a \to \Psi_N$, where 
$a = (N - N_0)/2$. 
Clearly $\vG_{\xi}$ gives a bijection $\SP_a \isom \Psi_N^{(c)}$. 
We also denote by $\Psi_N^{(0)}$ the set of cuspidal elements in $\Psi_N$. 

\begin{prop}  
There exists a partition  
\begin{equation*}
\tag{9.7.1}
\Psi_N = \coprod_{c \in \SC_N}\Psi_N^{(c)}.
\end{equation*}
\end{prop} 

\begin{proof}
For simplicity we assume that $N$ is odd.  The case where
$N$ is even is dealt similarly.  
Assume that $(\la,\tau) \in \Psi_N$ is not cuspidal.  
Then there exists 
some $i$ such that $\la_i >  \la_{i+1} + 2$ or that $\la_i = \la_{i+1} + 2$ 
with $\tau_i = \tau_{i+1}$. Put $\la' = (\la'_1, \dots, \la'_N) \in \SP_{N-2}$, 
where $\la_i' = \la_i -2, \la'_j = \la_j$ for $j \ne i$.  If we put $\tau' = \tau$,    
we have $(\la', \tau') \in \Psi_{N-2}$ by our assumption.
By induction on $N$, we may assume (9.7.1) holds for $\Psi_{N-2}$.  Hence 
there exists $c = (N_0, \nu, \s) \in \SC_{N-2}$ such that 
$(\la',\tau') \in \Psi^{(c)}_{N-2}$. In particular, 
$(\la',\tau') = \vG_{\xi}(\mu')$ for some $\mu' \in \SP_a$, where 
$\xi = (\nu, \s)$ and $a = (N -2 - N_0)/2$.  Here $\mu' = (\mu_1', \dots, \mu'_N)$
satisfies the condition that $\mu'_{i-1} \ge \mu'_i +1$, and if we define 
$\mu = (\mu_1, \dots, \mu_N) \in \SP_{a+1}$ by $\mu_i = \mu'_i +1, \mu_j = \mu'_j$ 
for $j \ne i$, $(\la,\tau) = \vG_{\xi}(\mu)$.  Hence 
$(\la,\tau) \in \Psi_N^{(c)}$. $c$ is determined by $(\la,\tau)$ uniquely. In fact, 
suppose that $(\la,\tau) \in \Psi^{(c')}_N$ for another $c'$. The above 
argument shows, since $(\nu,\s)$ is cuspidal, that $(\la', \tau') \in \Psi_{N-2}^{(c')}$.     
Hence by induction on $N$, we have $c = c'$.  
\end{proof}

\para{9.8.} 
We have a natural parametrization of $S_a\wg$ by $\SP_a$.  We denote by 
$\r_{\mu}$ the irreducible representation of $S_a$ corresponding to $\mu \in \SP_a$.
(Here the partition $(a)$ corresponds to the unit representation, $(1^a)$ corresponds 
to the sign representation.) 
\par
Under the identification $\Psi_N \isom  \SN_G$ in 9,5, we denote by $\SN_G^{(0)}$ the 
image of $\Psi_N^{(0)}$.
We also denote by $\SN_G^{(c)}$ the subset of $\SN_G$ which is the image of 
$\Psi_N^{(c)}$.  For each $c = (N_0, \nu, \s) \in \SC_N$ with $\xi = (\nu,\s)$, 
the map $\vG_{\xi}$ induces a bijection $\wt\vG_c : S_a\wg \isom \SN_G^{(c)}$, combined 
with $\SP_a \simeq S_a\wg$. Thus, by Proposition 9.7, 
we have a bijective map  

\begin{equation*}
\tag{9.8.1}
\wt\vG : \coprod_{c \in \SC_N}S_a\wg \isom \coprod_{c \in \SC_N}\SN_G^{(c)} = \SN_G
\end{equation*}
where $a = (N - N_0)/2$ for $c = (N_0, \nu, \s) \in \SC_N$.
\par 
Recall the definition of $\SN_G^{(\xi)}$ in 5.7 for 
$\xi = (L \subset P, \SO_L, \SE_1^{\dag}) \in \ScS_G$. 
The following result gives a combinatorial description of the generalized 
Springer correspondence for the symmetric space associated to orthogonal groups. 

\begin{thm}  
Let the notations be as above. 
\begin{enumerate}
\item
$\SN_G^{(0)}$ coincides with the set of cuspidal pairs in $\SN_G$. 
\item 
$c = (N_0, \nu,\s) \mapsto \xi = (L \subset P, \SO_L, \SE^{\dag}_1)$ gives a bijection 
$\SC_N \isom \ScS_G$ such that $\SN_G^{(c)} = \SN_G^{(\xi)}$, 
where $L$ is such that $L_H \simeq (GL_1)^a \times SO_{N_0}$,
and $(\SO_L, \SE^{\dag}_1)  = (\SO_{\nu}, \SE^{\dag}_{\s})$. 
\item
The map $\wt\vG$ in \rm{(9.8.1)} gives a bijection in Theorem 5.2 (ii). 
\end{enumerate}
\end{thm}

\begin{proof}
We prove the theorem by induction on $N$.
In the case where $N = 2, 3$, the claim (iii) is verified directly, see Appendix.     
First consider the case where $N$ is odd. 
Assume that the theorem 
holds for $L \ne G$. We show (iii). 
Choose $(L \subset P, \SO_L, \SE^{\dag}_1) \in \ScS_G$ with $L \ne G$.  By (i) for $L$, 
we can attach $(N_0, \nu,\s) \in \SC_N$ to $(L \subset P,\SO_L,\SE^{\dag}_1)$. 
Take $\r = \r_{\mu} \in S_a\wg$, and let $(\SO, \SE) \in \SN_G$ be the pair which belongs to
$(L \subset P,\SO_L,\SE^{\dag}_1)$ corresponding to $\r$. Put $(\la, \tau) = \vG_{\xi}(\mu)$ 
for $\xi = (\nu,\s)$.  In order to prove (iii), it is enough to see that 
$\SO = \SO_{\la}, \SE = \SE_{\tau}$.  Let $Q_H$ be the maximal parabolic subgroup 
of $H$ containing $P$ and $M_H$ its Levi subgroup containing $L$ as in Section 8. 
Let $\r' = \r_{\mu'}$ be an irreducible character of 
$S_{a-1}$, and $(\SO',\SE') \in \SN_{M}$ the pair which belongs to 
$(L \subset M \cap P,\SO_L, \SE^{\dag}_1)$ 
corresponding to $\r'$. By induction hypothesis, one can write 
as $(\SO', \SE') = (\SO'_{\la'}, \SE'_{\tau'})$ with $(\la', \tau') = \vG_{\xi}(\mu')$.
By Corollary 6.11 and Proposition 8.7, if $\lp \r, \r'\rp \ne 0$, then the Jordan type of $\SO$ is 
obtained from $\la'$ by the inverse operation of $(\text{A}'_i)$ for some $i$. This 
is true for any $\r'$ such that $\lp \r, \r' \rp \ne 0$.  In the case where $a \ge 3$,
this condition determines the Jordan type of $\SO$ uniquely, and we have $\SO = \SO_{\la}$. 
Moreover, in this case, by Proposition 8.16, $\SE$ is determined uniquely from 
$\SE'_{\tau'}$, namely we have $\SE = \SE_{\tau}$. Thus we have proved (iii) for 
$a \ge 3$.  
\par
Now assume that $a = 1$.  In this case, $S_1 = \{ 1\}$, and 
$\pi_*K_{\ol\SE_1}|_{X\uni} = \IC(\ol\SO_{\la'}, \SE_{\tau'})$ up to shift, 
where $\la' = (\nu_1 + 2, \nu_2, \dots, \nu_N), \tau' = (\s_1, \dots, \s_N)$
by Theorem 8.21. Hence the claim holds. Next assume that $a = 2$.
In this case, $S_2\wg$ has two representations, 
the unit representation $\r_{(2)}$ and the sign representation $\r_{(1^2)}$. 
By Lemma 8.19, the possibility for $(\SO, \SE)$ belonging to 
$(L \subset P,\SO_L, \SE^{\dag}_1)$ 
is only $\SO = \SO_{\la}$ with $\la = (\nu_1 + 4, \nu_2, \dots, \nu_N)$ 
or $\la = (\nu_1 + 2, \nu_1 + 2, \nu_3, \dots, \nu_N)$.    
By Theorem 8.21, $\r_{(2)}$ corresponds to $(\SO_0,\SE_0) = (\SO_{\la}, \SE_{\tau})$, 
where $\la = (\nu_1 + 4, \nu_2, \dots, \n_N)$ and $\tau = (\s_1, \dots, \s_N)$. 
Hence $\r_{(1^2)}$ corresponds to $(\SO_{\la}, \SE_{\tau})$ with 
$\la = (\nu_1 +2, \n_2 + 2, \nu_3, \dots, \nu_N)$.  The restriction of $\r_{(1^2)}$ 
on $S_1$ is $\r_{(1)}$ which corresponds to $(\SO_{\la'}, \SE_{\tau'})$. By 
a similar argument as above, $\tau$ is determined from $\tau'$, namely 
we see that  $(\la, \tau) = \vG_{\xi}(1^{(2)})$.  Thus the claim holds for $a = 2$.
Hence (iii) is proved for odd $N$. 
\par
Next consider the case where $N$ is even.  
We choose $(L \subset P, \SO_L, \SE_1^{\dag})\in \ScS_G$. 
Assume that the Jordan type of $\SO_L$ is not an even partition. Then a similar argument 
as before works.  (Note that in that case the Jordan type of $\SO$ is not even 
for any $(\SO, \SE)$ belonging to $(L \subset P, \SO_L, \SE_1^{\dag})$ by Lemma 8.19.)
Thus we assume that the Jordan type of $\SO_L$ is even.  Then by the same reason, 
the Jordan type of $\SO$ is even for such $(\SO, \SE)$. 
The previous argument shows that if $(\SO,\SE)$ corresponds to $\r = \r_{\mu} \in S_a\wg$, 
then the Jordan type of $\SO$ is $\la$, and 
$\SE = \SE_{\tau}$ for $\vG_{\xi}(\mu) = (\la, \tau)$.
But we have to show that $\SO = \SO_{\la}^{\pm}$ if 
$\vG_{\xi^{\pm}}(\mu) = (\la^{\pm}, \tau)$.   
This is done by using Lemma 9.4. Thus (iii) is proved for even $N$ similarly as above.    
\par
We show (i) and (ii).  Assume that $N$ is even.  If $N_0 = 0$, then $a = n$, 
and there exists two triples $c^+, c^- \in \SC_N$, where 
$c^{\pm} = (0, \emptyset^{\pm}, 1)$. We have $\Psi^{(c^{\pm})} \simeq \SP_n$.
On the other hand, 
$N_0 = 0$ corresponds to the case where $L = T$, $\SO_L = \{ 1\}$ and 
$\SE_1^{\dag} = \Ql$ : the constant sheaf.  Hence there exist two triples 
$(T \subset B, \{1\}, \Ql), (T \subset B_1, \{1\}, \Ql) \in \ScS_G$, 
where $B_1 = t_n B t_n\iv$. By (iii), the Springer correspondence gives 
a bijection $\Psi^{(c^{\pm})} \simeq \SN_G^{(c^{\pm})}$.   
For other cases (for any $N$), $c = (N_0, \nu,\s) \in \SC_N$ determines a unique 
triple $(L \subset P, \SO_L, \SE_1^{\dag}) \in \ScS_G$.  By (iii), 
and by Proposition 9.7, this proves (i).  (ii) follows from  (i).
The theorem is proved.
\end{proof}

\para{9.10.}
Following the discussion in [L1, 10.6], we shall give 
a formula which describes the cardinality of $\Psi^{(0)}_N$, namely, 
the number of the cuspidal pairs $(\SO,\SE)$ in $\SN_G$ .  
Note that the computation below is much simpler compared 
to the case of generalized Springer correspondence for orthogonal groups. 
\par
Let $t$ be an indeterminate.  Recall that the partition function 
$p(n) = |\SP_n|$ is defined by the formula 
$\prod_{i = 1}^{\infty}(1 - t^i)\iv = \sum_{n \ge 0}p(n)t^n$.  
We define functions $q_1(n), q_2(n)$ for $n \in \BZ_{\ge 0}$ as follows;

\begin{align*}
\prod_{i = 1}^{\infty}(1 + t^i)^2 &= \sum_{n \ge 0}q_1(n)t^n, \\
\prod_{i = 1}^{\infty}(1 + t^{2i}) &= \sum_{n \ge 0}q_2(n)t^n. 
\end{align*} 

\begin{prop}  
The cardinality $|\Psi^{(0)}_N| = |\SN_G^{(0)}|$ of cuspidal pairs 
is given by the formula, 

\begin{equation*}
|\Psi^{(0)}_N| = \begin{cases}
                   \displaystyle\frac{1}{2} q_1(N) &\quad\text{ if $N \ge 3$ : odd, } \\ \\
                   \displaystyle\frac{1}{2} q_1(N) + \frac{3}{2}q_2(N)
                       &\quad\text{ if $N \ge 2$ : even. }
                 \end{cases}
\end{equation*}
\end{prop}

\begin{proof}
Let $x_N$ be the number of elements in $\Psi_N$, and $x'_N$ (resp. $x_N''$)  
the number of elements $(\la, \tau) \in \Psi_N$ such that $\la$ is not an even partition 
(resp. even partition). Hence $x_N = x_N' + x_N''$. 
Then as in [L1, 10.6], 

\begin{align*}
1 + \sum_{N \ge 1}\biggl( 2x'_N + x''_N/2\biggr)t^N &= 
  \sum_{i_1, i_2, \dots \ge 0}2^{\sharp\{a : i_a> 0\}}t^{1i_1 + 2i_2 + 3i_3 + \cdots } \\
  &= \bigl(1 + \sum_{i_1 \ge 1}2t^{i_1}\bigr)\bigl(1 + \sum_{i_2 \ge 1}2t^{i_2})\cdots  \\
  &= \frac{1 + t}{1 - t}\frac{1 + t^2}{1 - t^2} \cdots \\
  &= \prod_{i = 1}^{\infty}(1 - t^{2i})\iv \prod_{i = 1}^{\infty}(1 + t^i)^2.  
\end{align*} 
It follows that 
\begin{equation*}
\tag{9.11.1}
2x_N' + x_N''/2 = \sum_{0 \le a \le N/2}p(a)q_1(N - 2a). 
\end{equation*}

By a similar computation shows that 
\begin{align*}
1 + \sum_{N \ge 2; \text{even}}(x_N''/2)t^N &= 
        \sum_{i_1, i_2, \dots\ge 0}2^{\sharp\{ a: i_a > 0\}}t^{2(1i_1 + 2i_2 + \cdots )}  \\
    &= \bigl(1 + \sum_{i_1 \ge 1}2t^{2i_1}\bigr)\bigl(1 + \sum_{i_2 \ge 1}2t^{4i_2}\bigr)\cdots \\
    &= \frac{1 + t^2}{1 - t^2}\frac{1 + t^4}{1 - t^4}\cdots \\
    &= \prod_{i = 1}^{\infty}(1 - t^{2i})\iv \prod_{i = 1}^{\infty}(1 + t^{2i}).
\end{align*}
Hence we have

\begin{equation*}
\tag{9.11.2}
x_N''/2 = \sum_{0 \le a \le N/2}p(a)q_2(N -2a). 
\end{equation*}
Now assume that $N$ is even.  Then by (9.11.1) and (9.11.2), we have

\begin{equation*}
x_N = x_N' + x_N'' = \sum_{0 \le a \le N/2}p(a)
                     \biggl\{ \frac{1}{2}q_1(N - 2a) + \frac{3}{2}q_2(N-2a)\biggr\}.
\end{equation*}

On the other hand, by the generalized Springer correspondence (Theorem 9.9), we have

\begin{equation*}
|\Psi_N| = \sum_{0 \le a \le N/2}|S_a\wg||\Psi^{(0)}_{N-2a}|. 
\end{equation*}
Then by induction on $N$, we obtain the formula for $|\Psi^{(0)}_N|$. 
The case where $N$ is odd si similar (in this case, $x_N'' = 0$). 
\end{proof}
\par\bigskip
\section{ Induction}

\para{10.1.}
We consider the group $G$ and $\th : G \to G$ as in (2.1.4).
We fix a $\th$-stable pair $(B, T)$ of a Borel subgroup $B$ and 
a maximal torus of $G$.  Let $P$ be a $\th$-stable parabolic subgroup of $G$
containing $B$ and $L$ the $\th$-stable Levi subgroup of $P$ containing $T$. 
As before, let $\eta_P : P^{\io\th} \to L^{\io\th}$ be the natural projection.
Consider the following diagram 

\begin{equation*}
\tag{10.1.1}
\begin{CD}
L^{\io\th} @<\psi << \wh X^P @>\vf'>> \wt X^P @>\vf''>>  G^{\io\th}, 
\end{CD}
\end{equation*}
where 

\begin{align*}
\wh X^P &= \{ (x, g) \in G^{\io\th} \times H \mid g\iv xg \in P^{\io\th} \}, \\ 
\wt X^P &= \{ (x, gP_H) \in G^{\io\th} \times H/P_H \mid g\iv xg \in P^{\io\th} \},
\end{align*}
and  $\vf': (x, g) \mapsto (x, gP_H), \vf'': (x, gP_H) \mapsto x$, and 
$\psi : (x, g) \mapsto \eta_P(g\iv xg)$. 
Then $\vf', \psi$ are smooth with connected fibres. 
Moreover, $H \times P_H$ acts on $\wh X^P$ by 
$(h, p) : (x,g) \mapsto (hxh\iv, hgp\iv)$, and $H$ acts on $\wt X^P$ by 
$h : (x, gP_H) \mapsto (hxh\iv, hgP_H)$.  $\vf', \vf''$ are $H$-equivariant,
and $\psi$ is $H \times P_H$-equivariant with respect to the trivial action of $H$ 
and the action of $P_H$ on $L^{\io\th}$ induced from the map $P_H \to L_H$.  
\par
Let $K$ be an $L_H$-equivariant perverse sheaf on $L^{\io\th}$, which is regarded 
as an $H \times P_H$-equivariant perverse sheaf.   Since $\psi$ is smooth with connected fibre, 
there exists a perverse sheaf $\psi^*K[\a]$ on $\wh X^P$, where $\a$ is the dimension of 
the fibre. $\psi^*K[\a]$ is $H \times P_H$-equivariant, and since $\vf'$ is a locally trivial 
principal $P_H$-bundle, there exits a perverse sheaf $K_1$ on $\wt X^P$ such that 
$\psi^*K[\a] \simeq {\vf'}^*K_1[\b]$, where $\b = \dim P_H$.  
We define $\ind K = \vf''_!K_1$.  Since $\vf''$ is proper, $\ind K$ is a semisimple complex 
on $G^{\io\th}$.  Since $K_1$ is $H$-equivariant, $\ind K$ is $H$-equivariant.  
We also write $\ind K = \ind_P^G K$. 
$\ind K$ is an analogue of the induction functor of the character sheaves ([L2]). 
\par 
Let $Q$ be a $\th$-stable parabolic subgroup of $G$ containing $P$, and $M$ the $\th$-stable 
Levi subgroup of $Q$ containing $L$.  Then $M \cap P$ is a $\th$-stable 
parabolic subgroup of $M$ with 
Levi subgroup $L$. Thus one can define functors $\ind_{M \cap P}^M$ and $\ind_Q^G$.    
The following transitivity property can be proved in a similar way as in 
[L2, Proposition 4.2]. 

\begin{prop}  
Let $K$ be an $L_H$-equivariant perverse sheaf on $L^{\io\th}$.  
Assume that $\ind_{M \cap P}^MK$ is an $M$-equivariant perverse sheaf on $M^{\io\th}$.
Then we have

\begin{equation*}
\ind_Q^G(\ind_{M \cap P}^M K) \simeq \ind_P^G K. 
\end{equation*}
\end{prop} 

\para{10.3.} 
Returning to the original setting, we consider $G = GL_N$, and let $P$ be 
the $\th$-stable parabolic subgroup such that $L_H \simeq (GL_1)^a \times SO_{N_0}$, 
where $N = N_0 + 2a$.  
Let $Q$ be the $\th$-stable parabolic subgroup of $G$ containing $P$ such that 
$M_H \simeq GL_a \times SO_{N_0}$.
Thus $(M \cap P)_H \simeq B_a \times SO_{N_0}$, where $B_a$ is a Borel subgroup of 
$GL_a$ such that $B_H \cap (GL_a \times GL_a)^{\th} \simeq B_a$. 
Let $\xi = (\SO_L, \SE^{\dag}_1)$ be a cuspidal pair on $L^{\io\th}\uni$, and consider 
the $L_H$-equivariant perverse sheaf $K_{\xi} = \IC(\ol\SO_L, \SE^{\dag}_1)[\dim \SO_L]$ 
on $L^{\io\th}$. 
Since $\ind_{M \cap P}^M K_{\xi}$ is isomorphic to 
$(\ind_{B_a}^{GL_a}K_0) \boxtimes K_{\xi}$, where $K_0$ is the constant sheaf $\Ql$ on 
$\{ 1\} \subset T_a$ ($T_a$ is the maximal torus of $B_a$), by a well-known 
result of Borho-MacPherson ([BM]) for $GL_a$, $\ind_{M \cap P}^MK_{\xi}$ is a semisimple 
perverse sheaf on $M^{\io\th}\uni$, equipped with $S_a$-action, and is decomposed as 

\begin{equation*}
\tag{10.3.1}
\ind_{M \cap P}^M K_{\xi} \simeq 
       \bigoplus_{\mu \in \SP_a}\r_{\mu} \otimes (K^a_{\mu}\boxtimes K_{\xi}),  
\end{equation*}   
where $K^a_{\mu} = \IC(\ol\SO^a_{\mu}, \Ql)[\dim \SO^a_{\mu}]$ (here 
$\SO^a_{\mu}$ is the unipotent class in $GL_a$ with Jordan type $\mu$).
Thus $\ind_{M \cap P}^MK_{\xi}$ is a $M_H$-equivariant perverse sheaf on $M^{\io\th}$, 
and one can apply the functor $\ind_{Q}^G$ on it.  By the transitivity of 
induction (Proposition 10.2), 
we have  
\begin{equation*}
\tag{10.3.2}
\ind_P^G K_{\xi} \simeq \bigoplus_{\mu \in \SP_a}
                   \r_{\mu} \otimes \ind_{Q}^G (K^a_{\mu}\boxtimes K_{\xi}).  
\end{equation*} 

On the other hand, by comparing the construction of $\ind_P^GK_{\xi}$ with 
the complex $\pi_*K_{\ol\SE_1}$ constructed in 3.7, we see that  
$\ind_P^G K_{\xi} \simeq \pi_*K_{\ol\SE_1}[-r]|_{X\uni}$. Hence by Theorem~5.2 
together with Theorem 9.9, we have

\begin{equation*}
\tag{10.3.3}
\ind_P^GK_{\xi} \simeq \bigoplus_{\mu \in \SP_a}
                  \r_{\mu} \otimes K_{\vG_{\xi}(\mu)}, 
\end{equation*}
where $K_{\vG_{\xi}(\mu)} = \IC(\ol\SO, \SE)[\dim \SO]$ if 
$(\SO,\SE) \in \SN_G$ corresponds to $\vG_{\xi}(\mu) \in \Psi_N$.  
\par
$\ind_Q^G(K_{\mu}^a\boxtimes K_{\xi})$ is an $H$-equivariant semisimple 
complex on $G^{\i\th}$. Since it is a direct summand of the semisimple 
perverse sheaf on $G^{\io\th}$, it is a semisimple perverse sheaf. 
By counting the multiplicities $\sum \dim \r_{\mu}$, 
we see that $\ind_Q^G(K_{\mu}^a\boxtimes K_{\xi})$ is a simple perverse sheaf 
on $G^{\io\th}\uni$.   Hence we have
\par\medskip\noindent
(10.3.4) \ $\ind_Q^G(K_{\mu}^a\boxtimes K_{\xi})$ coincides with one of 
the simple component appearing in the right hand side of (10.3.3). 
\par\medskip

\remark{10.4} 
If the isomorphisms in (10.3.2) and (10.3.3) are compatible with $S_a$-action, 
we will immediately get the isomorphism 
$\ind_Q^G(K_{\mu}^a\boxtimes K_{\xi}) \simeq K_{\vG_{\xi}(\mu)}$.
Note that the diagram in 3.7 with respect to $\pi : \wt X \to X$ corresponds to the 
diagram (10.1.1) with respect to $\ind_P^G$.
However, in the case of $\ind_Q^G$, we cannot construct an analogous diagram of  
3.6.  So, it is not certain  whether 
$\ind_Q^G$ commutes with the action of $S_a$.  
Nevertheless we will show in the discussion below that 
the above isomorphism actually holds. 

\para{10.5.}
Let $\SO'$ be an $M_H$-orbit in $M^{\io\th}\uni$.  Then there exists 
a unique $H$-orbit $\SO$ in $G^{\io\th}\uni$ such that 
$\SO \cap \eta_Q\iv(\SO')$ is open dense in $\eta_Q\iv(\SO')$. 
$\SO$ is called the $H$-orbit induced from $\SO'$, and is denoted by 
$\SO = \Ind_Q^G \SO'$.  Put

\begin{align*}
\wt X_{\SO'} &= \{ (x, gQ_H) \in G^{\io\th} 
         \times H/Q_H \mid g\iv xg \in \eta_Q\iv(\ol\SO') \}, \\
    X_{\SO'} &= \bigcup_{g \in H}g\eta_Q\iv(\ol\SO')g, 
\end{align*}
and let $\pi_{\SO'} : \wt X_{\SO'} \to X_{\SO'}$ be the first projection. 
Then $\pi_{\SO'}$ is proper, surjective.  By a similar argument as in the proof of 
Proposition 8.20, we see that $\SO$ is the unique open dense orbit contained in 
$X_{\SO'}$. 
Let $\SO_L$ be the $L_H$-orbit in $L^{\io\th}\uni$ as before, 
and consider $\SO' = \SO_{\mu}^a \times \SO_L$ 
for each $\mu \in \SP_a$.   We write $\SO' = \SO'_{\mu}$, 
and express $\pi_{\SO'} : \wt X_{\SO'} \to X_{\SO'}$ as 
$\pi_{\mu} : \wt X_{\mu} \to X_{\mu}$. 
The following result shows that the induction of $H$-orbits can be 
realized in the level of perverse sheaves. 

\begin{prop}   
Let $\mu \in \SP_a$.  Under the notation above, 
\begin{enumerate}
\item   
$\ind_Q^G(K_{\mu}\boxtimes K_{\xi}) \simeq K_{\vG_{\xi}(\mu)}$. 
\item  
Let $(\SO,\SE) \in \SN_G$ be the pair corresponding to 
$\vG_{\xi}(\mu) \in \Psi_N$.  Then $\SO = \Ind_Q^G(\SO^a_{\mu} \times \SO_L)$.  
In particular, the Jordan type of $\SO$ is equal to $\nu + 2\mu$. 
\item 
$\dim \wt X_{\mu} = \dim X_{\mu}$. 
\end{enumerate}
\end{prop}

\begin{proof}
First we note the following property, which is obtained from 
the explicit description of the generalized Springer correspondence 
in Theorem 9.9 (iii).
\par\medskip
\noindent
(10.6.1) \ Let $(\SO,\SE), (\SO', \SE') \in \SN_G$.  
Assume that $(\SO,\SE)$ corresponds to $\vG_{\xi}(\mu)$, and 
$(\SO', \SE')$ corresponds to $\vG_{\xi}(\mu')$.   If the Jordan type of 
$\SO$ is the same as that of $\SO'$, then $\mu = \mu'$.  
\par\medskip

We prove (i) by the backward induction on $n(\mu)$.  
Assume that the statement holds for $\mu'$ such that $n(\mu') > n(\mu)$.
Put $\ind_Q^G(K_{\mu}\boxtimes K_{\xi}) = \IC(\SO',\SE')[\dim \SO']$. 
It follows from the definition of $\ind_Q^G$, $\SO'$ is open dense in 
$X_{\mu}$. 
Let $\SO_{\mu, \mu, \nu}$ be the unipotent class 
in $M$ such that its intersection with $M^{\io\th}$ coincides with 
$\SO'_{\mu} = \SO^a_{\mu} \times \SO_L$.  Then  
$\eta_Q\iv(\SO'_{\mu}) \subset \SO_{\mu,\mu,\nu}U_Q$.  It is known 
from the theory of Hall polynomials (see [M]) that the unipotent class 
in $GL_N$ which has open dense intersection with $\SO_{\mu,\mu,\nu}U_Q$ 
has Jordan type $2\mu + \nu$, and all other unipotent classes with non-zero 
intersection with $\SO_{\mu,\mu,\nu}U_Q$ have Jordan type $< 2\mu + \nu$.    
Hence the Jordan type of $\SO'$ is of the form 
$2\mu' + \nu \le 2\mu + \nu$. 
By induction hypothesis, for any $\mu''$ such that $n(\mu'') > n(\mu)$, the pair 
$(\SO_1, \SE_1)$ such that $\SO_1$ has Jordan type $2\mu'' + \nu$ is already 
assigned to $\ind_Q^G(K_{\mu''}\boxtimes K_{\xi})$.  Thus $n(\mu') = n(\mu)$, 
and so the Jordan type of 
$\SO'$ must coincide with $2\mu + \nu$. Hence $(\SO',\SE')$ corresponds to 
$\vG_{\xi}(\mu)$ by (10.6.1).  This proves (i).  
\par
We show (ii) and (iii). 
It follows from the definition that $\SO'$ coincides with $\Ind_Q^G \SO'_{\mu}$. 
Thus (ii) follows from (i). We have

\begin{align*}
\dim X_{\mu} = \dim \SO' = \dim H - n(\nu + 2\mu). 
\end{align*}

On the other hand, since $\wt X_{\mu} \simeq H \times^{Q_H}\eta_Q\iv(\ol\SO'_{\mu})$, 
we have

\begin{align*}
\dim \wt X_{\mu} &= \dim H - \dim Q_H + \dim \SO'_{\m} + \dim U_Q^{\io\th}.
\end{align*}
Since $\dim \SO'_{\mu} = (a^2 - a - 2n(\mu)) + (\dim SO_{N_0} - n(\nu))$, 
$\dim U_Q^{\io\th} = \dim U_Q^{\th} + a$, 
we conclude that $\dim \wt X_{\mu} = \dim X_{\mu}$.  Thus (iii) holds. 
The proposition is proved. 
\end{proof}

\remark{10.7.} 
Proposition 10.6 is a generalization of Proposition 8.20 and of 
Lemma 9.2.  But in contrast to the previous cases, in this discussion, 
we don't need to show the existence of an element of Jordan type 
$\nu + 2\mu$ inside of $X_{\mu}$.   

\par\bigskip\medskip
\begin{center}
\sc{Appendix}
\end{center}

\par\medskip
We give some examples of the generalized Springer correspondence. 
Here we use a simplified notation to denote the signed partition $(\la, \tau)$.
If $\la = (a_1^{m_1}, a_2^{m_2}, \dots)$, we only denote the signature 
$\pm 1$ corresponding to the block $a_i^{m_i}$ as $\pm$.  For example, 
$(1^3; +1,+1, +1)$ is written as $(1^3;+)$ and $(2^21^2;-1,-1,+1,+1)$ is written as 
$(2^21^2;-+)$.  
If $c = (N_0,\nu,\s) \in \SC_N$, the set $\Psi_N^{(c)}$ is the set of 
signed partitions corresponding to $\r_{\mu} \in S_a\wg$ with $a = (N - N_0)/2$, 
which we denote by $S_a;(\nu,\s)$.  
Let $\Psi^{(0)}_N$ be the set of cuspidal elements in $\Psi_N$.
In the first column, we only list up the elements which are not 
contained in $\Psi^{(0)}_N$.  
\par\bigskip\medskip\noindent
$\bullet$ $N = 3.$

\begin{table}[h]
\begin{center}
\begin{tabular}{|c|c|}
\hline
  $(\la, \tau)$  &   $S_1; (1;+)$  \\  
\hline 
 $(3; +)$    &     (1)    \\
\hline
\end{tabular}
\end{center}
\end{table}
\par\noindent
$\Psi_N^{(0)} = \{ (21,++), (21;-+), (1^3;+) \}$.

\par\bigskip\bigskip\noindent
$\bullet \ N = 5.$
\par\medskip
\begin{table}[h]
\begin{center}
\begin{tabular}{|c|c|ccc|}
\hline
  $(\la, \tau)$  &   $S_2; (1;+)$  &   $S_1; (21;++)$  & $S_1; (21;-+)$ 
                 &   $S_1; (1^3;+)$ \\  
\hline 
 $(5; +)$     &     $(2)$  &           &          &          \\
 $(41; ++)$   &            &   $(1)$   &          &          \\    
 $(41;-+)$    &            &           &   $(1)$  &          \\
 $(32;++)$    &  $(1^2)$   &           &          &           \\
 $(31^2;++)$  &            &           &          &   $(1)$   \\
 \hline
\end{tabular}
\end{center}
\end{table}
\par\noindent
$\Psi_N^{(0)} = \{(32;+-), (31^2;+-), 
        (2^21; ++), (2^21; -+), (21^3; ++), (21^3;-+), (1^5;+) \}$. 

\par\bigskip\bigskip\noindent
$\bullet \ N = 7$ 
\par\medskip
\begin{table}[h]
\begin{tabular}{|c|c|ccc|}
\hline 
 $(\la, \tau)$  &   $S_3; (1;+)$  &   $S_2; (21;++)$  &  $S_2;(21;-+)$
                 &   $S_2; (1^3;+)$    \\  
\hline 
 $(7;+)$       &  $(3)$   &           &          &          \\
 $(61;++)$     &          &   $(2)$   &          &          \\    
 $(61;-+)$     &          &           &   $(2)$  &          \\ 
 $(52; ++)$    &  $(21)$  &           &          &          \\
 $(52; +-)$    &          &           &          &          \\ 
 $(51^2; ++)$  &          &           &          &  $(2)$   \\
 $(51^2; +-)$  &          &           &          &          \\ 
 $(43;++)$     &          &   $(1^2)$ &          &          \\
 $(43;-+)$     &          &           &  $(1^2)$ &          \\
 $(421;+++)$   &          &           &          &          \\
 $(421;--+)$   &          &           &          &          \\
 $(41^3;++)$   &          &           &          &          \\
 $(41^3;-+)$   &          &           &          &          \\
 $(3^21;++)$   &          &           &          &  $(1^2)$ \\
 $(32^2;++)$   &  $(1^3)$ &           &          &          \\     
 $(31^4;++)$   &          &           &          &          \\
\hline
\end{tabular}
\end{table}

The correspondence in the case of $S_1$ is as follows;
\par\medskip
$S_1;(32;+-) \lra (52;+-)$, \quad $S_1;(31^2;+-) \lra (51^2;+-)$,
\par
$S_1;(2^21;++) \lra (421;+++)$, \quad $S_1; (2^21;-+) \lra (421; --+)$,
\par
$S_1; (21^3:++) \lra (41^3;++)$, \quad $S_1; (21^3;-+) \lra (41^3;-+)$, 
\par
$S_1; (1^5;+) \lra (31^4; ++)$.
\par\medskip
Here 
\par\noindent
\begin{align*}
\Psi^{(0)}_N = 
\bigl\{ &(421; -++), (421; +-+), (3^21;+-), (32^2;+-), \\ 
   &(321^2;+++), (321^2;++-), (321^2;+-+), (321^2; +--), \\
   &(31^4;+-), (2^31;++), (2^31;-+), (2^21^3;++), (2^21^3; -+), \\
   & (21^5;++), (21^5;-+), (1^7;+) \bigr\}.
\end{align*} 

\par\bigskip\bigskip\noindent
$\bullet N = 2$. 

\par\medskip
\begin{table}[h]
\begin{center}
\begin{tabular}{|c|cc|}
\hline
  $(\la, \tau)$  &   $S_1; (\emptyset^+ ;+)$  &  $S_1; (\emptyset^- ; +)$ \\  
\hline 
 $((2)^+; +)$    &     (1)  &          \\
 $((2)^-; +)$    &          &    (1)   \\
\hline
\end{tabular}
\end{center}
\end{table}

\par\noindent
$\Psi^{(0)}_N = \{ (2)^+;-), ((2)^-;-), (1^2; +)\}$. 
\par\bigskip\medskip\noindent
$\bullet N = 4$. 

\par\medskip
\begin{table}[h]
\begin{center}
\begin{tabular}{|c|cc|ccc|}
\hline
  $(\la, \tau)$  &   $S_2; (\emptyset^+ ;+)$ & $S_2; (\emptyset^-;+)$  &  $S_1;((2)^+;-)$   
                     &  $S_1; ((2)^-;-)$  &  $S_1; (1^2;+)$  \\  
\hline 
 $((4)^+; +)$    &     $(2)$   &          &          &         &      \\
 $((4)^+; -)$    &             &          &   $(1)$  &         &       \\
 $((4)^-; +)$    &             &   $(2)$  &          &         &       \\
 $((4)^-;-)$     &             &          &          &   $(1)$ &        \\
 $(31; ++)$      &             &          &          &         &   $(1)$  \\
 $((2^2)^+; +)$   &    $(1^2)$ &          &           &        &      \\
 $((2^2)^-,+)$    &            &  $(1^2)$ &           &        &      \\
\hline
\end{tabular}
\end{center}
\end{table}

\par\noindent
$\Psi^{(0)}_N = \bigl\{ (31;+-), ((2^2)^+,-), ((2^2)^-;-),  
                 (21^2;++), (21^2;-+), (1^4;+) \bigr\}$.

\par\bigskip\bigskip\noindent
$\bullet N = 6$. 

\par\medskip
\begin{table}[h]
\begin{center}
\begin{tabular}{|c|cc|ccc|}
\hline
  $(\la, \tau)$  &   $S_3; (\emptyset^+ ;+)$ & $S_3; (\emptyset^-;+)$  &  $S_2;((2)^+;-)$   
                     &  $S_2; ((2)^-;-)$  &  $S_2; (1^2;+)$  \\  
\hline 
 $((6)^+; +)$    &     $(3)$   &          &          &         &      \\
 $((6)^+; -)$    &             &          &   $(2)$  &         &       \\
 $((6)^-; +)$    &             &   $(3)$  &          &         &       \\
 $((6)^-;-)$     &             &          &          &   $(2)$ &        \\
 $(51; ++)$      &             &          &          &         &   $(2)$  \\
 $(51; +-)$      &             &          &          &         &       \\ 
 $((42)^+; ++)$   &    $(21)$  &          &          &         &     \\
 $((42)^+; -+)$   &            &          &  $(1^2)$ &         &      \\
 $((42)^+; --)$   &            &          &          &         &      \\
 $((42)^-; ++)$   &            &   $(21)$ &          &         &     \\ 
 $((42)^-; -+)$   &            &          &          &  $(1^2)$ &    \\
 $((42)^-; --)$   &            &          &          &          &    \\
 $(41^2; ++)$     &            &          &          &          &      \\
 $(41^2; -+)$     &            &          &          &          &      \\
 $(3^2; +)$       &            &          &          &          &  $(1^2)$   \\
 $(31^3; ++)$     &            &          &          &          &     \\
 $((2^3)^+;+)$    &    $(1^3)$ &          &          &          &     \\  
 $((2^3)^-;+)$    &            &   $(1^3)$&          &          &      \\
\hline
\end{tabular}
\end{center}
\end{table}

The correspondence in the case of $S_1$ is as follows;

\par\medskip
$(31;+-) \lra (51;+-)$, \qquad $((2^2)^+; -) \lra ((42)^+;--)$, 
\par
$((2^2)^-;-) \lra ((42)^-;--)$,  \qquad $(21^2;++) \lra (41^2;++)$,  
\par
$(21^2; -+) \lra (41^2; -+)$, \qquad $(1^4;+) \lra (31^3);++)$.

\par\medskip
Here 
\begin{align*}
\Psi^{(0)}_N &= \{ ((42)^+;+-), ((42)^-;+-), (321; +++), (321;++-),  \\
             &\phantom{***}
                 (321;+-+), (321; +--), (31^3;+-), ((2^3)^+;-), ((2^3)^-;-), \\
            &\phantom{***}
                (2^21^2;++), (2^21^2;-+), (21^4; ++), (21^4:-+), (1^6;+) \}. \\
\end{align*}
\par

\par\vspace{1cm}

\noindent
T. Shoji \\
School of Mathematical Sciences, Tongji University \\ 
1239 Siping Road, Shanghai 200092, P. R. China  \\
E-mail: \verb|shoji@tongji.edu.cn|

\par\vspace{1cm}
\noindent
G. Yang \\
Department of Mathematics, 
Harbin Engineering University \\
145 Nantong Street, Harbin 150001, P.R. China \\
E-mail: \verb|yanggao_670206@163.com|

\end{document}